\documentclass[11pt,a4paper,reqno]{amsart}

\usepackage[utf8]{inputenc}
\usepackage{amssymb}
\usepackage{amsmath}
\usepackage{amsthm}
\usepackage{latexsym}
\usepackage{amsfonts}
\usepackage{mathrsfs}
\usepackage{bbm}
\usepackage{titlesec}
\usepackage{color}
\usepackage[pdfencoding=auto]{hyperref}
\usepackage{url}
\usepackage{dsfont}
\usepackage{esint}

\newtheorem{thm}{Theorem}[section]
\newtheorem{cor}[thm]{Corollary}
\newtheorem{lem}[thm]{Lemma}
\newtheorem{prop}[thm]{Proposition}

\theoremstyle{definition}

\newtheorem{rem}[thm]{Remark}
\newtheorem*{rem*}{Remark}

\newtheorem{lemma-A}[theorem]{Lemma}
\newtheorem{corollary-A}[theorem]{Corollary}
\newtheorem{example-A}[theorem]{Example}
\newtheorem{definition-A}[theorem]{Definition}

\titleformat{\section}{\normalfont\bfseries\centering}{\thesection.}{.25em}{}
\titleformat{\subsection}{\normalfont\bfseries}{\thesubsection.}{.25em}{}
\titleformat{\subsubsection}{\normalfont\it}{\thesubsubsection.}{.25em}{}
\titlespacing{\section}{0pt}{*5}{*1.5}
\titlespacing{\subsection}{0pt}{*4}{*0.5}
\titlespacing{\subsubsection}{0pt}{*4}{*0.5}
\numberwithin{equation}{section}
\allowdisplaybreaks
\renewcommand{\emptyset}{\varnothing}
\newcommand{\one}{\mathbbm{1}}
\newcommand{\braces}[1]{{\rm (}#1{\rm )}}
\newcommand{\rmref}[1]{{\rm\ref{#1}}}
\usepackage{xcolor}
\newcommand{\Id}{\operatorname{Id}}
\usepackage[normalem]{ulem}
\newcommand{\LowerBound}{C_1}
\newcommand{\UpperBound}{C_2}

\newcommand{\R}{\ensuremath{\mathbb R}}    % Reelle Zahlen
\renewcommand{\C}{\ensuremath{\mathbb C}}    % Komplexe Zahlen
\newcommand{\Q}{\ensuremath{\mathbb Q}}    % Rationale Zahlen
\newcommand{\N}{\ensuremath{\mathbb N}}    % Nat"urliche Zahlen
\newcommand{\Z}{\ensuremath{\mathbb Z}}    % Ganze Zahlen
\newcommand{\<}{\langle}
\renewcommand{\>}{\rangle}

\newcommand{\calF}{\mathcal F}         
\newcommand{\calG}{\mathcal G}         
\newcommand{\calH}{\mathcal H}

\newcommand{\calK}{\mathcal K}         
\newcommand{\calL}{\mathcal L}

\newcommand{\la}{\lambda}
\newcommand{\veps}{\varepsilon}
\newcommand{\vphi}{\varphi}

\newcommand{\smallvek}[2]{\left(\begin{smallmatrix}#1\\#2\end{smallmatrix}\right)}

\renewcommand{\Im}{\operatorname{Im}}
\renewcommand{\Re}{\operatorname{Re}}
\newcommand{\linspan}{\operatorname{span}}
\newcommand{\dom}{\operatorname{dom}}

\newcommand{\ran}{\operatorname{ran}}

\newcommand{\supp}{\operatorname{supp}}

\newcommand{\id}{\mathrm{id}}
\newcommand{\diag}{\operatorname{diag}}
\newcommand{\firstN}[1]{\{1,\dots,#1\}}
\newcommand{\gramian}{{\mathbf G}}
\newcommand{\frameOP}{{\mathbf S}}
\newcommand{\synthOP}{{\mathbf T}}
\newcommand{\boldX}{{\mathbf X}}
\newcommand{\GL}{\operatorname{GL}}
\newcommand{\SL}{\operatorname{SL}}
\newcommand{\eps}{\varepsilon}

\newcommand{\LTwoIndex}{L^2}
\newcommand{\vectorizeOperator}{\mathcal{V}}
\newcommand{\specialFourier}{\mathcal{U}}
\newcommand{\Fourier}{\mathcal{F}}

\newcommand{\indicator}{{\mathds{1}}}

\newcommand{\pseudo}[1]{#1^\dagger}

\newcommand{\Llra}{\Longleftrightarrow}

\newcommand{\ol}{\overline}

\newcommand{\wt}{\widetilde}
\newcommand{\wh}{\widehat}

\newcommand{\bP}{\mathbb P}
\newcommand{\dist}{\operatorname{dist}}
\newcommand{\HH}{\mathbb H}
\newcommand{\sinc}{\operatorname{sinc}}
\newcommand{\essinf}{\operatorname*{essinf}}

\renewcommand{\o}{\omega}
\newcommand{\bS}{{\mathbf S}}
\newcommand{\bG}{{\mathbf G}}
\newcommand{\Lebesgue}{\boldsymbol{\lambda}}
\binoppenalty=\maxdimen
\relpenalty=\maxdimen

\begin{document}
%%%%%%%%%%%%%%%%%%%%%%%%%%%%%%%%%%%%%%%%%%%%%%%%%%%%%%%%%%%%%%%%%%%%%%%%%%%%%
%%%
%%%  HEAD OF PAPER
%%%
\title[A quantitative subspace Balian-Low theorem]{A quantitative subspace Balian-Low theorem}

\author[A. Caragea]{Andrei Caragea}
\address{{\bf A.~Caragea:} KU Eichst\"att-Ingolstadt,
                           Mathematisch-Geographische Fakult\"at,
                           Ostenstra\ss e 26,
                           Kollegiengeb\"aude I Bau B,
                           85072 Eichst\"att,
                           Germany}
\email{andrei.caragea@gmail.com}
\urladdr{http://www.ku.de/?acaragea}

\author[D.G. Lee]{Dae Gwan Lee}
\address{{\bf D.G.~Lee:} KU Eichst\"att-Ingolstadt,
                         Mathematisch-Geographische Fakult\"at,
                         Ostenstra\ss e 26,
                         Kollegiengeb\"aude I Bau B,
                         85072 Eichst\"att,
                         Germany}
\email{daegwans@gmail.com}
\urladdr{http://www.ku.de/?lee}

\author[F. Philipp]{Friedrich Philipp}
\address{{\bf F.~Philipp:} Technische Universit\"at Ilmenau, Institute for Mathematics,
Weimarer Stra\-\ss e~25, D-98693 Ilmenau, Germany}
\email{friedrich.philipp@tu-ilmenau.de}

\author[F. Voigtlaender]{Felix Voigtlaender}
\address{{\bf F.~Voigtlaender:}
KU Eichst\"att-Ingolstadt,
Mathematisch-Geographische Fakult\"at,
Ostenstra\ss e 26,
Kollegiengeb\"aude I Bau B,
85072 Eichst\"att,
Germany}
\email{felix@voigtlaender.xyz}
\urladdr{http://www.ku.de/?voigtlaender}
\urladdr{http://voigtlaender.xyz}

\begin{abstract}
Let $\calG \subset L^2(\R)$ be the subspace spanned by a Gabor Riesz sequence
$(g,\Lambda)$ with $g \in L^2(\R)$ and a lattice $\Lambda \subset \R^2$ of rational density.
It was shown recently that if $g$ is well-localized both in time and frequency,
then $\calG$ cannot contain any time-frequency shift $\pi(z) g$ of $g$
with $z \in \R^2 \setminus \Lambda$.
In this paper, we improve the result to the quantitative statement that the $L^2$-distance
of $\pi(z)g$ to the space $\calG$ is equivalent to the Euclidean distance
of $z$ to the lattice $\Lambda$, in the sense that the ratio between those two distances
is uniformly bounded above and below by positive constants.
On the way, we prove several results of independent interest,
one of them being closely related to the so-called weak Balian-Low theorem for subspaces.
\end{abstract}



\maketitle
\thispagestyle{empty}

\section{Introduction}

The Balian-Low theorem is a well known and fundamental result in time-frequency analysis,
which asserts that a Gabor system cannot be a Riesz basis for $L^2(\R)$
if its generating window is well localized both in time and frequency.
More precisely, it states the following:

\begin{thm}[Balian-Low Theorem]
%%% Let $\Lambda\subset\R^2$ be a lattice of density $1$ and let $g\in L^2(\R)$
%%% be such that the Gabor system $\{e^{2\pi ibx}g(x-a) : (a,b)\in\Lambda\}$
%%% is a Riesz basis for $L^2(\R)$.
Let $g\in L^2(\R)$ and let $\Lambda\subset\R^2$ be a lattice
such that the Gabor system $\{e^{2\pi ibx}g(x-a) : (a,b)\in\Lambda\}$ is a Riesz basis for $L^2(\R)$
(and therefore $\Lambda$ is of density $1$).
Then
\begin{equation}\label{e:up}
  \left(\int x^2 \, |g(x)|^2\,dx\right)
  \left(\int \omega^2 \, |\widehat g(\omega)|^2\,d\omega\right)
  = \infty.
\end{equation}
\end{thm}

%In the recent paper \cite{clpp},
Recently, the following generalization of the Balian-Low theorem was proved in \cite{clpp}
(see also \cite{cmp} for a similar generalization of the amalgam Balian-Low theorem).

\begin{thm}[\cite{clpp}]\label{t:clpp}
  %%% Let $\Lambda\subset\R^2$ be a lattice of rational density and let $g\in L^2(\R)$
  %be such that the Gabor system $\{e^{2\pi ibx}g(x-a) : (a,b)\in\Lambda\}$ is a Riesz basis
  %for its closed linear span $\calG(g,\Lambda)$.
  Let $g\in L^2(\R)$ and let $\Lambda\subset\R^2$ be a lattice of rational density
  such that the Gabor system $\{e^{2\pi ibx}g(x-a) : (a,b)\in\Lambda\}$ is a Riesz basis
  for its closed linear span $\calG(g,\Lambda)$.
  If there exists a time-frequency shift $e^{2\pi i\eta x} g(x-u)$,
  $(u,\eta) \in \R^2 \backslash \Lambda$, of $g$
  which is contained in $\calG(g,\Lambda)$, then \eqref{e:up} holds.
\end{thm}

Note that condition \eqref{e:up} is equivalent to having $g \notin H^1(\R)$
or $\widehat g\notin H^1(\R)$, where $H^1(\R)$ denotes the usual Sobolev space
in $L^2(\R)$ of regularity order $1$.
Therefore, Theorem~\ref{t:clpp} can be rephrased as follows:
if $g,\widehat g\in H^1(\R)$, then the time-frequency shift $e^{2\pi i\eta x}g(x-u)$
has a positive $L^2$-distance to the space $\calG(g,\Lambda)$
whenever $(u,\eta) \in \R^2$ has a positive Euclidean distance to the lattice $\Lambda$.
As our main result, we are going to prove the following quantitative version of Theorem~\ref{t:clpp}
which relates the two mentioned distances.
In the sequel, we denote by $\HH^1(\R)$ the set of all $g\in H^1(\R)$
satisfying $\widehat g\in H^1(\R)$.

\begin{thm}\label{t:quanty}
% Let $\Lambda\subset\R^2$ be a lattice of rational density and let $g\in\HH^1(\R)$
% be such that $\{e^{2\pi ibx}g(x-a) : (a,b)\in\Lambda\}$ is a Riesz basis
% for its closed linear span $\calG(g,\Lambda)$.
Let $g\in\HH^1(\R)$ and let $\Lambda\subset\R^2$ be a lattice of rational density such that
$\{e^{2\pi ibx}g(x-a) : (a,b)\in\Lambda\}$ is a Riesz basis
for its closed linear span $\calG(g,\Lambda)$.
Then there exist constants $\LowerBound,\UpperBound > 0$ such that for all $(u,\eta)\in\R^2$ we have
\begin{equation}\label{e:main}
  \LowerBound \cdot \dist \big( (u,\eta), \Lambda \big)
  \,\le\, \dist \big( e^{2\pi i\eta x} g(x-u) , \calG(g,\Lambda) \big)
  \,\le\, \UpperBound \cdot \dist \big( (u,\eta), \Lambda \big).
\end{equation}
\end{thm}

The \emph{upper} bound in \eqref{e:main} in fact holds for any $g\in\HH^1(\R)$
and any lattice $\Lambda\subset\R^2$, regardless of $\{e^{2\pi ibx}g(x-a) : (a,b)\in\Lambda\}$
being a Riesz sequence or the lattice $\Lambda$ having rational density;
besides, an explicit constant $\UpperBound$ can be found easily; see Proposition~\ref{p:upper} below.
On the other hand, finding an explicit constant $\LowerBound$ is more elusive.
%%% Although we could derive a constant $\alpha$ such that \eqref{e:main} holds
%%%for $(u,\eta)$ close to the lattice $\Lambda$ in the orthogonal case (cf.~Theorem~\ref{t:ortho}),
%%%we were unable to find an explicit global constant $\alpha$.
% but not for all $(u,\eta)$ in $\R^2$.
Even in the case where $(g,\Lambda)$ forms an orthonormal system,
we were only able to derive a constant $\LowerBound$ such that \eqref{e:main}
holds for $(u,\eta)$ close to the lattice $\Lambda$; see Theorem~\ref{t:ortho}.
We expect such a constant to depend on the Riesz bounds of $\{e^{2\pi ibx}g(x-a) : (a,b)\in\Lambda\}$
and on the norms $\|g\|_{L^2}$, $\|g\|_{H^1}$, and $\|\widehat g\|_{H^1}$.
%
%\todo[inline]{The setting described in Theorem~\ref{t:quanty} has applications in OFDM
%              (orthogonal frequency division multiplexing), where complex-valued symbols
%              $g_{a,b} = e^{2\pi ibx}g(x-a)$, $(a,b)\in\Lambda$, are transmitted
%              through a channel in form of a Gabor-synthesized waveform.
%              The receiver's task is to perform a Gabor analysis to get a first estimate of the $g_{a,b}$.
%              Instead of $L^2(\R)$, the correct space obviously is $\calG(g,\Lambda)$.
%              A possible model for the channel is an operator which performs randomly weighted
%              time-frequency shifts which, of course, do not follow the lattice $\Lambda$.
%              To keep the system under control on the receiver's side one typically chooses $g\in\HH^1(\R)$.
%              As a consequence, the channel ``pushes the signals out of the space $\calG(g,\Lambda)$''
%              and \eqref{e:main} quantifies this.}

\medskip{}

\emph{Quantitative Balian-Low estimates for general elements in the Gabor space.}
Writing $\pi(u,\eta)f(x) = e^{2\pi i\eta x}f(x-u)$,
one might wonder whether the estimate
\[
  \dist \big( \pi(u,\eta) f, \mathcal{G}(g,\Lambda) \big)
  \asymp \dist \big( (u,\eta) , \Lambda \big) \cdot \|f\|_{L^2}
\]
holds for general $f \in \mathcal{G}(g,\Lambda)$ and not just for $f = g$.
In general this is \emph{not} the case.
Indeed, if $g$ is a Gaussian, then $(g,2 \Z {\times} \tfrac{2}{3} \Z)$
is a Riesz basis for its closed linear span $\mathcal{G}(g, 2 \Z {\times} \tfrac{2}{3} \Z)$,
but there exists a function $0 \neq f \in \mathcal{G}(g, 2 \Z {\times} \tfrac{2}{3} \Z)$
satisfying $f(\cdot - 1) \in \mathcal{G}(g, 2 \Z {\times} \tfrac{2}{3} \Z)$
(see Example~\ref{exa:DistanceOnlyForGenerator} for details).
Therefore, we see that the distance
$\dist \big( \pi(1,0) f , \calG (g, 2 \Z {\times} \tfrac{2}{3} \Z) \big)$
vanishes, even though $\dist( (1,0) , 2 \Z {\times} \tfrac{2}{3} \Z) \cdot \|f\|_{L^2} \neq 0$.
%Let us point out that generalizing Theorem~\ref{t:quanty} to a statement
%about $\dist(e^{2\pi i\eta x}f(x-u),\calG(f,\Lambda))$ for a generic element $f$
%in $\calG(f,\Lambda)$ is not possible.
%For example, if $g$ is a Gaussian, then $\{e^{2\pi ibx}g(x-a) : (a,b) \in 2\Z{\times}\frac{2}{3}\Z \}$
%is a Riesz basis of $\calG(g,2\Z{\times}\frac{2}{3}\Z)$ but there exists
%\footnote{Since $\{e^{2\pi ibx}g(x-a) : (a,b) \in \Z{\times}\frac{2}{3}\Z \}$ is an overcomplete
%frame of $L^2(\R)$ \cite{g}, there exist $\ell_2$ sequences $c = \{ c_{m,n} \}_{m,n\in\Z}$
%and $d = \{ c_{m,n} \}_{m,n\in\Z}$ with $(c,d) \neq 0$ and
%\(
%  \sum_{m,n \in \Z}
%    c_{m,n} e^{2\pi i(2/3)nx} g(x-2m)
%  - \sum_{m,n \in \Z}
%      d_{m,n} e^{2\pi i(2/3)nx}g(x-2m+1)
%  = 0
%\),
%but the fact that $\{ e^{2\pi i(2/3)nx} g(x-2m) : m,n \in \Z \}$
%and $\{ e^{2\pi i(2/3)nx} g(x-2m) : m,n \in \Z \}$ are Riesz sequences
%implies $c \neq 0$ and $d \neq 0$.
%Setting $f(x) = \sum_{m,n \in \Z} c_{m,n} e^{2\pi i(2/3)nx} g(x-2m)$, we have
%that $f(x), f(x -1) \in \calG(g,2\Z{\times}\frac{2}{3}\Z)$.}
%a function $f \in \calG(g,2\Z{\times}\frac{2}{3}\Z)$ satisfying
%$f(\cdot - 1) \in \calG(g,2\Z{\times}\frac{2}{3}\Z)$,
%which implies $\dist(f(\cdot-1),\calG(f,2\Z{\times}\frac{2}{3}\Z)) = 0$
%while we have $\dist ( (1,0),2\Z{\times}\frac{2}{3}\Z ) = 1$.

\medskip{}

\emph{Implications regarding the OFDM communication scheme.}
One motivation for analyzing the distance of the time-frequency shift
$\pi(u,\eta) g$ to the Gabor space $\mathcal{G}(g,\Lambda)$
stems from the communication scheme called \emph{orthogonal frequency division multiplexing} (OFDM).
In OFDM, the sender wants to transmit the coefficients
$c = (c_{k,\ell})_{k,\ell \in \Z} \in \ell^2(\Z^2)$ to the receiver.
This is done by selecting a fixed Gabor Riesz sequence
$\big( \pi (k \alpha, \ell \beta) g \big)_{k,\ell \in \Z}$ to form the
\emph{transmission signal} $F c = \sum_{k,\ell \in \Z} c_{k,\ell} \, \pi(k \alpha, \ell \beta) g$,
which is then sent to the receiver through a \emph{communication channel}.
Mathematically, the effect of the channel is modeled as a linear operator $T : L^2(\R) \to L^2(\R)$;
that is, the signal that arrives at the receiver is $T F c$ instead of $F c$.

The first step of the reconstruction procedure in OFDM is to apply
the \emph{reconstruction operator} $R$ given by
$R f = \big( \langle f , \pi(k \alpha, \ell \beta) g^\circ \rangle \big)_{k,\ell \in \Z}$
to the signal $T F c$,
thereby obtaining the sequence $\widetilde{c} = R T F c \in \ell^2(\Z^2)$.
%The usual reconstruction procedure employed in OFDM mandates
%that the receiver compute the coefficients
%\(
%  \widetilde{c}
%  = R T F c
%  = \big( \langle T F c , \pi(k \alpha, \ell \beta) g^\circ \rangle \big)_{k,\ell \in \Z}
%\),
Here, $g^\circ$ is the dual window for the Riesz sequence
$\big( \pi (k \alpha, \ell \beta) g \big)_{k,\ell \in \Z}$;
i.e., ${g^\circ \in \mathcal{G} := \mathcal{G}(g, \alpha \Z {\times} \beta \Z)}$
satisfies the biorthogonal property
\(
  \langle \pi(k \alpha, \ell \beta) g, \pi(k' \alpha, \ell' \beta) g^\circ \rangle
  = \delta_{k,k'} \delta_{\ell,\ell'} 
\) 
for $k,k',\ell,\ell' \in \Z$.
%that is \red{$g,g^\circ$ satisfy the biorthogonal property}
%\(
%  \langle \pi(k \alpha, \ell \beta) g, \pi(k' \alpha, \ell' \beta) g^\circ \rangle
%  = \delta_{k,k'} \delta_{\ell,\ell'}
%\)
%and $g^\circ \in \mathcal{G} := \mathcal{G}(g, \alpha \Z {\times} \beta \Z)$.
At least for the ideal communication channel $T = \operatorname{Id}_{L^2(\R)}$, this guarantees
perfect reconstruction, meaning that $\widetilde{c} = c$.
For more general channels, this is not the case, but one might hope to reconstruct
$c$ by applying a suitable (linear) \emph{post-processing operator} $P$ to $\widetilde{c}$.

In fact, there exists such a bounded post-processing operator $P$ satisfying $P R T F c = c$
for all $c \in \ell^2(\Z^2)$ if and only if the operator $R T$ is bounded below on the Gabor space
$\mathcal{G}$, meaning that $\|R T f\|_{\ell^2} \gtrsim \|f\|_{L^2}$
for all $f \in \mathcal{G}$.
It is not hard to see that $\|R h\|_{\ell^2} \asymp \| \bP h \|_{L^2}$ for $h \in L^2(\R)$,
where we denote by $\bP$ the orthogonal projection onto the Gabor space $\mathcal{G}$.
For the important special case that $T$ is a pure time-frequency shift
$\pi(u,\eta)$, reconstruction is thus possible if and only if
$\| \bP \pi(u,\eta) f \|_{L^2} \gtrsim \|f\|_{L^2}$ for all $f \in \mathcal{G}$,
which is equivalent to the existence of a constant $c < 1$ satisfying
\begin{equation}
  \dist\big(\pi(u,\eta)f,\calG\big)
  = \big\| (I - \bP) \pi(u,\eta) f \big\|_{L^2} \leq c \, \|f\|_{L^2} ,
  \qquad \forall \, f \in \mathcal{G} .
  \label{eq:OffBandEnergyLoss}
\end{equation}
The term $\dist(\pi(u,\eta)f,\calG)$ measures the \emph{off-band energy loss}
caused by the time-fre\-quen\-cy shift $\pi(u,\eta)$, that is, the proportion of the signal energy
that gets ``pushed out of the Gabor space'' by applying the time-frequency shift $\pi(u,\eta)$.
Even in the case where the off-band energy loss is small enough so that \eqref{eq:OffBandEnergyLoss}
holds, it is interesting to know more precise upper and lower bounds for this quantity,
since it influences the stability of the reconstruction.
Theorem~\ref{t:quanty} shows that in the case $f = g$, the off-band energy loss is of the order
$\dist ( (u,\eta), \alpha \Z {\times} \beta \Z)$.
%Theorem~\ref{t:quanty} shows---for the case $f = g$---that the off-band energy loss
%is of the order $\dist ( (u,\eta), \alpha \Z {\times} \beta \Z)$.

\medskip{}

\emph{Structure of the proof.}
The outline of the proof of Theorem~\ref{t:quanty} is as follows.
First of all, we note that it suffices to establish the inequality \eqref{e:main}
for $(u,\eta)$ in a neighborhood of the origin $(0,0)$;
the inequality then holds for all $(u,\eta) \in \R^2$
(with possibly different constants $\LowerBound$ and $\UpperBound$) by Theorem~\ref{t:clpp}
and a compactness argument.
%%% This is a consequence of Theorem \ref{t:clpp} and a compactness argument.
%%% In order to treat the behaviour of the subspace distance for $(u,\eta)$ close to the origin,
In order to analyze the behavior of the quantity $\dist(\pi(u,\eta) g,\calG(g,\Lambda))$
for $(u,\eta)$ close to the origin, we first show that the \emph{time-frequency map}
$S_g : (a,b) \mapsto \pi(a,b) g$ is differentiable at $(0,0)$
with (Fr\'echet) derivative $(a,b) \mapsto -ag'+2\pi i b X g$, where $X$ is the position operator
defined formally by $Xf (x) = x f(x)$; see Lemma~\ref{l:differentiable}.
We then prove in Proposition~\ref{p:weak} that $-ag'+2\pi i b X g$ is not contained
in $\calG(g,\Lambda)$ unless $a=b=0$.
Denoting by $\bP$ the orthogonal projection from $L^2(\R)$ onto $\calG(g,\Lambda)$,
this implies that there exists a constant $\gamma>0$
with $\|(I-\bP)(-ag'+2\pi i b X g)\|_{\LTwoIndex} \ge \gamma \, \|(a,b)\|_2$ for all $(a,b)\in\R^2$.
The claim then follows immediately because $(a,b) \mapsto (I-\bP)(-ag'+2\pi i b X g)$ linearizes
the map $(a,b) \mapsto (I-\bP)(e^{2\pi ib x}g(x-a))$
in a neighborhood of $(0,0)$.

The main ingredients of the proof are thus the differentiability of the time-fre\-quen\-cy map
(see Section~\ref{s:diff}) and the fact that none of its directional derivatives
${-ag' + 2\pi i b X g}$ with $(a,b)\in\R^2\backslash\{(0,0)\}$ are contained in $\calG(g,\Lambda)$
(see Proposition~\ref{p:weak}).
While the former is probably folklore (although we could not find a reference),
the latter seems to be a new result and should be interesting in its own right.
%%% While the first of these two results is probably folklore (although we could not find a reference),
%the second result seems to be new and should be interesting in its own right.
%To the best of our knowledge, these results are new and should be interesting in their own right.
%\todo{I doubt that the differentiability of the time-frequency map is really new;
%this might be folklore.}
%%%%% Let us remark that Proposition~\ref{p:weak} is closely related
%%%%% to the so-called {\it weak Balian-Low theorem for subspaces} \cite{ghhk}
%%%%% which states that for $g\in L^2(\R)$ with $\{e^{2\pi ibx}g(x-a) : (a,b)\in\Lambda\}$
%%%%% being a Riesz sequence, at least one of the distributions $g'$, $X g$, $\wt{g}'$, $X \wt{g}$
%%%%% is not contained in $\calG(g,\Lambda)$, where $\wt{g}$ denotes the dual window to $g$.
%%%%% The conclusion of Proposition~\ref{p:weak} is that \emph{neither} $g'$ nor $X g$---nor any
%%%%% non-trivial linear combination of these functions---belongs to $\calG(g,\Lambda)$,
%%%%% assuming that $g \in \HH^1 (\R)$ and that $\Lambda$ has rational density.
We also point out a close relationship between Proposition~\ref{p:weak}
and the weak Balian-Low theorem for subspaces from \cite{ghhk};
see Remark~\ref{r:weak} for a detailed discussion.

As mentioned above, we were unable to derive a closed-form formula for the constant $\LowerBound$
in Equation~\eqref{e:main}.
However, if we assume the Gabor system $\{e^{2\pi ibx}g(x-a) : (a,b) \!\in\! \Lambda\}$
to be \emph{orthonormal}, then we can find an \emph{explicit} constant $\LowerBound > 0$
such that \eqref{e:main} holds for all $(u,\eta)$ in a neighborhood of the lattice $\Lambda$;
see Theorem~\ref{t:ortho}.
This result then leads to a statement similar to Theorem~\ref{t:clpp}
but without assuming the rational density of $\Lambda$; see Corollary~\ref{c:new}.
%%% This result then leads to a statement similar to Theorem~\ref{t:clpp},
%%% where the assumption on the rational density of the lattice is dropped (see Corollary~\ref{c:new}).

\medskip{}

The paper is organized as follows:
In Section~\ref{s:prep},
%we provide the necessary notions,
%notations, and some basic results which will be used throughout the paper.
%In particular,
we show how the main properties that we are interested in
(the regularity of $g$, the property of $(g,\Lambda)$ being a Riesz sequence,
and the distance $\dist (\pi(\mu), \calG(g,\Lambda))$)
can be described via the Zak transform and certain associated matrix multiplication operators.
Section~\ref{s:diff} contains the aforementioned differentiability result
for the time-frequency map of $\HH^1(\R)$ functions.
%In Section \ref{s:ODE} we state a result on the solutions of certain matrix ODEs
%that we shall use in the proof of the Theorem \ref{t:quanty}
%which Section \ref{s:quanty_proof} is devoted to.
The proof of Theorem~\ref{t:quanty} is given in Section~\ref{s:quanty_proof}.
Finally, in Section~\ref{s:ortho} we provide an explicit local lower bound $\LowerBound$
in the case where the Gabor system is orthonormal.

Several results that are technical or only tangentially related to the core arguments
are deferred to \ref{sec:Auxiliary}.
Although most of them should be well-known or be considered folklore,
we either give detailed references or include their proofs for the sake of completeness.
%%% Most of these results should be well-known or folklore; for the convenience of the reader,
%%% we nevertheless provide proofs for most of these results.

\section{Preparations}
\label{s:prep}

\noindent{\bf Notation.}
%\todo[inline]{Should we recall some of the following notions: frames, frame sequences,
%Riesz bases, Riesz sequences?}
Let us begin with collecting some notation which will be used throughout the paper.
We set $\N := \{1,2,\ldots\}$ and $\N_0 := \N\cup\{0\}$.
The closure of a subset $M \subset X$ of a metric space $X$ will be denoted by $\overline{X}$.
The \emph{Lebesgue measure} of a Borel set $E \subset \R^n$ is denoted by $\Lebesgue(E)$.
If $g : \R\to\C$ is measurable, we write $Xg$ for the function $x \mapsto x \, g(x)$,
that is, $(X g)(x) = x \, g(x)$, $x\in\R$.

Let $\calH$ be a Hilbert space, and $\Phi = (\varphi_i)_{i \in I}$ be a family of vectors
in $\calH$.
This family is called a \emph{frame} for $\calH$ if
\({
  A \, \|f\|_{\calH}^2
  \leq \sum_{i \in I} |\langle f, \varphi_i \rangle|^2
  \leq B \, \|f\|_{\calH}^2
}\)
for all $f \in \calH$ and certain constants $A,B \in (0,\infty)$.
If $\Phi$ is a frame for its closed linear span $\overline{\linspan} \{\varphi_i \colon i \in I \}$,
then we say that $\Phi$ is a \emph{frame sequence}.
We say that $\Phi$ is a \emph{Riesz sequence} if there are $A,B \in (0,\infty)$ such that
${A \, \|c\|_{\ell^2} \leq \big\| \sum_{i \in I} c_i \, \varphi_i \big\| \leq B \, \|c\|_{\ell^2}}$
for all finitely supported sequences $c = (c_i)_{i \in I} \in \ell^2 (I)$.
If $\Phi$ is a Riesz sequence and $\linspan \bigl\{\varphi_i \colon i \in I\bigr\}$
is dense in $\calH$, we say that $\Phi$ is a \emph{Riesz basis} for $\calH$.
Each Riesz basis is a frame.

Let $T : \calH\to\calH$ be a bounded linear operator on a (complex) Hilbert space $\calH$.
The \emph{spectrum} of $T$ will be denoted by $\sigma(T)$; that is,
\[
  \sigma(T)
  = \big\{
      \la \in \C
      \colon
      T - \la I \text{ is not boundedly invertible}
    \big\} \, .
\]
We denote by $\varrho(T)$ the complement set of $\sigma(T)$ in $\C$ which is called
the {\em resolvent set} of $T$.
%The complement of $\sigma(T)$ in $\C$ is denoted by $\rho(T)$ and is called
%the {\em resolvent set} of $T$.
For a bounded linear operator $A : \calH \to \calK$ between two Hilbert spaces $\calH$ and $\calK$,
we define
\begin{equation}
  \sigma_0(A) \!:=\! \sqrt{\min\sigma(A^*A)} \! \in \! [0,\infty)
  \!\!\quad\text{and}\quad\!\!
  \sigma_1(A) \!:=\! \sqrt{\inf[\sigma(A^*A)\backslash\{0\}]} \! \in \! [0,\infty] .
  \label{eq:SigmaDefinitions}
\end{equation}
Note that $\sigma(A^\ast A) = \{0\}$ if and only if $A = 0$,
in which case we have $\sigma_1 (A) = \infty$;
on the other hand, we have $\sigma_1 (A) < \infty$ for $A \neq 0$.
If $A$ is a matrix, then $\sigma_0(A)$ is the smallest singular value of $A$,
while $\sigma_1(A)$ is the smallest {\em positive} singular value of $A$.

We occasionally consider the vector-valued $L^2$ space $L^2(\Omega;\C^k)$,
which we equip with the inner product
$\langle f,g \rangle = \strut \int \langle f(\omega), g(\omega) \rangle_{\C^k} d\mu(\omega)$,
where $\langle \cdot, \cdot \rangle_{\C^k}$ denotes the standard inner product on $\C^k$.

The \emph{Fourier transform} $\widehat{g}$ of $g \in L^2(\R)$ is defined by
\[
  \widehat{g} (\omega)
  := \lim_{R \to \infty}\int_{-R}^R g(x) \, e^{-2\pi ix\omega}\,dx \, ,
\]
where the limit is taken in $L^2(\R)$.
For $a,b\in\R$ we also define the time-frequency shift operator
\[
  [\pi(a,b)g](x)
  := e^{2\pi ibx}g(x-a),\qquad x\in\R \, ,
\]
which can be expressed as $\pi(a,b) = M_b T_a$ where $T_a$ and $M_b$
denote the operators of \emph{translation} by $a\in\R$
and \emph{modulation} by $b\in\R$, respectively.
For $k \in \N$, we set $\HH^k (\R) := \{ f \in H^k(\R) \colon \widehat{f} \in H^k(\R) \}$,
with the usual (complex-valued) $L^2$-Sobolev space $H^k (\R) = W^{k,2}(\R)$.
%Denoting by $T_a$ and $M_b$ the operators of \emph{translation} by $a\in\R$
%and \emph{modulation} by $b\in\R$, respectively, we have $\pi(a,b) = M_bT_a$.

A \emph{lattice} in $\R^2$ is a set $\Lambda =  A\Z^2$ with $A\in \GL(2,\R)$.
Its \emph{density} is defined as $|\det A\,|^{-1}$.
If $\Lambda$ is a lattice in $\R^2$ and $g\in L^2(\R)$,
we denote by $(g,\Lambda)$ the \emph{Gabor system} generated by $g$ and $\Lambda$, that is,
\[
  (g,\Lambda) := \{\pi(\la)g : \la\in\Lambda\}.
\]
The \emph{Gabor space} generated by $g$ and $\Lambda$ is
defined as ${\calG(g,\Lambda) := \ol{\linspan}\,(g,\Lambda)}$,
with the closure taken in $L^2 (\R)$.

The \emph{Zak transform} of $g\in L^2(\R)$ is defined as
\begin{equation}
  Zg(x,\omega)
  = \lim_{N\to\infty}
      \sum_{k=-N}^N
        e^{2\pi ik\omega} g(x-k),
  \qquad (x,\omega)\in (0,1)^2,
  \label{eq:ZakTransformDefinition}
\end{equation}
where the limit is taken in $L^2((0,1)^2)$.
The Zak transform $g \mapsto Z g$ is a unitary operator from $L^2(\R)$ to $L^2((0,1)^2)$.
%Usually, the Zak transform $Z g$ of $g\in L^2(\R)$ is extended to a function on $\R^2$
%by means of the first of the following properties
%(each of which holds for a.e.\ $(x,\omega)\in\R^2$):\label{p:Zak}
In the following, we will consider the Zak transform $Z g$ of $g \in L^2(\R)$ as
an (a.e.~defined) function on $\R^2$, by using Equation~\eqref{eq:ZakTransformDefinition}
on all of $\R^2$, where the limit is taken in $L_{\mathrm{loc}}^2 (\R^2)$.
This extended Zak transform has the following properties
(all of which hold for a.e.~$(x,\omega) \in \R^2$):
\label{p:Zak}
\begin{enumerate}
  \item[{\rm (a)}] $Zg(x+m,\omega+n) = e^{2\pi im\omega} \, Zg(x,\omega)$ for all $m,n\in\Z$.

  \item[{\rm (b)}] $Z[\pi(u,\eta)g](x,\omega) = e^{2\pi i\eta x} \, Zg(x-u,\omega-\eta)$
                   for all $(u,\eta) \in \R^2$.

  \item[{\rm (c)}] $(Z[\pi(m,n)g])(x,\omega) = e^{2\pi i(nx-m\omega)} \, Zg(x,\omega)$
                   for all $m,n\in\Z$.

  \item[{\rm (d)}] $Z\widehat{g}(x,\omega) = e^{2\pi i x\omega}\,Zg(-\omega,x)$.

  \item[{\rm (e)}] $g(x) = \int_0^1 Zg(x,\omega)\,d\omega$
                   \;and\; $\wh g(\omega) = \int_0^1 e^{-2\pi ix\omega}Zg(x,\omega)\,dx$.
\end{enumerate}
For all these properties, we refer to \cite[Chapter~8]{g}.
The property (a) of $Zg$ is called {\em quasi-periodicity}.

\subsection{Reduction to matrix multiplication operators}%
%\label{sub:ReductionToMultiplicationOperators}

In this subsection, we show that the properties and quantities that we are interested in%
---the distance $\dist(\pi(\mu) g, \calG(g,\Lambda))$ and whether $(g,\Lambda)$ is a Riesz sequence%
---can be conveniently reformulated using certain matrix multiplication operators
\[M_A : L^2(\Omega; \C^k) \to L^2(\Omega; \C^\ell), f \mapsto A(\cdot) f(\cdot).\]
Here, the matrix function $A : \Omega \to \C^{\ell \times k}$ will be defined using
the Zak transform.
More details regarding these matrix multiplication operators can be found in \ref{ss:multi}.

We start by considering the Gabor system $(g,\Lambda)$
associated to the lattice $\Lambda = \frac{1}{Q} \Z {\times} P \Z$ (where $P,Q \in \N$)
and connect the spectral properties of the \emph{frame operator}
\[
  \frameOP : \quad
  L^2 (\R) \to L^2(\R), \quad
  f \mapsto \sum_{\lambda \in \Lambda} \langle f, \pi(\lambda) g \rangle \, \pi(\lambda) g
\]
and the \emph{Gram operator} $\gramian : \ell^2 (\Z^2) \to \ell^2 (\Z^2)$ defined by
\[
  \gramian (c_{n,k})_{n,k \in \Z}
  = \Big(
      \big\langle
        \sum_{n,k \in \Z}
          c_{n,k} \pi(Q^{-1} n, Pk) g, \,\,
        \pi(Q^{-1} m, P \ell) g
      \big\rangle
    \Big)_{m,\ell \in \Z}
\]
to matrix multiplication operators on the domain $R_P := (0, \tfrac{1}{P}) {\times} (0,1)$.
This relies on using the unitary operators
$\vectorizeOperator : L^2( (0,1)^2)\to L^2(R_P, \C^P)$
and ${\specialFourier : \ell^2(\Z^2)\to L^2(R_P, \C^Q)}$, defined by
\begin{equation}
  (\vectorizeOperator f)(x,\omega) := \big( f(x+\tfrac{k}{P}, \omega) \big)_{k=0}^{P-1}
  \quad \text{and} \quad
  \specialFourier c = \bigg(\sum_{s,n\in\Z}c_{sQ+\ell,n} \, e_{s,n}\bigg)_{\ell=0}^{Q-1},
  \label{eq:PeriodizationOperators}
\end{equation}
%\todo{The naming scheme here is a bit unconventional:
%In many papers, $\calF$ denotes the Fourier transform.}
where $f\in L^2( (0,1)^2)$ and $c = (c_{n,m})_{n,m \in \Z} \in \ell^2(\Z^2)$,
and where we use the function
${e_{s,n}(x,\omega) := P^{1/2} \cdot e^{2\pi i(nPx-s\omega)}}$ defined for $(x,\o)\in R_P$.
Furthermore, we denote by $S_n \in \C^{n {\times} n}$ the cyclic shift operator
satisfying $S_n e_i = e_{i-1}$ for $i \in \firstN{n-1}$ and $S_n e_0 = e_{n-1}$
for the standard basis $\{e_0,\dots,e_{n-1}\}$ of $\C^n$.
Finally, for $\omega\in\R$ we define the matrices
\[
  L_\omega
  := S_P \, \diag (e^{2\pi i \omega}, 1, \dots, 1) \in \C^{P {\times} P}
\]
and
\[
  %\rsout{M_\omega} 
  R_\omega
  = \diag(e^{-2\pi i \omega},1,\dots,1)\,S_Q^{-1} \in \C^{Q {\times} Q} \, .
\]

%\todo[inline]{Note: The multiplication operator $M_B$ that is heavily used in the following lemma
%is only defined in the appendix. We should probably change this!
%
%Also, it could be a good idea to indicate the domain $\Omega$ used in the definition of a
%multiplication operator. For example, $A_g A_g^\ast : \R^2 \to \C^{P {\times} P}$,
%but in the lemma below, the multiplication operator $M_{A_g A_g^\ast}$ only acts on
%$L^2 (R_P; \C^P)$ instead of $L^2 (\R^2; \C^P)$.}

\begin{lem}\label{l:matrix}
For $P,Q\in\N$ and $g\in L^2(\R)$, $g\neq 0$,
let us define the matrix function $A_g : \R^2 \to \C^{P{\times} Q}$ by
\[
  A_g(x,\omega)
  := \tfrac{1}{\sqrt P}\big(Zg(x+\tfrac{k}{P} - \tfrac{\ell}{Q}, \omega)\big)_{k,\ell=0}^{P-1,Q-1}.
\]
Then for a.e.\ $(x,\omega)\in\R^2$ we have
\begin{equation}\label{e:shifts}
  A_g(x + \tfrac 1P, \omega) = L_\omega A_g(x,\omega)
  \qquad\text{and}\qquad
  A_g(x - \tfrac 1Q, \omega) =  A_g(x,\omega) R_\omega \, .
  %M_\omega \, .
\end{equation}
In particular, $A_g^*A_g$ is $(\tfrac{1}{P},1)$-periodic
and $A_gA_g^*$ is $(\tfrac{1}{Q}, 1)$-periodic.

If $\Lambda = \tfrac{1}{Q} \Z {\times} P \Z$, then $(g,\Lambda)$ is a Bessel sequence
if and only if $Zg\in L^\infty(\R^2)$.
In this case, the synthesis operator
%\(
\[
  \synthOP : \quad
  \ell^2 (\Z^2) \to L^2(\R), \quad
  (c_{n,m})_{n,m \in \Z} \mapsto \sum_{n,m \in \Z} c_{n,m} \, \pi (Q^{-1}n, Pm) g ,
\]
%\),
the frame operator $\frameOP$, and the Gram operator $\gramian$ of $(g,\Lambda)$ satisfy
\begin{equation}\label{e:operators}
  \synthOP = (\vectorizeOperator Z)^\ast M_{A_g} \specialFourier,
  \quad
  \frameOP = (\vectorizeOperator Z)^\ast M_{A_g A_g^*} (\vectorizeOperator Z),
  \quad\text{and}\quad
  \gramian = \specialFourier^\ast M_{A_g^* A_g} \specialFourier,
\end{equation}
respectively, where $M_{A_gA_g^*}$
(respectively $M_{A_g^*A_g}$  or $M_{A_g}$) is the matrix multiplication operator
(cf.~\ref{ss:multi}) with respect to $A_gA_g^*$ (resp.~$A_g^*A_g$ or $A_g$)
acting on $L^2(R_P;\C^P)$ (resp.~$L^2(R_P;\C^Q)$).

If $Zg\in L^\infty(\R^2)$, the following statements hold:
\begin{enumerate}
\item[{\rm (a)}] $(g,\Lambda)$ is a Riesz sequence
                 if and only if $\essinf_{z \in \R^2} \sigma_0(A_g(z)) > 0$.\\[-.4cm]

\item[{\rm (b)}] $(g,\Lambda)$ is a frame sequence
                 if and only if $\essinf_{z \in \R^2} \sigma_1(A_g(z)) > 0$.\\[-.4cm]

\item[{\rm (c)}] $(g,\Lambda)$ is a frame for $L^2(\R)$
                 if and only if $\essinf_{z \in \R^2} \sigma_0(A_g(z)^*) > 0$.
\end{enumerate}
\end{lem}

\begin{proof}
Let $A := A_g$. We have
$A(x + \tfrac 1P, \omega)
 = P^{-\frac{1}{2}}
   \cdot \big(
           Z g(x + \frac{k+1}{P} - \frac{\ell}{Q}, \omega)
         \big)_{k,\ell = 0}^{P-1,Q-1}$,
where---due to the quasi-periodicity of $Zg$---we see that
\begin{align*}
  &\lefteqn{  Z g \big( x + \tfrac{k+1}{P} - \tfrac{\ell}{Q}, \omega \big) } \\
  &= \begin{cases}
      \sqrt{P} \cdot \big( A(x,\omega) \big)_{k+1,\ell}
      & \text{if } k < P-1 \, , \\
      %Z g(x+1 - \frac{\ell}{Q}, \omega)
      e^{2 \pi i \omega} Z g (x - \frac{\ell}{Q}, \omega)
      = \sqrt{P} \cdot e^{2 \pi i \omega} \cdot \big( A(x,\omega) \big)_{0,\ell}
       & \text{if } k = P-1 \, .
    \end{cases}
\end{align*}
In matrix notation, this means precisely that $A$ satisfies the first relation in \eqref{e:shifts},
and the $(\tfrac{1}{P}, 1)$-periodicity of $A^*A$ follows from $L_\omega^*L_\omega = \Id_{\C^P}$
and from $A(x,\omega+1) = A(x,\omega)$.
The second relation in \eqref{e:shifts} can be proved similarly
and shows that $AA^*$ is $(\tfrac{1}{Q}, 1)$-periodic.

\medskip{}

Let $\synthOP_0$ denote the pre-synthesis operator of $(g,\Lambda)$, that is,
\[
  \synthOP_0 : \quad
  \ell_0(\Z^2)\to L^2(\R) , \quad
  \synthOP_0(c_{m,n})_{m,n \in \Z}
  := \sum_{m,n \in \Z}
       c_{m,n} \, \pi(\tfrac{m}{Q},nP)g \, ,
\]
where $\ell_0(\Z^2)$ is the space of all elements of $\ell^2(\Z^2)$
with only finitely many non-zero entries.
For $c \in \ell_0(\Z^2)$, the properties of the Zak transform listed after
Equation~\eqref{eq:ZakTransformDefinition} show that
\begin{align*}
  (Z \, \synthOP_0 \, c)(x,\omega)
  & = \sum_{m,n \in \Z}
        c_{m,n} \, Z [ \pi(\tfrac{m}{Q}, nP) \, g ] (x,\omega) \\
  & = P^{-1/2}
      \cdot \sum_{\ell=0}^{Q-1}
              h_\ell(x,\omega) \, Z g(x-\tfrac{\ell}{Q},\omega)
    = \<A(x,\omega)h(x,\omega),e_0\>_{\C^P},
\end{align*}
where $h_\ell(x,\omega) := P^{1/2} \sum_{s,n\in\Z} c_{sQ+\ell,n} \, e^{2\pi i(nPx-s\omega)}$
and $h := (h_\ell)_{\ell=0}^{Q-1} = \specialFourier c$
with $\specialFourier$ defined in Equation~\eqref{eq:PeriodizationOperators}.
Since $h$ is $(\tfrac{1}{P}, 1)$-periodic, we obtain for $k \in \{ 0,\dots,P-1 \}$ that
\[
  (Z \, \synthOP_0 \, c)(x + \tfrac{k}{P}, \omega)
  = \< A(x + \tfrac{k}{P}, \omega) h(x, \omega), e_0\>_{\C^P}
  = \< A(x, \omega) h(x, \omega), e_k \>_{\C^P} \, .
\]
Here, we used the identity $A(x+\tfrac 1P, \omega) = L_\omega A(x,\omega)$ from the beginning of
the proof to get
\[
  \langle A(x+\tfrac{k}{P}, \omega) h(x,\omega), e_0 \rangle
  = \langle L_\omega^k A(x,\omega) h(x,\omega), e_0 \rangle
  = \langle A(x,\omega) h(x,\omega), (L_\omega^\ast)^k e_0 \rangle \, ,
\]
where a straightforward induction shows
${(L_\omega^\ast)^k e_0 = \big( \diag(e^{-2\pi i \omega}, 1,\dots,1) \, S_P^\ast \big)^k e_0 = e_k}$
for $k=0,\dots,P-1$.
With the operator $\vectorizeOperator$ defined in Equation~\eqref{eq:PeriodizationOperators},
we have thus shown
\[
  (\vectorizeOperator Z \synthOP_0 \, c)(x,\omega) = A(x,\omega) h(x,\omega);
  \quad \text{that is,} \quad
  \vectorizeOperator Z \synthOP_0 = M_A \, \specialFourier|_{\ell_0(\Z^2)} \, .
\]
Since the operators $\vectorizeOperator,Z,\specialFourier$ are unitary,
this shows that $\synthOP_0$ is bounded if and only if $M_A$ is bounded,
that is, if and only if each entry of $A$ is essentially bounded (on $R_P$),
which---by quasi-periodicity---exactly means that $Zg\in L^\infty(\R^2)$.
In particular, this shows that $(g,\Lambda)$ is a Bessel sequence
if and only if $\synthOP_0$ is bounded, if and only if $Z g \in L^\infty (\R^2)$.

\medskip{}

Let us assume for the rest of this proof that $Zg\in L^\infty(\R^2)$.
Then ${\vectorizeOperator Z \synthOP = M_A \specialFourier}$,
where $\synthOP = \ol{\synthOP_0} = (\vectorizeOperator Z)^\ast M_A \specialFourier$
is the synthesis operator of $(g,\Lambda)$.
Clearly, ${M_A^\ast = M_{A^\ast}}$ is the (bounded) multiplication operator with $A^\ast$;
thus $M_A^*M_A = M_{A^*A}$ and $M_AM_A^* = M_{AA^*}$.
Since $\frameOP = \synthOP \synthOP^\ast$ and $\gramian = \synthOP^\ast \synthOP$,
this proves \eqref{e:operators}.

\medskip{}

By definition, $(g,\Lambda)$ is a Riesz sequence if and only if
the synthesis operator $\synthOP$ is bounded below.
Lemma~\ref{lem:BoundedBelowAStarA} shows that this holds if and only if
$\bG = \synthOP^\ast \synthOP$ is boundedly invertible, that is, if and only if $0 \in \varrho(\bG)$.
%Equivalently, $\bG = U^*U$ is bounded below, i.e., $0\in\varrho(\bG)$.
Similarly, $(g,\Lambda)$ is a frame for $L^2(\R)$ if and only if $0 \in \varrho(\bS)$.
Likewise, $(g,\Lambda)$ is a frame sequence
if and only if $(0,\veps_0]\subset\varrho(\gramian)$ for some $\veps_0 > 0$
(see Lemmas~\ref{lem:closed_equis} and \ref{lem:FrameSequenceIffGramianHasClosedRange}).
Hence, $(g,\Lambda)$ is a Riesz sequence if and only if $0\in\varrho(M_{A^*A})$,
a frame sequence if and only if $(0,\veps_0]\subset\varrho(M_{A^*A})$ for some $\veps_0 > 0$,
and a frame for $L^2(\R)$ if and only if $0\in\varrho(M_{AA^*})$.

The statements (a)--(c) now follow from Lemma~\ref{l:multi} (ii) and (iii).
Here, it is used for properties (a) and (b) that $\sigma_i (A_g (z))$ only depends on
$A_g^\ast (z) A_g (z)$, which is $(P^{-1}, 1)$-periodic, so that
$\essinf_{z \in \R^2} \sigma_i (A_g (z)) = \essinf_{z \in R_p} \sigma_i (A_g (z))$
for $i \in \{1,2\}$.
Finally, for property (c), it is used that if $(g,\Lambda)$ is a frame for $L^2(\R)$,
then $P/Q \leq 1$ (see \cite[Corollary~7.5.1]{g}), so that $R_Q \subset R_P$.
This implies
${\essinf_{z \in \R^2} \sigma_0 (A_g^\ast (z)) = \essinf_{z \in R_P} \sigma_0 (A_g^\ast (z))}$,
since $z \mapsto A_g (z) A_g^\ast (z)$ is $(Q^{-1}, 1)$-periodic.
Conversely, if $\essinf_{z \in \R^2} \sigma_0 (A_g^\ast (z)) > 0$,
then we also have ${\essinf_{z \in R_P} \sigma_0(A_g^\ast (z)) > 0}$, so that
$0 \in \varrho(M_{AA^\ast})$ by Lemma~\ref{l:multi}.
\end{proof}

In the next lemma, we derive a formula for the matrix function $A_{\widetilde{g}}$ associated
to the \emph{dual window} $\widetilde{g}$ of the Riesz sequence $(g,\Lambda)$.
This means that---considered on the Gabor space $\calG (g,\Lambda)$---the Gabor system
$(\widetilde{g}, \Lambda)$ is the canonical dual frame to $(g,\Lambda)$.
In the proof of the lemma, we will use that ${\widetilde{g} = \pseudo{\frameOP} g}$
satisfies this property, where $\pseudo{\frameOP}$ is the \emph{pseudo-inverse}
(see \ref{sub:PseudoInverseGeneral}) of the (pre)-frame operator
$\frameOP$ of $(g,\Lambda)$, which is given by
\(
  %\frameOP :
  %L^2(\R) \to L^2(\R),
  \frameOP f = \sum_{\lambda \in \Lambda} \langle f, \pi(\lambda) g \rangle \, \pi(\lambda) g
\).
For completeness, we sketch a proof of this fact.
Let ${\frameOP_0 := \frameOP|_{\calG(g,\Lambda)} : \calG(g,\Lambda) \to \calG(g,\Lambda)}$ denote the
restriction of $\frameOP$ to $\calG(g,\Lambda)$.
Note that $\frameOP_0$ is invertible since $(g,\Lambda)$ is a frame for $\calG(g,\Lambda)$,
and that $\calG(g,\Lambda) = \ran \frameOP = (\ker \frameOP)^{\perp}$
since $\frameOP$ is self-adjoint.
Therefore, the pseudo-inverse of $\frameOP$ is given by
$\pseudo{\frameOP} = \frameOP_0^{-1} \bP$,
where $\bP$ denotes the orthogonal projection
from $L^2(\R)$ onto $\calG(g,\Lambda)$.
Hence, $g_0 := \pseudo{\frameOP} g = \frameOP_0^{-1} g$.
Finally, a straightforward but tedious computation shows that
$\pi(\lambda) \frameOP = \frameOP \pi(\lambda)$ for all $\lambda \in \Lambda$,
which also implies that $\pi(\lambda) \calG(g,\Lambda) = \calG(g,\Lambda)$.
Therefore, $\frameOP_0 [\pi(\lambda) g_0] = \pi(\lambda) \frameOP_0 g_0 = \pi(\lambda) g$,
showing that
\({
  \big( \pi(\lambda) g_0 \big)_{\lambda \in \Lambda}
  = \big( \frameOP_0^{-1} (\pi(\lambda) g) \big)_{\lambda \in \Lambda}
}\)
is indeed the canonical dual frame of $(g,\Lambda)$.

With this preparation, we can now prove the announced lemma.
%derive a useful expression for the matrix function
%$A_{\widetilde{g}}$ defined in Lemma~\ref{l:matrix}, in terms of the function $A_g$.

\begin{lem}\label{l:dist}
Let $g\in L^2(\R)$, $P,Q \in \N$, $\Lambda = \tfrac 1Q\Z{\times} P\Z$,
and assume that $(g,\Lambda)$ is a Riesz sequence.
Let $\wt{g}$ be the dual window of $(g,\Lambda)$
%\todo{We never actually define what we mean by ``dual window''. Should we do this?}
%of $g$
and $G := Zg$, $\wt{G} := Z \wt{g}$, $A := A_g$,
and $\wt{A} := A_{\wt{g}}$, with $A_g$ and $A_{\widetilde{g}}$ as in Lemma~\rmref{l:matrix}.
Then
\begin{equation}\label{e:dual2}
  \wt{A} = A(A^*A)^{-1}
  \quad \text{almost everywhere on } \R^2 .
\end{equation}
Moreover, for arbitrary $\mu = (u,\eta)\in\R^2$ we have
\[
  \operatorname{dist}^2\!\big(\pi(\mu)g,\calG(g,\Lambda)\big)
  = \|g\|_{\LTwoIndex}^2
    - \int_0^1
        \int_0^{1/P}
          \left\| H_\mu(x,\omega) \, e_0 \right\|_{\C^P}^2
        \,dx
      \,d\omega \, ,
\]
where $e_0 = (1,0,\ldots,0)^T \in \C^Q$,
\[
  H_\mu(z)
  = P^{1/2}
    \cdot A(z) \cdot \big( A(z)^*A(z) \big)^{-1} \cdot A(z)^*
    \cdot e^{2\pi i\eta D_P}
    \cdot A(z-\mu)
  \in \C^{P {\times} Q} \, ,
\]
and $D_P = \operatorname{diag}(k/P)_{k=0}^{P-1}$, with the notation
$e^{2\pi i\eta D_P}: = \diag \big( (e^{2\pi i \eta k/P})_{k=0}^{P-1} \big)$.
\end{lem}

\begin{proof}
In this proof we shall make use of the notion of the \emph{pseudo-inverse} $\pseudo{T}$
of an operator $T$ with closed range.
For the definition of this notion and a review of some of its properties,
we refer to \ref{sub:PseudoInverseGeneral}.

Let $\bS$ be the frame operator of $(g,\Lambda)$.
Then the range of $\bS$ is $\ran\bS = \calG(g,\Lambda)$, which is closed in $L^2(\R)$.
Hence, by Lemma \ref{l:vphi} we have $\bS^\dagger = \vphi(\bS)$,
where $\varphi : \R \to \R$ is defined by $\varphi(0) = 0$ and $\varphi(t) = 1/t$ for $t \neq 0$.
%Note that $\wt{g} = \pseudo{\frameOP} g = \vphi(\bS)g$
%\todo{Should we add a proof for $\widetilde{g} = \pseudo{\frameOP} g$?}
As seen before the statement of the lemma, $\widetilde{g} = \pseudo{\frameOP} g = \varphi(\frameOP) g$.
Furthermore, Lemma~\ref{l:vphi} shows $\vphi(A(z)A(z)^*) = \pseudo{(A(z)A(z)^*)}$ for every $z\in R_P$.
%by Lemma~\ref{l:multi} (ii), Equation~\eqref{e:operators}, and Lemma~\ref{l:vphi}.
Hence, an application of Equation~\eqref{e:operators} and of Lemma~\ref{l:multi} (iv) shows that
\[
  \wt{G}
  = Z \wt{g}
  = Z [\vphi(\bS)g]
  = Z (\vectorizeOperator Z)^\ast \vphi(M_{A A^\ast})(\vectorizeOperator Z) g
  = \vectorizeOperator^\ast M_{\vphi(A A^\ast)} \vectorizeOperator G,
\]
with $\vectorizeOperator$ as defined in Equation~\eqref{eq:PeriodizationOperators}.
Therefore,
\begin{equation}
  \vectorizeOperator \wt{G} = \pseudo{(AA^*)}(\vectorizeOperator G)
  \qquad\text{a.e.\ on $R_P$}.
  \label{eq:DualWindowZakOnRectangle}
\end{equation}

In order to extend this relation to $\R^2$, define
\(
  (\wt{\vectorizeOperator} f)(x,\omega)
  := \big( f(x+\tfrac{k}{P}, \omega) \big)_{k=0}^{P-1}
\)
for $f : \R^2 \to \C$ and $(x,\omega) \in \R^2$.
Let $z=(x,\omega)\in R_P$ be arbitrary, and set $z_{n,k} = (x + \tfrac{n+k}{P}, \o)\vphantom{\sum^N}$
for $k \in \{0,\dots,P-1\}$ and $n \in \Z$.
Using Equation~\eqref{e:shifts}, we see that
\begin{align*}
  \widetilde{G} (z_{n,k})
  & = (\, Z \, \widetilde{g} \,) \big( x + \tfrac{n+k}{P}, \omega \big)
    = \sqrt{P} \cdot \Big( A_{\widetilde{g}} \big( x + \tfrac{n}{P}, \omega \big) \Big)_{k,0} \\
  & = \sqrt{P} \cdot \Big( L_\omega^n \, A_{\widetilde{g}} (x,\omega) \Big)_{k,0}
    = \Big( L_\omega^n \, \big( \widetilde{G} (z_{0,\ell}) \big)_{\ell=0}^{P-1} \Big)_k \, .
\end{align*}
Similarly, Equation~\eqref{e:shifts} shows that
$\big( G(z_{n,k}) \big)_{k=0}^{P-1} = L_\omega^{n} \big( G(z_{0,k}) \big)_{k=0}^{P-1}$.
Thus, we get for $(x,\omega) \in R_P$ and $n \in \Z$ that
%All in all, we thus get for $(x,\omega) \in R_P$ and $n \in \Z$ that
\begin{align*}
  (\widetilde{\vectorizeOperator} \widetilde{G}) (x + \tfrac{n}{P}, \omega)
  &= \left(\wt{G} (z_{n,k})\right)_{k=0}^{P-1}
   = L_\o^n \Big[ \left(\wt{G} (z_{0,k})\right)_{k=0}^{P-1} \Big]
   = L_\omega^n \big( [\vectorizeOperator \widetilde{G} ] (x,\omega) \big) \\
  ({\scriptstyle{\text{Eq.~}\eqref{eq:DualWindowZakOnRectangle}}})
  &= L_\o^n
     \pseudo{\big[
              A(x,\omega) A(x,\omega)^*
            \big]}
     \big( G(z_{0,k}) \big)_{k=0}^{P-1} \\
  ({\scriptstyle{\text{Eq.~}\eqref{e:shifts}}})
  &= L_\o^n
     \left[
       L_\o^{-n} A(x + \tfrac{n}{P}, \omega) A(x + \tfrac{n}{P}, \omega)^* L_\o^n
     \right]^\dagger
     L_\o^{-n}
     \big(
       G(z_{n,k})
     \big)_{k=0}^{P-1} \\
  ({\scriptstyle{\text{Corollary } \ref{cor:PseudoInverseConjugation}}})
  &= \pseudo
     {
       \big[
         A(x + \tfrac{n}{P}, \omega) A(x + \tfrac{n}{P}, \omega)^*
       \big]
     }
     \big(G(z_{n,k})\big)_{k=0}^{P-1} \\
  &= \pseudo
     {
       \big[
         A(x + \tfrac{n}{P}, \omega) A(x + \tfrac{n}{P}, \omega)^*
       \big]
     }
     (\widetilde{\vectorizeOperator} G) (x + \tfrac{n}{P}, \omega) \, .
\end{align*}
%Hence, setting $(\wt Vf)(x,\omega) := (f(x + \tfrac{k}{P}, \omega))_{k=0}^{P-1}$ for $f : \R^2\to\C$,
In combination with the $1$-periodicity in the second variable of all involved functions,
this implies
\[
  \wt{\vectorizeOperator} \wt{G} = \pseudo{(AA^\ast)} (\wt{\vectorizeOperator} G)
  \qquad\text{a.e.\ on $\R^2$}.
%%%  \label{eq:ZakTransformDualWindowRelation} % equation number removed !!!
%\tag{\rsout{11}}
\]
Since $(\wt{\vectorizeOperator} G)(x-\tfrac{\ell}{Q}, \o)$ is the $\ell$-th column
of the matrix $\sqrt{P} \cdot A(x,\o)$,
since $(\widetilde{\vectorizeOperator} \widetilde{G}) (x - \tfrac{\ell}{Q}, \omega)$ is
the $\ell$-th column of $\sqrt{P} \, \widetilde{A}(x,\omega)$,
and because $AA^*$ is $(\tfrac{1}{Q}, 1)$-periodic,
we obtain the identity $\wt{A} = \pseudo{(AA^*)} A = A \pseudo{(A^\ast A)} = A (A^*A)^{-1}$,
see Lemma~\ref{lem:TechnicalPseudoInverseIdentities} (iv).
%see \eqref{e:flip}.
Here, we used that $A^\ast A$ is invertible almost everywhere by
Lemma~\ref{l:matrix} (a). We have thus proved Equation~\eqref{e:dual2}.

\medskip{}

Now, denote the orthogonal projection from $L^2(\R)$ onto $\calG(g,\Lambda) = \ran \frameOP$ by $\bP$.
Then, for any $\mu = (u,\eta)\in\R^2$ we have
\[
  \operatorname{dist}^2\!\big(\pi(\mu)g,\calG(g,\Lambda)\big)
  = \|(I-\bP)\pi(\mu)g\|_{\LTwoIndex}^2
  = \|g\|_{\LTwoIndex}^2 - \|\bP \, \pi(\mu) g \|_{\LTwoIndex}^2 \, .
\]
Next, Lemmas~\ref{l:vphi} and \ref{l:multi} show
\({
  \pseudo{M_{A_g A_g^\ast}}
  = \varphi(M_{A_g A_g^\ast})
  = M_{\varphi(A_g A_g^\ast)}
  = M_{\pseudo{(A_g A_g^\ast)}}
  ,
}\)
which implies
\(
  \pseudo{\frameOP}
  = (\vectorizeOperator Z)^\ast \pseudo{M_{A_g A_g^\ast}} \vectorizeOperator Z
  = (\vectorizeOperator Z)^\ast M_{\pseudo{(A_g A_g^\ast)}} \vectorizeOperator Z
\)
thanks to Equation~\eqref{e:operators} and Corollary~\ref{cor:PseudoInverseConjugation}.
Now, Lemma~\ref{lem:TechnicalPseudoInverseIdentities} shows $\bP = \bS \pseudo{\bS}$.
Hence, Equations~\eqref{e:operators} and \eqref{e:range_identities} show
\[
  \bP
  \!=\! (\vectorizeOperator Z)^\ast M_{AA^*} M_{(AA^*)^\dagger} (\vectorizeOperator Z)
  = (\vectorizeOperator Z)^{\ast} M_{(AA^{*})(AA^{*})^\dagger} (\vectorizeOperator Z) %\\
  = (\vectorizeOperator Z)^\ast M_{P_{\ran A}} \! (\vectorizeOperator Z)
\]
and
\(
  P_{\ran A}
  = P_{\ran (A A^\ast)}
  = (A A^\ast) \, \pseudo{(A A^\ast)}
  = A \pseudo{(A^\ast A)} A^\ast
  = A (A^\ast A)^{-1} A^\ast.
\)
For arbitrary $f\in L^2(\R)$, we thus see that
\[
  \|\bP f\|_{\LTwoIndex}^2
  = \|M_{P_{\ran A}} \vectorizeOperator Zf\|_{L^2(R_P, \C^P)}^2
  = \int_0^1
      \int_0^{1/P}
        \|A(A^*A)^{-1}A^* \vectorizeOperator Zf\|_{\C^P}^2
      \,dx
    \,d\omega \, .
\]
Finally, since $Z [\pi(\mu) g] (x,\o) = e^{2\pi i\eta x} Zg(x-u,\omega-\eta)$
for $\mu = (u,\eta)$, we see
\begin{align*}
  (\vectorizeOperator Z [\pi(\mu) g]) (x,\omega)
  &= \big(
       e^{2\pi i\eta (x + \frac{k}{P})}
       Z g (x + \tfrac{k}{P} - u, \omega-\eta)
     \big)_{k=0}^{P-1} \\
  &= e^{2\pi i\eta x}
     e^{2\pi i\eta D_P}
     \big(
       Zg(x + \tfrac{k}{P} -u, \omega-\eta)
     \big)_{k=0}^{P-1} \\
  &= e^{2\pi i\eta x} \,
     P^{1/2} \,
     e^{2\pi i\eta D_P}
     A(x-u,\o-\eta) \, e_0 \, .
\end{align*}
Now, the claim follows from $|e^{2\pi i\eta x}|=1$.
\end{proof}

In proving the next result, we crucially use that if $\lambda = (\alpha,\beta) \in \R^2$
and ${\mu = (a,b) \in \R^2}$, then
$\pi(\lambda) \pi(\mu) f = e^{-2 \pi i \alpha b} \pi(\lambda + \mu) f$,
as can be verified by a direct calculation.
In particular, this implies
$\| T \pi(\lambda)\pi(\mu) f\|_{\LTwoIndex} = \| T \pi(\lambda + \mu) f\|_{\LTwoIndex}$
for any linear operator ${T : L^2(\R) \to L^2(\R)}$.

\begin{lem}\label{l:commute}
Let $g \in L^2(\R)$ and let $\Lambda\subset\R^2$ be a lattice.
If $\bP$ denote the orthogonal projection from $L^2(\R)$ onto $\calG(g,\Lambda)$.
Then $\bP$ commutes with the operators $\pi(\la)$, $\la\in\Lambda$.
In particular%, for any $\mu\in\R^2$, any $f \in L^2(\R)$, and any $\la\in\Lambda$ we have
\[
  \dist \big(\pi(\mu+\la)f, \calG(g,\Lambda) \big)
  = \dist \big(\pi(\mu)f, \calG(g,\Lambda) \big)
  \qquad \forall \, \mu\in\R^2, f \in L^2()\R, \lambda \in \Lambda.
\]
\end{lem}

\begin{proof}
Let $\calG := \calG(g,\Lambda)$.
Lemma~\ref{lem:GaborSpaceInvariance} shows $\pi(-\lambda) \calG \subset \calG$
for all $\lambda \in \Lambda$.
%We have $\pi(\lambda) \pi(\mu) g = e^{-2 \pi i \lambda_1 \mu_2} \pi(\lambda + \mu) g \in \calG$
%for any $\mu,\lambda \in \Lambda$, as discussed before the statement of the lemma.
%This easily implies $\pi(\la) \calG \subset \calG$ and---by symmetry of $\Lambda$---also
%$\pi(-\la) \calG \subset \calG$.
This implies $\pi(\la)\calG^\perp\subset\calG^\perp$:
Indeed, $\<\pi(\la)f,h\> = e^{2 \pi i \lambda_1 \lambda_2} \<f,\pi(-\la)h\> = 0$
for $f\in\calG^\perp$ and $h\in\calG$.
Now, if $f \in L^2(\R)$, then we can write $f = f_1 + f_2$ with $f_1 \in \calG$
and $f_2 \in \calG^\perp$;
hence, $\pi(\lambda) f = \pi(\lambda) f_1 + \pi(\lambda) f_2$ with $\pi(\lambda) f_1 \in \calG$
and $\pi(\lambda) f_2 \in \calG^{\perp}$,
which implies $\bP [\pi(\la)f] = \pi(\la)f_1 = \pi(\la)[\bP f]$.

As to the ``in particular''-part, we observe for $\mu \in \R^2$ and $\lambda \in \Lambda$ that
\begin{align*}
  \|(I-\bP)\pi(\mu+\la)f\|_{\LTwoIndex}
  & = \|(I-\bP)\pi(\la)\pi(\mu)f\|_{\LTwoIndex} \\
  & = \|\pi(\la)(I-\bP)\pi(\mu)f\|_{\LTwoIndex}
    = \|(I-\bP)\pi(\mu)f\|_{\LTwoIndex} \,.
\end{align*}
The claim now follows by noting $\dist(f, \calG) = \|(I-\bP)f\|_{\LTwoIndex}$ for $f\in L^2(\R)$.
\end{proof}

\subsection{Describing the regularity of \texorpdfstring{$g$}{g} via the Zak transform}%

%We assume that the following lemma is well known as a folklore.
The following lemma is probably folklore.
However, since we could not find any reference for it
(one direction is proved in \cite[Proof of Thm.\ 2.3]{d}), we give a full proof here.
Recall that $\HH^1(\R) = \{ f \in H^1(\R) \colon \widehat{f} \in H^1(\R) \}$.
%%% Recall that we denote by $\HH^1(\R)$ the space of all $g\in H^1(\R)$
%%%such that also $\widehat g\in H^1(\R)$.

\begin{lem}\label{l:gZg}
Let $g\in L^2(\R)$.
Then $g\in\HH^1(\R)$ if and only if $Zg\in H^1_{\rm loc}(\R^2)$.
In this case, the weak derivatives of $Z g$ are given by
\begin{equation}\label{e:diff_zak}
  \partial_1 Zg = Z(g')
  \quad \text{and} \quad
  \partial_2 Z g(x, \omega) = -2\pi i \big( [Z(Xg)](x,\omega) - x \cdot Zg(x,\omega) \big) \, .
\end{equation}
\end{lem}

\begin{proof}
``$\Rightarrow$:''
Assume that $g \in \HH^1(\R)$ and let $V \subset \R^2$ be nonempty, open, and bounded.
Let us first assume that $g\in C_c^\infty(\R)$ (such a function of course is in $\HH^1(\R)$).
Recalling the definition \eqref{eq:ZakTransformDefinition} of the Zak transform,
we see that on $V$, $Z g$ is defined by a \emph{finite} sum (hence $Zg\in C^\infty(V)$),
and the first relation in \eqref{e:diff_zak} is easily verified.
For the second relation, we note
\begin{align*}
  \partial_2 [Zg (x,\omega)]
  & = \sum_{k=-\infty}^\infty
        \partial_\omega [e^{2\pi i k \omega} \, g (x - k)]
    = \sum_{k=-\infty}^\infty
        2 \pi i k \, e^{2\pi i k \omega} \, g (x - k) \\
  & = \sum_{k=-\infty}^\infty
        e^{2\pi i k \omega} [2 \pi i x \, g(x-k) - 2\pi i (Xg)(x - k)] \\
  & = 2 \pi i \cdot \big[ x \cdot Z g (x,\omega) - Z[Xg] (x,\omega) \big] \, ,
\end{align*}
as claimed in \eqref{e:diff_zak}.

Now, let $g \in \HH^1(\R)$ be arbitrary.
%As $C_c^\infty(\R)$ is dense in $H^1(\R)$ (see, e.g., \cite{ml}),
Since $C_c^\infty(\R)$ is dense in $H^1(\R)$
(see for instance \cite[Section~E10.8]{AltLinearFunctionalAnalysis}),
we find a sequence $(\vphi_n)_{n \in \N}\subset C_c^\infty(\R)$ which converges to $g$ in $H^1(\R)$,
that is, $\vphi_n\to g$ and $\vphi_n'\to g'$ in $L^2(\R)$.
For $\phi\in C_c^\infty(V)$ we then have (with $\<\cdot,\cdot\> := \<\cdot,\cdot\>_{L^2(V)}$)
\begin{align*}
  \left|\left\<Zg,\partial_1\phi\right\> \right.&\left.+ \left\<Zg',\phi\right\>\right|\\
  &\le
  \left|\left\<Z(g-\vphi_n),\partial_1\phi\right\>\right|
  + \left|\left\<Z\vphi_n,\partial_1\phi\right\> + \left\<Z\vphi_n',\phi\right\>\right|
  + \left|\left\<Z(g' - \vphi_n'),\phi\right\>\right|.
\end{align*}
The middle term vanishes by partial integration and since $\partial_1 (Z \varphi_n) = Z(\varphi_n ')$;
the other two terms tend to zero as $n \to \infty$.
Hence, $\<Zg,\partial_1\phi\> = -\<Zg',\phi\>$.

The relation $\<Zg,\partial_2\phi\> = 2\pi i\<Z(Xg) - X Zg,\phi\>$ is proven similarly,
by noting that since $(1 + |X|) \, g \in L^2(\R)$, one can find%
\footnote{Indeed, given $\eps > 0$, there is a compactly supported
$h$ such that ${\|(1 + |X|) \, (g - h)\|_{\LTwoIndex} < \eps}$, say $\supp h \subset [-N,N]$.
Pick $\varphi \in C_c^\infty (\R)$ satisfying
$\|\varphi - h\|_{\LTwoIndex} < \eps/(1 + 2N)$ and ${\supp \varphi \subset [-2N, 2N]}$, whence
$\|(1 + |X|) \, (\varphi - h)\|_{\LTwoIndex} \leq (1 + 2N) \, \|\varphi - h\|_{\LTwoIndex} < \eps$,
so that finally $\|(1 + |X|) \, (g - \varphi) \|_{\LTwoIndex} < 2 \eps$.}
a sequence $(\varphi_n)_{n \in \N} \subset C_c^\infty (\R)$ satisfying
${\| (1 + |X|) \, (g - \varphi_n)\|_{\LTwoIndex} \to 0}$,
whence $\varphi_n \to g$ in $L^2$ and $X \varphi_n \to X g$ in $L^2$,
and therefore $Z(X \varphi_n) \to Z(X g)$ and $X \, Z \varphi_n \to X \, Zg$ with convergence in
$L_{\mathrm{loc}}^2 (\R^2)$.

Because of $Z g' \in L^2 (V)$ and $Z(Xg) - X Zg \in L^2(V)$,
this proves that ${Zg \in H^1(V)}$ and that \eqref{e:diff_zak} holds on $V$.
Since $V \subset \R^2$ was an arbitrary non-empty, open, bounded set,
we have proved one implication.

\medskip{}

``$\Leftarrow$:'' Assume that $G := Zg \in H^1_{\rm loc}(\R^2)$.
Lemma~\ref{l:slices} shows that, after changing $G$ on a null-set,
we can assume that $G(x,\cdot)$ is locally absolutely continuous
on $\R$ with derivative $(\partial_2 G)(x, \cdot)\in L^2_{\rm loc}(\R)$ for almost every $x \in \R$
and \emph{simultaneously} that $G(\cdot, \omega)$ is locally absolutely continuous on $\R$
with derivative $(\partial_1 G)(\cdot, \omega)\in L^2_{\rm loc}(\R)$
for almost every $\omega \in \R$.

According to the properties of the Zak transform, $g(x) = \int_0^1 G(x,\omega) \, d \omega$
for almost all $x \in \R$; see the list of properties below Equation~\eqref{eq:ZakTransformDefinition}.
Let us fix one $x_0 \in \R$ for which this is true.
Hence, for almost all $x \in \R$ we have
\begin{align*}
  g(x)
  &= \int_0^1
       G(x,\omega)
     \,d\omega
   = \int_0^1
       \left(
         G(x_0,\omega)
         + \int_{x_0}^x
             \partial_1G(t,\omega)
           \,dt
       \right)
     \,d\omega\\
  &= g(x_0) + \int_{x_0}^x
                \left(
                  \int_0^1
                    \partial_1G(t,\omega)
                  \,d\omega
                \right)
              \,dt
   = g(x_0) + \int_{x_0}^x \phi(t) \, d t \, ,
\end{align*}
where $\phi(t) := \int_0^1\partial_1G(t,\omega)\,d\omega$.
Note that $\phi\in L^1_{\rm loc}(\R)$ since ${\partial_1 G \in L^2_{\rm loc}(\R^2)}$.
Hence, possibly after redefining $g$ on a set of measure zero, $g$ is locally
absolutely continuous on $\R$.
To see that actually $\phi\in L^2(\R)$ (and hence $g\in H^1(\R)$),
recall from the properties of the Zak transform that
$G(t + n, \omega) = e^{2\pi i n \omega} G(t,\omega)$ for almost all $(t,\omega) \in \R^2$.
Hence,
\begin{align*}
  %\int_\R|\phi(t)|^2\,dt
  \| \phi \|_{L^2}^2
   = \sum_{n\in\Z}
       \int_0^1
         \left|
           \int_0^1 \!\!
             \partial_1G(t+n,\omega)
           \,d\omega
         \right|^2
       \!\! dt %\\
   = \int_0^1
       \sum_{n\in\Z}
         \left|
           \int_0^1
             e^{2\pi in\omega}
             \partial_1 G(t,\omega)
           \,d\omega
         \right|^2
     \!\! dt .
\end{align*}
Now, set $g_t(\omega) := \partial_1G(t,\omega)$ (which is in $L^2 ( (0,1) )$ for a.e.\ $t\in\R$).
Then
\begin{align*}
  \int_\R|\phi(t)|^2\,dt
  = \int_0^1\sum_{n\in\Z}\left|\wh{g_t}(n)\right|^2\,dt
  = \int_0^1\int_0^1|g_t(\omega)|^2\,d\omega\,dt
  = \|\partial_1G\|_{L^2([0,1]^2)}^2.
\end{align*}
Hence, $g\in H^1(\R)$ with $g'(x) = \int_0^1\partial_1 G(x,\omega)\,d\omega$.

\medskip{}

To see that also $\widehat{g} \in H^1(\R)$, define
$F : \R^2 \to \C, (x,\omega) \mapsto e^{-2 \pi i x \omega} G(x,\omega)$.
Since $G_x := G(x, \cdot)$ is locally absolutely continuous for almost all $x \in \R$,
the product rule for Sobolev functions
(see for instance \cite[Section~4.25]{AltLinearFunctionalAnalysis})
shows that also $F_x := F(x, \cdot)$ satisfies this property.
Moreover, the product rule also shows for almost all $x \in \R$ that we have
\begin{align*}
  F_x ' (\omega)
  & = e^{-2 \pi i x \omega} ( -2\pi i x G_x (\omega) + G_x ' (\omega) ) \\
  & = e^{-2 \pi i x \omega} \big( -2\pi i x \, G (x, \omega) + 
  %\rsout{(\partial_1 G)}
   (\partial_2 G) (x,\omega) \big)
    =: H(x,\omega)
\end{align*}
for almost all $\omega \in \R$.
Note that $H \in L_{\mathrm{loc}}^2 (\R^2)$, since $G \in H^1_{\mathrm{loc}}(\R^2)$.
This easily implies that the function $\psi : \R \to \C, \omega \mapsto \int_0^1 H(x,\omega) \, d x$,
is almost everywhere well-defined and satisfies $\psi \in L^1_{\mathrm{loc}}(\R)$.

%Since $G \in L^2_{\mathrm{loc}}(\R^2)$ and $G_x ' (\omega) = (\partial_1 G)(x,\omega)$
%for almost all $(x,\omega)$, we see that $\psi \in L^1_{\mathrm{loc}}(\R)$,
%where $\psi(\omega) := \int_0^1 F_x ' (\omega) \, d x$ for $\omega \in \R$.

Next, recall the inversion formula of the Zak transform 
(see the list of properties below Equation~\eqref{eq:ZakTransformDefinition}), stating
$\widehat{g} (\omega) = \int_0^1 e^{-2\pi ix\omega} G(x,\omega) \,d x = \int_0^1 F(x,\omega) \, d x$
for almost all $\omega \in \R$.
Fix some $\omega_0 \in \R$ for which this holds, and note for almost all $\omega \in \R$ that
\[
  \widehat{g}(\omega)
  = \int_0^1 F_x (\omega) \, d x
  = \int_0^1 \Big( F_x (\omega_0) + \int_{\omega_0}^\omega F_x ' (\gamma) \, d \gamma \Big) \, d x
  = \widehat{g}(\omega_0) + \int_{\omega_0}^{\omega} \psi(\gamma) \, d \gamma \, .
\]
Hence---possibly after changing $\widehat{g}$ on a null-set---we see that $\widehat{g}$ is locally
absolutely continuous, with $\widehat{g}'(\omega) = \psi(\omega)$, so that it remains to show
$\psi \in L^2(\R)$.

%To see this, define
%\(
%  H(x,\omega)
%  := e^{2 \pi i x \omega} \big( (\partial_1 G) (x,\omega) - 2 \pi i x \, G(x,\omega) \big) \, ,
%\)
%noting $H \in L_{\mathrm{loc}}^2(\R^2)$ and for almost all $x \in \R$ that
%$H(x,\omega) = F_x ' (\omega)$ for almost all $\omega \in \R$.

To see this, note for arbitrary $n \in \Z$ that
${G_x (\omega \!+\! n) \!=\! Zg (x, \omega \!+\! n) \!=\! Zg (x,\omega) \!=\! G_x (\omega)}$,
and hence also $G_x ' (\omega + n) = G_x ' (\omega)$, which finally implies
for almost all $x \in \R$ that $H(x,\omega+n) = e^{-2 \pi i n x} H(x,\omega)$
for almost all $\omega \in \R$.
Therefore, we see for any $n \in \Z$ that
$\psi(\omega + n) = \int_0^1 e^{-2\pi i n x} H(x,\omega) \, dx = \widehat{H_\omega}(n)$,
where $H_\omega$ is defined by $H_\omega (x) := H(x,\omega)$ for $x \in [0,1]$,
so that $H_\omega \in L^2([0,1])$ for almost all $\omega \in \R$.
Thus, we finally arrive at
\begin{align*}
  \int_{\R} |\psi(\omega)|^2 \, d \omega
  & = \sum_{n \in \Z}
        \int_0^1
          |\psi(\omega + n)|^2
        \, d \omega
    = \int_0^1
        \sum_{n \in \Z}
          |\widehat{H_\omega} (n)|^2
       \, d \omega \\
  & = \int_0^1 \|H_\omega\|_{\LTwoIndex}^2 \, d \omega
    = \int_0^1 \int_0^1 |H(x,\omega)|^2 \, d x \, d \omega
    < \infty \, .
    \qedhere
\end{align*}
%Similarly, using the inversion formula
%$\widehat{g} (\omega) = \int_0^1 e^{-2\pi ix\omega} G(x,\omega) \,d x$,
%one can show that $\widehat{g} \in H^1(\R)$.
\end{proof}

\subsection{Symplectic operators and the regularity of the dual window}%

In this subsection, we show that if $(g,\Lambda)$ is a Riesz sequence
with $g \in \HH^1(\R)$, then the canonical dual window $\widetilde{g}$
belongs to $\HH^1 (\R)$ as well.

For proving this---and also several other results---we shall make use of so-called
\emph{symplectic operators} to generalize statements involving lattices of the form
$Q^{-1}\Z{\times} P\Z$, $P,Q\in\N$,
to general lattices of rational density.
To explain this, let $\Lambda \subset \R^2$ be such a general lattice of rational density.
Then there exists a matrix $B \in \R^{2 {\times} 2}$ with $\det B = 1$ such that
$B\Lambda = Q^{-1}\Z{\times} P\Z$ with $P,Q\in\N$ co-prime.
Indeed, we have $\Lambda = A \Z^2$ for some $A \in \R^{2 {\times} 2}$
with $\det A \in \Q \backslash \{0\}$, that is, $|\det A| = P/Q$ for some co-prime $P,Q \in \N$.
Now define $B_0 := |\det A|^{1/2} \cdot A^{-1}$ if $\det A > 0$,
and if instead $\det A < 0$, then let $B_0 := |\det A|^{1/2} \cdot \diag (-1, 1) \cdot A^{-1}$.
It is not hard to check that $\det B_0 = 1$, and that $B_0 \Lambda = |\det A|^{1/2} \, \Z^2$.
Thus, the matrix $B := \diag \big( (P Q)^{-1/2}, (P Q)^{1/2} \big) \, B_0$
satisfies $\det B = 1$ and $B \Lambda =  Q^{-1}\Z{\times} P\Z$.

Next, since $\det B = 1$, we see from \cite[Lemma~9.4.1 and Equation~(9.39)]{g}
that there is a unitary operator $U_B : L^2(\R)\to L^2(\R)$ satisfying
\begin{equation}\label{e:mp}
  U_B \, \rho(z) = \rho(Bz) \, U_B
  \qquad \text{for all } z \in \R^2 \,,
\end{equation}
where
\begin{equation}\label{e:rho}
\rho(a,b)f := e^{-\pi iab}\pi(a,b)f,\quad \text{for } f\in L^2(\R) \text{ and } a,b \in \R .
\end{equation}
Such an operator $U_B$ is called \emph{symplectic}.
As a consequence of Schur's Lemma (see \cite[Lemma~9.3.2]{g}),
the operator $U_B$ is unique up to multiplication with unimodular constants;
thus, we see for $B,B_1,B_2 \in \SL(2, \R)$ that
\begin{equation}
  U_{B_1 B_2} = c_{B_1,B_2} \, U_{B_1} U_{B_2}
  \qquad\text{and}\qquad
  U_B^* = c_B \cdot U_{B^{-1}}
  \label{eq:MetaplecticMultiplicative}
\end{equation}
for certain constants $c_{B_1, B_2}, c_B \in \C$ with $|c_{B_1, B_2}| = 1 = |c_B|$.

For us, an important property of symplectic operators is that they leave $\HH^1 (\R)$ invariant.
To see this, recall from \cite[discussion around Equation~(4.5)]{clpp}
that each matrix $B \in \SL (2, \R)$ can be written as a product of matrices of the form
\[
  B_0
  := \left(
       \begin{matrix}
         0  & 1 \\
         -1 & 0
       \end{matrix}
     \right) \, ,
  \quad
  B_\alpha^{(1)}
  := \left(
       \begin{matrix}
         \alpha^{-1} & 0 \\
         0      & \alpha
       \end{matrix}
     \right) \, ,
  \quad \text{and} \quad
  B_\beta^{(2)}
  := \left(
       \begin{matrix}
         1     & 0 \\
         \beta & 1
       \end{matrix}
     \right)
  %\label{eq:SLFactorization}
\]
with $\alpha, \beta \in \R \backslash \{0\}$.
Furthermore, if we define operators $D_\alpha : L^2(\R) \to L^2(\R)$ and
$C_\beta : L^2 (\R) \to L^2(\R)$ by $D_\alpha f (x) := |\alpha|^{1/2} \cdot f(\alpha x)$
and $C_\beta f (x) = e^{\pi i \beta x^2} \cdot f(x)$, then a direct computation shows
that the choices $U_{B_\alpha^{(1)}} := D_\alpha$ and $U_{B_\beta^{(2)}} := C_\beta$
make \eqref{e:mp} valid.
%
%\todo[inline,caption=]{It turns out that the computation is not as straightforward as one would think
%(that is, the claim is \emph{false}).
%To see this, note that \eqref{e:mp} is equivalent to $U_B \pi(z) U_B^\ast = \pi(B z)$.
%Now, pick $B = S_\beta^{(2)}$.
%The claim is that we can choose $U_B = C_\beta$.
%But for $z = (x,\omega)$, we have $B z = (x, \omega + \beta x)$ and hence
%$\big( \pi(B z) f \big) (y) = e^{2 \pi i y (\omega + \beta x)} \, f(y - x)$.
%Likewise,
%\begin{align*}
  %(U_B \pi(z) U_B^\ast f) (y)
  %& = e^{2 \pi i \beta y^2} \big( \pi(x,\omega) U_B^\ast f \big) (y)
    %= e^{2 \pi i \beta y^2} \cdot e^{2 \pi i y \omega} (U_B^\ast f) (y-x) \\
  %& = e^{2 \pi i \beta y^2} \cdot e^{2 \pi i y \omega} \cdot e^{-2 \pi i \beta (y - x)^2} f(y-x) \, .
%\end{align*}
%Now, if we choose $f \in L^2$ to be continuous and satisfy $f > 0$ everywhere
%(for instance, take $f$ to be a Gaussian), then $U_B \pi (z) U_B^\ast f = \pi(B z) f$ holds
%if and only if
%\[
  %e^{2 \pi i y (\omega + \beta x)}
  %= e^{2 \pi i \beta y^2} e^{2 \pi i y \omega} e^{-2 \pi i \beta (y-x)^2} \, ,
%\]
%and this is equivalent to
%\[
  %e^{2 \pi i y \beta x} = e^{4 \pi i \beta x y} e^{-2 \pi i \beta x^2} \, .
%\]
%For $y = 0$, this implies $1 = e^{-2 \pi i \beta x^2}$, which is not satisfied in general.
%}
%
Likewise, if we let $U_{B_0} := \calF$ be the Fourier transform, then \eqref{e:mp} is satisfied
as well.

Thus, in view of \eqref{eq:MetaplecticMultiplicative}, it suffices to show that $\HH^1(\R)$
is invariant under the operators $\calF$, $D_\alpha$, and $C_\beta$.
For $\calF$ and $D_\alpha$, this is trivial.
Finally, for $C_\beta$ recall that $f\in L^2(\R)$ is in $\HH^1(\R)$
if and only if $Xf \in L^2(\R)$ and if $f$ is locally absolutely continuous with $f' \in L^2(\R)$.
As a consequence of the product rule for Sobolev functions
(see for instance \cite[Section~4.25]{AltLinearFunctionalAnalysis}),
it follows that if $g \in \HH^1(\R)$, then $C_\beta g$ is locally absolutely continuous,
with
\[
  (C_\beta \, g)'(x)
  = 2\pi i \beta x \cdot e^{\pi i \beta x^2}\cdot g(x) + e^{\pi i \beta x^2} \cdot g'(x)
  \in L^2 (\R) \, .
\]
Since $XC_\beta \, g\in L^2(\R)$ holds trivially,
we have $C_\beta \, g \in \HH^1(\R)$, as desired.

\medskip{}

To see an application of symplectic operators, note that if $\Lambda$ is a lattice of
rational density with $B \Lambda = Q^{-1} \Z {\times} P \Z$ for some $B \in \SL(2,\R)$,
and if $g\in L^2(\R)$ is such that $(g,\Lambda)$ is a Riesz sequence, one may define
\[
  g_1 := U_B \, g
  \qquad\text{and}\qquad
  \Lambda_1 := B \Lambda = \tfrac{1}{Q} \Z {\times} P \Z \, .
\]
Then \eqref{e:mp} implies $\pi(B\la)g_1 = c_\la U_B \, \pi(\la)g$, $\la\in\Lambda$,
where $c_\la = c_{\lambda} (B)$ is a unimodular constant.
Hence, $(g_1,\Lambda_1)$ is a Riesz basis for its closed linear span
$\calG(g_1,\Lambda_1) = U_B \, \calG(g,\Lambda)$.
This reduction to the separable lattice $\Lambda_1$ will be crucial in the proof of the following
proposition.

\begin{prop}\label{p:dual}
Let $g\in L^2(\R)$ and let $\Lambda\subset\R^2$ be a lattice of rational density
such that $(g,\Lambda)$ is a Riesz sequence.
Let $\wt{g}$ be the dual window of $(g,\Lambda)$.
%of $g$.
Then $g\in\HH^1(\R)$ if and only if $\wt{g} \in \HH^1(\R)$.
\end{prop}

\begin{proof}
Let us first prove the claim for $\Lambda = Q^{-1}\Z{\times} P\Z$, where $P,Q\in\N$.
Assume that $g\in\HH^1(\R)$.
By Lemma \ref{l:gZg}, $Zg\in H^1_{\rm loc}(\R^2)$.
Let us denote by $A_g$ and $A_{\widetilde{g}}$ the matrix functions introduced in
Lemma~\ref{l:matrix}.
Using that lemma, we conclude that each entry of $A_g$
is contained in $L^\infty(\R^2)$ and that there exists $c > 0$ such that $\sigma_0(A_g (z)) \ge c$
for a.e.\ $z\in\R^2$.
Therefore, a combination of Equation~\eqref{e:dual2} and
Lemmas~\ref{lem:TechnicalPseudoInverseIdentities} and \ref{lem:SobolevMatrices} shows that each entry of
$A_{\widetilde{g}} = A_g(A_g^*A_g)^{-1} = (A_g^\dagger)^*$
is contained in $H^1_{\rm loc}(\R^2)\cap L^\infty(\R^2)$.
In view of the definition of $A_{\widetilde{g}}$, this shows that
${Z \wt{g} \in H^1_{\rm loc}(\R^2)}$, whence Lemma \ref{l:gZg} implies $\wt{g} \in \HH^1(\R)$.
Since $(\wt{g}, \Lambda)$ is also a Riesz basis
for $\calG(g,\Lambda) = \calG(\widetilde{g}, \Lambda)$ with $(g,\Lambda)$ being the dual Riesz basis,
interchanging the roles of $g$ and $\wt{g}$ in the above arguments shows that
$\wt{g} \in \HH^1(\R)$ implies $g \in \HH^1(\R)$.
%%% As also $(\wt{g}, \Lambda)$ is a Riesz basis for $\calG(g,\Lambda) = \calG(\widetilde{g}, \Lambda)$
%%% with $\widetilde{\widetilde{g}} = g$, it follows from the above that $\wt{g} \in \HH^1(\R)$
%%% implies $g \in \HH^1(\R)$.

\smallskip

Now, let $\Lambda\subset\R^2$ be an arbitrary lattice of rational density.
As seen before Equation~\eqref{e:mp}, there is a matrix $B\in \SL(2, \R)$
such that ${\Lambda_1 := B\Lambda = Q^{-1}\Z{\times} P\Z}$ for certain $P,Q\in\N$.
Let $g_1 := U_B \, g$.
%and $\wt{g}_1 := U_B \, \wt{g} \in U_B \calG(g,\Lambda) = \calG(g_1, \Lambda_1)$.
Then $(g_1,\Lambda_1)$ is a Riesz basis for $\calG(g_1,\Lambda_1) = U_B \, \calG(g,\Lambda)$.
Furthermore, since $\pi(B \lambda) g_1 = c_\lambda U_B \pi(\lambda) g$ for $\lambda \in \Lambda$,
where $|c_\lambda| = 1$, it is not hard to see that the frame operator $\widetilde{\frameOP}$
for $(g_1, \Lambda_1)$ is given by $\widetilde{\frameOP} = U_B \frameOP U_B^\ast$,
where $\frameOP$ is the frame operator of $(g,\Lambda)$.
Hence, as discussed before Lemma~\ref{l:dist},
the dual window of $(g_1, \Lambda_1)$ is given by via the pseudo-inverse as
\({
  \widetilde{g_1}
  = \pseudo{\widetilde{\frameOP}} g_1
  = U_B \pseudo{\frameOP} U_B^\ast U_B g
  = U_B \widetilde{g}
  ,
}\)
where used that $\pseudo{\widetilde{\frameOP}} = U_B \pseudo{\frameOP} U_B^\ast$ due to
Corollary~\ref{cor:PseudoInverseConjugation}.
%and for $\la_1,\lambda_2\in\Lambda$ we have
%\[
%  \< \pi(B\la_1) g_1, \pi(B\la_2) \wt{g}_1 \>
%  = c_{\la_1}\ol{c_{\la_2}}\< \pi(\la_1)g, \pi(\la_2) \wt{g} \, \>
%  = c_{\la_1}\ol{c_{\la_2}}\cdot\delta_{\la_1, \la_2} \,,
%\]
%where $(c_\la)_{\la\in\Lambda}$ is a sequence of unimodular constants.
%This shows that $\wt{g}_1$ is the dual window to $g_1$ with respect to the basis $(g_1,\Lambda_1)$.

Now, suppose that $g\in\HH^1(\R)$.
As seen in the discussion before this proposition,
symplectic operators leave $\HH^1(\R)$ invariant;
%(see \cite[Proof of Thm.\ 1.4]{clpp}),
thus, $g_1 = U_B \, g\in\HH^1(\R)$.
Hence, by what we showed above, we see that $\wt{g}_1\in\HH^1(\R)$, which implies
${\wt{g} = U_B^* \, \wt{g}_1 = c_B \, U_{B^{-1}} \, \wt{g}_1\in\HH^1(\R)}$.
Finally, by interchanging the roles of $g$ and $\widetilde{g}$
we see that $\widetilde{g} \in \HH^1 (\R)$ implies $g \in \HH^1 (\R)$.
\end{proof}

\section{Differentiability of the Time-Frequency Map}
\label{s:diff}

In this section, we show that for $g \in \mathbb{H}^1 (\R)$ the map
$(a,b) \mapsto e^{2 \pi i b x} g(x - a)$ is differentiable at the origin,
with the derivative given by $(a,b) \mapsto -a g' + 2\pi i b X g$.
In the proof, we will make use of the following simple estimate.
Recall that the $\sinc$ function is defined by $\sinc(x) :=\tfrac{\sin(\pi x)}{\pi x}$
for $x \in \R \backslash \{0\}$ and $\sinc(0) := 1$.

\begin{lem}\label{lem:SincEstimate}
We have $\big| \tfrac{\sin(x)}x - e^{-ix} \big| \leq |x|$ for all $x \in \R \backslash \{0\}$.
Consequently,
\begin{equation}\label{e:sinc}
  |\sinc(x) - e^{-i\pi x}| \le \min\{2,\pi|x|\}
  \quad \text{for all $x \in \R$} \, .
\end{equation}
\end{lem}

\begin{proof}
The first inequality is equivalent to $|\sin(x) - xe^{-ix}| \le x^2$ and thus to
\[
  f(x)
  := (\sin(x)-x\cos(x))^2 + x^2 \sin^2(x) - x^4
  \le 0 \, .
\]
Since $f$ is even, it suffices to prove $f(x) \le 0$ for $x > 0$.
We have
\begin{align*}
  f'(x)
  &= 2x \sin(x) \, \big( \sin(x)-x\cos(x) \big)
     + 2 x \sin^2(x)
     + 2 x^2 \sin(x) \cos(x)
     - 4 x^3 \\
  &= 4x \sin^2(x) - 4 x^3
   = 4x (\sin(x) - x) (\sin(x) + x) \, .
\end{align*}
As $\sin(x) < x$ and $\sin(x)+x > 0$ for $x > 0$, we have that $f'(x) < 0$ for $x > 0$.
Since $f(0) = 0$, this proves the claim.
Equation~\eqref{e:sinc} is a direct consequence of the first estimate combined with
$|\sinc(x)| \leq 1$ and $|e^{-i \pi x}| \leq 1$.
\end{proof}

For $g\in L^2(\R)$ define the map $S_g : \R^2\to L^2(\R)$ by
\[%\label{e:Tg}
  S_g(a,b)
  := \pi(a,b)g
  = e^{2\pi ib\,(\cdot)} g(\,\cdot-a),
  \quad a,b\in\R.
\]
It is well known (see e.g.~\mbox{\cite[Lemma~2.9.2]{ChristensenIntroductionToFrames}})
that $S_g$ is continuous for every $g\in L^2(\R)$.
Here, we will show that $S_g$ is differentiable if $g\in\HH^1(\R)$.
We will first prove the differentiability of $S_g$ at the origin
and then use it to prove the differentiability of $S_g$ at arbitrary points $(u,\eta) \in \R^2$. 
%We first investigate the differentiability of $S_g$ at the origin.
%\red{For the convenience of the reader, we recall the notion of \braces{Fr\'echet}-differentiability:
%The map $S_g : \R^2 \to L^2(\R)$ is \braces{Fr\'echet}-differentiable at the origin
%if there exists a bounded linear map $T : \R^2 \to L^2(\R)$
%(the derivative of $S_g$ at the origin, denoted by $S_g '(0,0)$) satisfying
%\(
%  \frac{|S_g (a,b) - S_g (0,0) - T \left(\begin{smallmatrix} a \\ b \end{smallmatrix}\right)|}
%       {|(a,b)|}
%  \to 0
%\)
%as $(a,b) \to 0$; see e.g.~\cite[Chapter~XIII, §2]{LangRealFunctional} for more details.}
For the convenience of the reader, we recall the notion of \braces{Fr\'echet}-differentiability:
The map ${S_g : \R^2 \to L^2(\R)}$ is \emph{\braces{Fr\'echet}-differentiable at $(u,\eta) \in \R^2$}
if there is a bounded linear map $T : \R^2 \!\to\! L^2(\R)$ satisfying
\(
  \frac{|S_g (u+a,\eta+b) - S_g (u,\eta) - T \left(\begin{smallmatrix} a \\ b \end{smallmatrix}\right)|}
       {|(a,b)|}
  \to 0
\)
as $(a,b) \to 0$.
Such a map $T$ is referred to as the derivative of $S_g$ at $(u,\eta)$,
denoted by $S_g ' (u,\eta)$; see e.g.~\cite[Chapter~XIII, §2]{LangRealFunctional} for more details.

%\blue{///// Reviewer 1: A quick explanation of how to think about the derivative
%at the origin as a linear map, and an explanation of the notation used in Lemma~3.2
%would have helped me get oriented faster as someone who doesn't do these kinds of computations often.}

\begin{lem}\label{l:differentiable}
For any $g\in\HH^1(\R)$, the map $S_g$ is \braces{Fr\'echet}-differentiable at $(0,0)$ with
\[
  S_g'(0,0)\smallvek{a}{b} = -ag' + 2\pi i b Xg ,
\]
where $X$ is the position operator defined formally by $Xf (x) = x f(x)$.

If ${g\in\HH^2(\R)}$ (that is, $g \in H^2(\R)$ and $\widehat{g} \in H^2 (\R)$), then
\begin{equation}\label{e:H2}
  \|S_g(a,b) - g - (-ag' + 2\pi ibXg)\|_{\LTwoIndex}
  \,\le\, C_g \cdot \|(a,b)\|_2^2
  \quad \forall \, (a,b) \in \R^2 \, ,
\end{equation}
where
\begin{equation}\label{e:Cg}
  C_g := 3 \pi^2 \max \big\{
                        \|X^2g\|_{\LTwoIndex},
                        \|\omega^2 \widehat{g}\|_{\LTwoIndex},
                        \|Xg'\|_{\LTwoIndex}
                      \big\} \, .
\end{equation}
\end{lem}

\begin{rem*}
As shown in Lemma~\ref{lem:H2FunctionsNonSymmetricFourierOperators},
we indeed have $Xg'\in L^2(\R)$ if $g \in \HH^2(\R)$.
\end{rem*}

\begin{proof}
%We first prove that $S_g$ is differentiable at $(0,0)$, and then compute its derivative.
Let $\Psi_g : \R^2\to L^2(\R)$ be defined by $\Psi_g \smallvek{a}{b} := -a g' + 2\pi i b X g$;
in particular, $\Psi_g$ is linear.
We have to prove that
\begin{equation}
  \lim_{(a,b)\to (0,0)}
    \frac{\| S_g(a,b) - g - \Psi_g \smallvek{a}{b} \|_{\LTwoIndex}}
         {\sqrt{a^2+b^2}}
  = 0.
  \label{eq:FrequencyMapDerivativeTargetEstimate}
\end{equation}
To see this, we write
\begin{align}
  &  [S_g(a,b) - g - \Psi_g \smallvek{a}{b}] (x) \notag \\
  &= e^{2\pi ibx}g(x-a) - g(x) + ag'(x) - 2\pi ibxg(x) \notag \\
  &= e^{2\pi ibx}
     \bigl(g(x \!-\! a) - g(x)\bigr)
     + (e^{2\pi ibx} \!-\! 1 \!-\! 2\pi ibx) g(x)
     + a g'(x) \notag \\
  \begin{split}
    &= e^{2\pi ibx} \big(T_a g \!-\! g \!+\! a g'\big)(x)
       + \big[ e^{2\pi i b x} \!-\! 1 \!-\! 2\pi i b x\big] g(x)
       + a \big[1 \!-\! e^{2\pi i b x}\big] g'(x) \, .
    \label{ha}
  \end{split}
\end{align}
To estimate the middle term in \eqref{ha},
%We have
recall that ${\sinc(x) = \frac{\sin(\pi x)}{\pi x} = \frac{e^{\pi i x} - e^{- \pi i x}}{2\pi i x}}$
and hence
\begin{equation}
  e^{2\pi i b x} - 1 - 2\pi i b x
  = 2\pi i b x e^{\pi i b x} \big[\sinc(bx) - e^{-\pi i b x}\big] .
  \label{eq:ExpIDerivativeSinc}
\end{equation}
Therefore,
\begin{align}\label{huhu}
  \int \!
    \big|
      e^{2\pi i b x} \!-\! 1 \!-\! 2\pi i b x
    \big|^2
    |g(x)|^2
  dx
  = 4 \pi^2 b^2 \!\!
    \int \!\!
      x^2
      \left|
        \sinc(bx) \!-\! e^{-\pi i b x}
      \right|^2
      |g(x)|^2
    dx .
\end{align}
Using the estimate \eqref{e:sinc}, we find that this expression is not larger than
\[
  4\pi^4|b|^3
  \int_{|x| < 1/\sqrt{|b|}} \,\,
    x^2|g(x)|^2
  \,dx
  \;+\;
  16\pi^2b^2
  \int_{|x|\ge 1/\sqrt{|b|}} \,\,
    x^2|g(x)|^2
  \,dx \, .
\]
Hence, we obtain
\begin{equation}
  \begin{split}
    & \left(
        \frac{\| (e^{2\pi i b x} - 1 - 2\pi i b x) \cdot g \|_{\LTwoIndex}}
             {\sqrt{a^2 + b^2}}
      \right)^2 \\
    & \leq \frac{1}{b^2}
           \int
             \left|e^{2\pi ibx} - 1 - 2\pi ibx\right|^2|g(x)|^2
           \,dx \\
    & \leq 4\pi^4 |b| \, \|X g\|_{\LTwoIndex}^2
           + 16\pi^2
             \int_{|x| \geq |b|^{-1/2}} \,\,
               \!\!
               x^2 \, |g(x)|^2
             \,dx \, ,
  \end{split}
  \label{eq:ModulationDerivative}
\end{equation}
which tends to zero as $b \to 0$ as a consequence of $Xg \in L^2$
and the dominated convergence theorem.

For the first term in \eqref{ha}, observe that Plancherel's theorem yields
\begin{equation}
  %\begin{split}
  \|T_ag - g + ag'\|_{\LTwoIndex}^2
  %& = \|
  %      M_{-a} \widehat{g} - \widehat{g} + 2\pi ia\omega \widehat{g}
  %    \|_{\LTwoIndex}^2 \\
    = \int
        | e^{2\pi i(-a)\omega} - 1 - 2\pi i(-a)\omega |^2
        \cdot | \widehat{g} (\omega) |^2
      \,d\omega \, .
  %\end{split}
  \label{eq:TranslationDerivativeReducedToModulation}
\end{equation}
Thus, using that $\big(\omega \mapsto \omega \cdot \widehat{g}(\omega) \big) \in L^2$,
we can conclude from our calculations
in \eqref{eq:ModulationDerivative} that
\[
  0 \leq \left(
           \frac{\| T_a g - g + a \, g' \|_{\LTwoIndex}}
                {\sqrt{a^2 + b^2}}
         \right)^2
  %\lim_{a\to 0}
    \leq \frac{\|T_a g - g + a g'\|_{\LTwoIndex}^2}{a^2}
    \xrightarrow[a \to 0]{} 0 \, .
  %= 0 \, .
\]
%since $\omega \cdot \widehat{g} \in L^2$.
Finally, using the estimates $|e^{2 \pi i b x} - 1| \leq 2 \pi |b x|$
and ${|e^{2 \pi i b x} - 1| \leq 2}$, we can treat the last summand in \eqref{ha} as follows:
\[
  %a^2 \int
  %      \big|e^{2\pi ibx} - 1\big|^2 |g'(x)|^2
  %    \,dx
  \big\| a \cdot [1 - e^{2 \pi i b x}] \cdot g'(x) \big\|_{L^2}^2
  \le 4 \pi^2 a^2 |b|
      \int_{|x| \leq |b|^{-\frac{1}{2}}}
        |g'(x)|^2
      \, d x %\\
      %\|g'\|_2^2
      + 4 a^2 \!\!
        \int_{|x| \ge |b|^{-\frac{1}{2}}}
          |g'(x)|^2
        \, d x .
\]
Hence,
\[
  \left(
    \frac{\| a \cdot g'(x) \cdot (1 - e^{2 \pi i b x}) \|_{\LTwoIndex}}
         {\sqrt{a^2 + b^2}}
  \right)^2
  %& = \frac{a^2}{a^2+b^2}\int\left|e^{2\pi ibx} - 1\right|^2|g'(x)|^2\,dx \\
  \,\le\, 4\pi^2 |b| \cdot \|g'\|_{\LTwoIndex}^2
            + 4 \int_{|x|\ge 1/\sqrt{|b|}} \,\,
                  |g'(x)|^2
                \,dx \, ,
\]
which tends to zero as $(a,b) \to (0,0)$, again as a consequence of the dominated convergence
theorem and $g' \in L^2$.
By recalling \eqref{ha}, we thus see that \eqref{eq:FrequencyMapDerivativeTargetEstimate} holds.

\medskip
%\rsout{(b)} 
Assume now that $g \in \HH^2(\R)$.
In order to prove \eqref{e:H2}, we recall Equations~\eqref{huhu} and \eqref{e:sinc} to see that
\begin{align*}
  \int\left|e^{2\pi ibx} - 1 - 2\pi ibx\right|^2|g(x)|^2\,dx
  &= 4\pi^2 b^2
     \int
       x^2
       \left| \sinc(bx) - e^{-\pi i b x} \right|^2
       |g(x)|^2
     \,dx\\
  &\le 4 \pi^4 b^4 \int
                     x^4|g(x)|^2
                   \,dx
    =  4 \pi^4 b^4 \|X^2 g\|_{\LTwoIndex}^2 \, .
\end{align*}
Likewise, we use Equations~\eqref{eq:TranslationDerivativeReducedToModulation},
\eqref{e:sinc}, and \eqref{eq:ExpIDerivativeSinc} to obtain
\[
  \|T_a g - g + a g'\|_{\LTwoIndex}^2
  \,\le\, 4 \pi^4 a^4 \int \omega^4 \, | \widehat{g} (\omega)|^2 \,d\omega
  =       4 \pi^4 a^4 \| \omega^2 \widehat{g} (\omega) \|_{\LTwoIndex}^2 \, .
\]
Furthermore,
\[
  a^2 \int
        \left| e^{2\pi ibx} - 1 \right|^2
        | g'(x) |^2
      \,dx
  \,\le\, 4 \pi^2 a^2 b^2 \int x^2 |g'(x)|^2 \, d x
  =       4 \pi^2 a^2 b^2 \|Xg'\|_{\LTwoIndex}^2 \, .
\]
Thus, Equation~\eqref{ha},
combined with the elementary estimate $|ab| \leq \frac{1}{2} (a^2 + b^2)$, shows that
\[
  \|S_g(a,b) - g - (-ag' + 2\pi i b X g)\|_{\LTwoIndex}
  \,\le\, \tfrac 2 3\,C_g\cdot (a^2+b^2+|a||b|)
  \,\le\,C_g \cdot \|(a,b)\|_2^2 \, ,
\]
and the lemma is proved.
\end{proof}

\begin{cor}\label{c:derivative}
For any $g\in\HH^1(\R)$, the map $S_g$ is continuously \braces{Fr\'echet}-differentiable with
\[
  S_g'(\mu)\smallvek{a}{b}
  = -a \pi(\mu) g' + 2\pi i b X \pi(\mu) g,
  \qquad \mu \in \R^2.
\]
\end{cor}

\begin{proof}
Let $\mu,\la\in\R^2$, $\mu = (u,\eta)$, $\la = (a,b)$.
Then
\begin{align*}
  \pi(\mu+\la)g - \pi(\mu)g
  &= M_{\eta+b} T_{u+a} \, g - M_\eta T_u \, g
   = M_\eta \left( \pi(a,b) - I \right) T_u \, g \, .
\end{align*}
Now, since $T_u \, g \in \HH^1 (\R)$ with $(T_u \, g)' = T_u \, g'$,
Lemma~\ref{l:differentiable} shows that
\[
  \left(\pi(a,b) - I\right)T_u \, g
  = -a T_u \, g' + 2\pi i b X T_u \, g + \veps(a,b) \, ,
\]
where $\eps(a,b) = \eps_u(a,b) \in L^2(\R)$ satisfies
$\lim_{(a,b)\to (0,0)} \tfrac{\| \veps(a,b) \|_{L^2}}{\|(a,b)\|_2} = 0$.
Thus,
\begin{align*}
  \pi(\mu+\la)g - \pi(\mu)g
  &= -a M_\eta T_u g' + 2\pi i b M_\eta(X T_u \, g) + M_\eta \, \veps(a,b)\\
  &= -a\pi(\mu)g' + 2\pi i b X \pi(\mu) g + \widetilde{\veps}(a,b) \, ,
\end{align*}
where $\widetilde{\veps} := M_\eta \, \veps$.
As $\|\widetilde{\veps}(a,b)\|_{\LTwoIndex} = \|\veps(a,b)\|_{\LTwoIndex}$, the claim is proved.
\end{proof}

\section{Proof of Theorem~\ref{t:quanty}}
\label{s:quanty_proof}

As mentioned in the introduction, an upper bound in \eqref{e:main} is not difficult to achieve.
It even holds without the additional assumptions of $\Lambda$
having rational density or $(g,\Lambda)$ forming a Riesz sequence.

\begin{prop}\label{p:upper}
%%% Let $g\in\HH^1(\R)$, $\Lambda$ a lattice in $\R^2$, and $\mu\in\R^2$. Then
Let $g\in\HH^1(\R)$ and let $\Lambda$ be a lattice in $\R^2$.
Then
\[
  \dist\big(\pi(\mu)g,\calG(g,\Lambda)\big)
  \,\le\, \sqrt{\|g'\|_{\LTwoIndex}^2 + \|2\pi i X g\|_{\LTwoIndex}^2}
          \cdot \dist(\mu,\Lambda)
          \quad \text{for all} \;\; \mu\in\R^2 .
\]
\end{prop}

\begin{proof}
Let $\la\in\Lambda$ be a closest point (in Euclidean distance) in $\Lambda$ to $\mu$.
Then $(0,0)$ is a closest point in $\Lambda$ to $z:=\mu-\la$,
and thus $\dist(\mu,\Lambda) = \dist(z,\Lambda) = \|z\|_2$.
By Lemma \ref{l:commute} we have
\[
  \dist \big( \pi(\mu)g, \calG(g,\Lambda) \big)
  = \dist\big( \pi(z)g, \calG(g,\Lambda) \big)
  %= \|(\Id - \bP)(\pi(z) - g)\|_2
  \le \| \pi(z)g - g \|_{\LTwoIndex} \, .
\]
Now, if $z = (u,\eta)$, then Plancherel's theorem shows that
\begin{align*}
  %\dist\big(\pi(\mu)g,\calG(g,\Lambda)\big)
  %&= \dist\big(\pi(z)g,\calG(g,\Lambda)\big) = \|(\Id - \bP)(\pi(z) - g)\|_2\\ &\le
  \| \pi(z)g - g \|_{\LTwoIndex}
  &\le \|
         (\pi(u,\eta) - \pi(0,\eta))g
       \|_{\LTwoIndex}
       + \| \pi(0,\eta)g - g \|_{\LTwoIndex} \\
  &= \|M_\eta (T_u - I) g \|_{\LTwoIndex} + \| (M_\eta - I) g \|_{\LTwoIndex} \\
  &= \| (M_{-u} - I) \widehat g \|_{\LTwoIndex} + \| (M_\eta - I) g \|_{\LTwoIndex} \, .
\end{align*}
%where $M_af(x) = e^{2\pi iax}f(x)$ and $T_af(x) = f(x-a)$.
Next, recall that $|e^{2 i x} - 1| = |e^{i x} - e^{- i x}| = 2 |\sin(x)| \leq 2 |x|$
for $x \in \R$.
Using this estimate, we observe for $f\in L^2(\R)$ with $\widehat{f} \in H^1(\R)$
and $\alpha\in\R$ that
\[
  \|(M_\alpha - I)f\|_{\LTwoIndex}^2
  =   \int
         \Big|
           e^{2\pi i\alpha x} - 1\Big
         |^2 \,
         |f(x)|^2
       \, dx
   %=   4 \int
   %        \left|\sin(\pi\alpha x)\right|^2
   %        |f(x)|^2
   %      \, dx \\
  \le 4 \pi^2 \alpha^2 \int x^2 \, |f(x)|^2 \, dx \, ,
\]
that is, $\| (M_\alpha - I)f \|_{\LTwoIndex}
          \le 2\pi |\alpha| \cdot \| X f \|_{\LTwoIndex}
          =   |\alpha| \cdot \| 2\pi i X f \|_{\LTwoIndex}$.
%Hence, \red{if we define $\Omega \widehat{f} (\omega) := \omega\cdot \wh f(\omega)$
Hence, if we define $\Omega \widehat{f} (\omega) := \omega \, \wh f(\omega)$
for $\omega\in\R$ and $f\in\HH^1(\R)$, we find that
\[
  \dist\big(\pi(\mu)g,\calG(g,\Lambda)\big)
  \, \le \, |u| \cdot \| 2 \pi i \, \Omega \, \widehat{g} \|_{\LTwoIndex}
          + |\eta| \cdot \|2 \pi i X g\|_{\LTwoIndex} \, .
\]
Since $2\pi i\,\Omega \, \widehat{g} = \Fourier [g']$, Plancherel's theorem
and the Cauchy-Schwarz inequality yield the claim.
\end{proof}

\begin{rem}
If $\Lambda = A\Z^2$ with $A\in \GL(2,\R)$,
then the maximal distance of a point $\mu\in\R^2$ to the lattice $\Lambda$
is bounded above by $2^{-1/2}\|A\|_{\mathrm{op}}$.
Therefore, for each time-frequency shift $\pi(\mu)g$ of $g$ we have that
\[
  \dist \big(\pi(\mu)g,\calG(g,\Lambda)\big)
  \leq \sqrt{\frac{\| g' \|_{\LTwoIndex}^2 + \| 2\pi i X g \|_{\LTwoIndex}^2}{2}} \,
       \|A\|_{\mathrm{op}} \, .
\]
In other words, the better $g$ is localized in both time and frequency,
the closer the time-frequency shifts of $g$ scatter around $\calG(g,\Lambda)$.
However, due to the uncertainty principle (see e.g., \cite[Theorem 2.2.1]{g}),
the constant in the above inequality is easily seen to satisfy
$\sqrt{\frac{\| g' \|_{\LTwoIndex}^2 + \| 2\pi i X g \|_{\LTwoIndex}^2}{2}} \geq \sqrt{\pi}$.
\end{rem}

In the proof of the next proposition we consider matrix-valued
ordinary differential equations (ODEs) of the form
\begin{equation}\label{e:ODE}
  X'(t) = X(t)M(t),
\end{equation}
where $X : \R \to \C^{m {\times} n}$ and where $M : \R \to \C^{n{\times} n}$
has locally integrable entries.
A \emph{solution} of this ODE is a matrix function $X : \R \to \C^{m{\times} n}$
with (locally) absolutely continuous entries for which $X'(t) = X(t)M(t)$ holds for a.e.\ $t\in\R$.

\begin{lem}\label{l:unique}
  If $X_1$ and $X_2$ are two solutions to the ODE \eqref{e:ODE} such that ${X_1(0) = X_2(0)}$,
  then $X_1 (t) = X_2 (t)$ for \emph{all} $t \in \R$.
\end{lem}

\begin{proof}
  Since the classical ODE theory deals with continuously differentiable solutions to
  equations with coefficient functions fulfilling a Lipschitz condition, we cannot
  quite apply that theory.
  As we will see, however, the same proof idea still works.

  Indeed, since $X := X_1 - X_2$ is a solution to the ODE $X' = X \cdot M$ with $X(0) = 0$,
  it suffices to show that any such function satisfies $X \equiv 0$.
  Since $X$ is continuous, the set $\Gamma := \{t \in \R \colon X(t) = 0\}$ is closed.
  Since $\R$ is connected and since $0 \in \Gamma \neq \emptyset$,
  it is therefore enough to show that $\Gamma$ is also open.

  Thus, let $x_0 \in \Gamma$ be fixed but arbitrary.
  Since $M$ is locally integrable, there is some $\eps > 0$ such that
  $\int_{x_0 - \eps}^{x_0 + \eps} \|M(t)\|_{\mathrm{op}} \, d t \leq \frac{1}{2}$.
  Now, set $I := [x_0 - \eps, x_0 + \eps]$, and denote by $\boldX := C(I; \C^{m {\times} n})$
  the space of all continuous functions $f : I \to \C^{m {\times} n}$, equipped with the norm
  $\|f\|_{\boldX} := \sup_{t \in I} \|f(t)\|_{\mathrm{op}}$.
  It is not hard to see that $\boldX$ is a Banach space.
  Furthermore, define the linear operator
  \[
    T : \boldX \to \boldX, f \mapsto T f
    \quad \text{where} \quad
    (T f)(t) := \int_{x_0}^t f(s) \, M(s) \, d s \quad \text{for } t \in I \, .
  \]
  Note that indeed $T f \in \boldX$ if $f \in \boldX$, since $M$ is locally integrable, so that
  $f \cdot M$ is integrable on $I$.
  Next, observe
  \[
    \| Tf (t) \|_{\mathrm{op}}
    \leq \Big|
           \int_{x_0}^t \|f(s)\|_{\mathrm{op}} \cdot \|M(s)\|_{\mathrm{op}} \, d s
         \Big|
    \leq \|f\|_{\boldX} \cdot \int_I \|M(s)\|_{\mathrm{op}} \, d s
    \leq \frac{1}{2} \, \|f\|_{\boldX} \, ,
  \]
  and hence $\|T\|_{\boldX \to \boldX} \leq \frac{1}{2} < 1$.
  From this, it follows using a Neumann series argument that $\id - T : \boldX \to \boldX$
  is invertible.

  Finally, since $X(x_0) = 0$ and $X'(t) = X(t) M(t)$, we have
  \[
    X(t)
    = X(t) - X(x_0)
    = \int_{x_0}^t X'(s) \, d s
    = \int_{x_0}^t X(s) M(s) \, d s
    = \big( T [X|_{I}] \big)(t)
  \]
  for all $t \in I$, which means that $f := X|_{I}$ satisfies $(\id - T) f = 0$.
  Hence $f = 0$, which means that $X \equiv 0$ on $(x_0 - \eps, x_0 + \eps)$.
  Thus, $(x_0 - \eps, x_0 + \eps) \subset \Gamma$, so that $\Gamma$ is open.
\end{proof}

%Indeed, also $Y := X_1-X_2$ is a solution to \eqref{e:ODE}
%and its rows $y_j^T$ satisfy $y_j' = M^Ty_j$ and $y_j(0) = 0$.
%Hence, by well known ODE theory, $Y = 0$.

The following proposition can be seen as a weak Balian-Low-type theorem for subspaces.
For a comparison with related results, see Remark~\ref{r:weak} below.

\begin{prop}\label{p:weak}
Let $g\in\HH^1(\R)$ and let $\Lambda\subset\R^2$ be a lattice of rational density
such that $(g,\Lambda)$ is a Riesz basis for its closed linear span $\calG(g,\Lambda)$.
Then
%%%  Let $\Lambda\subset\R^2$ be a lattice of rational density and let $g\in\HH^1(\R)$ be
%%%  such that $(g,\Lambda)$ is a Riesz basis for its closed linear span $\calG(g,\Lambda)$. Then
%%%  Then for any $(a,b)\in\R^2\backslash\{(0,0)\}$ we have that
\[
  -ag' + 2\pi i b X g \notin \calG(g,\Lambda)
  \quad \text{for all} \;\; (a,b)\in\R^2\backslash\{(0,0)\} .
\]
\end{prop}

\begin{proof}
%%% Let $(a,b) \in \R^2$ such that $\gamma := -a g' + 2\pi i b X g \in \calG(g, \Lambda)$,
%%% and assume towards a contradiction that $(a,b) \neq (0,0)$.
Let us assume towards a contradiction that $\gamma := -a g' + 2\pi i b X g \in \calG(g, \Lambda)$
for some $(a,b)\in\R^2\backslash\{(0,0)\}$.
We divide the proof into five steps.

\medskip{}

\textbf{Step 1:}
In the first four steps of the proof,
we only consider separable lattices of the form
$\Lambda = \tfrac{1}{Q} \Z {\times} P \Z$ for certain $P, Q \in \N$.

Let $G := Z g \in L^2_{\mathrm{loc}} (\R^2)$ denote the Zak transform of $g$,
and recall from Lemma~\ref{l:matrix} the definition of the function
$A_g \in L^2_{\mathrm{loc}} (\R^2; \C^{P {\times} Q})$ given by
\[
  A_g (x,\omega) = P^{-1/2}
                   \cdot \big(
                           G (x + \tfrac{k}{P} - \tfrac{\ell}{Q}, \omega)
                         \big)_{k,\ell=0}^{P-1, Q-1} \, .
\]
Since $g \in \HH^1(\R)$, Lemma~\ref{l:gZg} shows that $G \in H^{1}_{\mathrm{loc}} (\R^2)$,
so that all component functions of $A_g$ are in $H^1_{\mathrm{loc}} (\R^2)$ as well.
%Furthermore, for each $(x,\omega) \in \R^2 \backslash N_3$, the function
%$t \mapsto A(x + t a, \omega + t b)$ is continuous, and each entry is in $H^1_{\mathrm{loc}}(\R)$.
In this step, we show that $A_g$ satisfies a certain differential equation;
see Equation~\eqref{eq:fundament} below.

%Since $(g,\Lambda)$ is a Riesz sequence, Lemma~\ref{l:matrix}
%shows that $\essinf_{z \in \R^2} \sigma_0 (A_g (z)) > 0$, which means that
%$(A^\ast A) (x,\omega)$ is invertible for almost all $(x,\omega) \in \R^2$,
%say for all $(x,\omega) \in \R^2 \backslash N_1$.

Since $\gamma \in \calG(g, \Lambda)$ and $\Lambda = \tfrac{1}{Q} \Z {\times} P \Z$,
Lemma~\ref{lem:GaborSpaceInvariance} shows $\pi(\tfrac{L}{Q}, 0) \gamma \in \calG(g, \Lambda)$
for each $L\in\{0,\ldots,Q-1\}$.
This means that for each $L \in \{0,\dots,Q-1\}$ there is a sequence
$(c_{m,n}^{(L)})_{m,n \in \Z} \in \ell^2(\Z^2)$ such that
\[
  %\pi(\tfrac{L}{Q}, 0)\left(-a g' + 2\pi i b x g\right)
  \pi(\tfrac{L}{Q}, 0) \gamma
  = \sum_{m, n\in\Z}
      c_{m,n}^{(L)} \, \pi(\tfrac{n}{Q}, Pm)g
  = \sum_{\ell=0}^{Q-1}
      \sum_{m, s \in \Z}
        c_{m, sQ + \ell}^{(L)} \,\,
        \pi(s + \tfrac{\ell}{Q}, Pm) g \, .
\]
By using the properties (a)--(c) of the Zak transform
listed below Equation~\eqref{eq:ZakTransformDefinition}, this implies
for each $L \in \{0, \dots, Q-1\}$ that
\begin{align*}
  (Z \gamma) \big(x \!-\! \tfrac{L}{Q}, \omega \big)
  & \!=\! Z \big[ \pi(\tfrac{L}{Q}, 0) \gamma \big] (x,\omega) %\\
    \!=\! \sum_{\ell=0}^{Q-1}
        \sum_{m, s \in \Z} \!
          c_{m, sQ + \ell}^{(L)} \,
          Z \big[ \pi(s \!+\! \tfrac{\ell}{Q}, P m) g \big] (x,\omega) \\
  & \!=\! - \sum_{\ell = 0}^{Q-1}
          f_\ell^{(L)} (x,\omega) \cdot G \big( x - \tfrac{\ell}{Q}, \omega \big)
\end{align*}
where $f_\ell^{(L)}(x,\omega)
       := - \sum_{m, s \in \Z}
              c_{m,sQ+\ell}^{(L)} e^{2\pi i(P m x - s \omega)}$.
Note that each $f_\ell^{(L)}$ is locally square-integrable on $\R^2$
and $(\tfrac{1}{P}, 1)$-periodic.
%with sequences $(c_{mn}^{(L)})_{m,n} \in \ell^2(\Z^2)$, $L=0,\ldots,Q-1$.
%Denote the Zak transform of $g$ by $G$.

Now, recall from Lemma~\ref{l:gZg} that
$(\partial_2 G) (x,\omega) \!=\! 2\pi i \big( x \, G(x,\omega) - Z(Xg)(x,\omega) \big)$
and $\partial_1 G = Z g'$.
Therefore,
\begin{align*}
  (Z \gamma)(x,\omega)
  &= Z[- a g' + 2 \pi i b X g] (x,\omega) \\
  &= - a \cdot \partial_1 G (x,\omega)
    + 2 \pi i b \cdot x\, G (x,\omega)
    - b \cdot \partial_2 G (x,\omega).
\end{align*}
%Thus, taking into account the properties (a)--(c) of the Zak transform on Page~\pageref{p:Zak},
Thus, we arrive at
\begin{align*}
  & a \, \partial_1 G (x - \tfrac{L}{Q}, \omega) + b \, \partial_2 G(x - \tfrac{L}{Q}, \omega) \\
  &= 2\pi i b \, \big( x - \tfrac{L}{Q} \big) \, G(x -\tfrac{L}{Q}, \omega)
     + \sum_{\ell=0}^{Q-1}
        f_\ell^{(L)}(x,\omega) \,
        G (x - \tfrac{\ell}{Q}, \omega).
\end{align*}
Denoting by $e_0, \dots, e_{Q-1}$ the standard basis vectors of $\C^Q$,
plugging $x + \frac{k}{P}$ instead of $x$ into the preceding displayed equation,
and recalling that $f_\ell^{(L)}$ is $(\frac{1}{P}, 1)$-periodic,
we obtain for each $L \in \{0,\dots,Q-1\}$ that
\begin{align*}
  & a \, \partial_1 A_g (x,\omega) \, e_L + b \, \partial_2 A_g (x,\omega) \, e_L \\
  & = 2\pi i b \big[
                \big( x - \tfrac{L}{Q} \big) A_g (x,\omega) e_L
                + D_P A_g (x,\omega) \, e_L
               \big]
      + A_g (x,\omega) \, f^{(L)} (x,\omega) \, ,
\end{align*}
where $f^{(L)} := (f_\ell^{(L)})_{\ell=0}^{Q-1}$ and $D_P := \operatorname{diag}(k/P)_{k=0}^{P-1}$.
This leads to
\begin{align*}
  a \, \partial_1 A_g (x,\omega) + b \, \partial_2 A_g (x,\omega)
  &= 2\pi i b [(x A_g(x,\omega) + D_P A_g (x,\omega)] \\
  &\quad + A_g(x,\omega) (F(x,\omega) - 2\pi i b D_Q) \, ,
\end{align*}
where $D_Q := \operatorname{diag}(L/Q)_{L=0}^{Q-1}$
and $F := [f^{(0)} \,|\, \ldots \,|\, f^{(Q-1)}] \in L_{\mathrm{loc}}^2 (\R^2; \C^{Q {\times} Q})$.
As a consequence of Fubini's theorem (and since $(a,b) \neq (0,0)$), there is
a null-set $N_0 \subset \R^2$ such that
$(t \mapsto F(x + t a, \omega + t b)) \in L_{\mathrm{loc}}^2 (\R; \C^{Q {\times} Q})$
for all $(x,\omega) \in \R^2 \backslash N_0$.

Note that the preceding displayed equation holds for almost all $(x,\omega) \in \R^2$.
Therefore, if we let $v_t := v + t \, (a,b)$ for $v \in \R^2$ and $t \in \R$,
then Lemma~\ref{lem:FubiniGeneralDirection} yields a null-set $N_1 \subset \R^2$ such that
if $v = (x,\omega) \in \R^2 \backslash N_1$, then
\begin{equation}
  \begin{split}
    & a \, (\partial_1 A_g) (v_t) + b \, (\partial_2 A_g) (v_t) \\
    & = 2 \pi i b \big[
                    (x + t a) \, A_g (v_t) + D_P A_g (v_t)
                  \big]
        + A_g (v_t) \cdot \big( F(v_t) - 2 \pi i b \, D_Q \big) \\
    & = 2 \pi i b \, D_P A_g (v_t) + A_g (v_t) [2 \pi i b (x + ta) + F(v_t) - 2 \pi i b \, D_Q ] \\
    & = 2 \pi i b \, D_P A_g (v_t) + A_g (v_t) W_v (t)
  \end{split}
  \label{eq:fundament}
\end{equation}
for almost all $t \in \R$.
In the last step we introduced the matrix
\[
%  M_v (t)
  W_v (t) := 2 \pi i b (x + ta) I_Q + F(v_t) - 2 \pi i b \, D_Q \in \C^{Q {\times} Q},
  \quad t \in \R ,
\]
where $I_Q$ denotes the $Q$-dimensional identity matrix.
Note $W_v \in L_{\mathrm{loc}}^2 (\R; \C^{Q {\times} Q})$ for all $v \in \R^2 \backslash N_0$.

\medskip{}

\textbf{Step 2:}
In this step, we construct a particularly nice representative of $G = Z g$.

Recall from Step~1 that $G \in H_{\mathrm{loc}}^1(\R)$.
Next, define ${\varrho := (a,b) \in \R^2 \backslash \{0\}}$,
and choose $\theta \in \R^2$ with $\|\theta\|_2 = 1$ and $\theta \perp \varrho$.
Define ${T : \R^2 \to \R^2, (t,s) \mapsto t \varrho + s \theta}$,
and note that $T$ is linear and bijective, so that the same holds also for $T^{-1}$.
In particular, $T$ and $T^{-1}$ are Lipschitz continuous, and thus map null-sets to null-sets.
Furthermore, since $T$ and $T^{-1}$ are Lipschitz continuous,
the change-of-variables formula for Sobolev functions
(see for instance \cite[Theorem 2.2.2]{ZiemerWeaklyDifferentiable}) shows that
$\widetilde{G} := G \circ T \in H^1_{\mathrm{loc}} (\R^2)$, and that
\begin{align}
    D \widetilde{G} (t,s)
    & = D G (T (t,s)) \cdot D T (t, s)
      = \big( [\partial_1 G] (t \varrho + s \theta), [\partial_2 G] (t \varrho + s \theta) \big)
        \cdot \left( \begin{matrix} a & \theta_1 \\ b & \theta_2 \end{matrix} \right)
        \nonumber \\
    & = \big(
          a \, [\partial_1 G] (t \varrho + s \theta) + b \, [\partial_2 G] (t \varrho + s \theta),
          \ast
        \big)
  \label{eq:LinearTransformedDerivative}
\end{align}
for almost all $(t,s) \in \R^2$.
By Lemma~\ref{l:andrei}, there is a null-set $N_2 \subset \R$ such that
for all $s \in \R \backslash N_2$, Equation~\eqref{eq:LinearTransformedDerivative}
holds for almost all $t \in \R$.

Lemma~\ref{l:slices} yields a null-set $N_3 \subset \R$,
and a (pointwise defined) Borel function $\widetilde{G}_0 : \R^2 \to \C$
such that $\widetilde{G}_0 = \widetilde{G}$ almost everywhere, and such that
%there is a null-set $N_1 \subset \R$, such that
for all $s \in \R \backslash N_3$, the function
$t \mapsto \widetilde{G}_0 (t, s)$ is continuous and in $H^1_{\mathrm{loc}} (\R)$ with
$\frac{d}{dt} \widetilde{G}_0 (t,s) = (\partial_1 \widetilde{G}) (t,s)$ almost everywhere.
In view of Equation~\eqref{eq:LinearTransformedDerivative},
we thus see for all $s \in \R \backslash (N_2 \cup N_3)$ that
\[
  \frac{d}{dt} \widetilde{G}_0 (t,s)
  %= (\partial_1 \widetilde{G}) (t,s)
  = a \, [\partial_1 G] (t \varrho + s \theta) + b \, [\partial_2 G] (t \varrho + s \theta)
  \quad \text{for almost all } t \in \R .
\]

Note that since $\widetilde{G}_0 = \widetilde{G} = G \circ T$ almost everywhere and since
$T$ and $T^{-1}$ map null-sets to null-sets, we have $G = \widetilde{G}_0 \circ T^{-1} =: G_0$
almost everywhere.
By Lemma~\ref{lem:FubiniGeneralDirection}, there is thus a null-set $N_4 \subset \R^2$
such that
\begin{equation}
  \forall \, (x,\omega) \in \R^2 \backslash N_4 \colon \quad
  G_0(x + ta, \omega + tb) = G(x+ta, \omega + tb)
  \text{ for a.e. } t \in \R \, .
  \label{eq:GIsGNaughtOnSlices}
\end{equation}
%if $(x,\omega) \in \R^2 \backslash N_4$, then
%$G_0(x + ta, \omega + tb) = G(x+ta, \omega + tb)$ for almost all $t \in \R$.
%In particular, the weak derivatives $\partial_1 G$ and $\partial_2 G$ are also
%weak derivatives for $G_0$.
%\emph{In the remainder of the proof, we will thus use the representative $G_0$ of $G$.}

Since $T$ is Lipschitz continuous,
the set $N_5 := T (\R {\times} (N_2 \cup N_3)) \subset \R^2$ is a null-set.
For any $(x,\omega) \in \R^2 \backslash N_5$, we have
$(x, \omega) = T (t_0, s_0) = t_0 \varrho + s_0 \theta$
for certain $(t_0, s_0) \in \R {\times} (\R \backslash (N_2 \cup N_3))$.
By the properties from above, this means that the map
\begin{align*}
  \R \to     \C, \quad
  t  \mapsto G_0 (x + t a, \omega + t b)
           &= G_0 \big( (x,\omega) + t \varrho \big) \\
           &= G_0 \big( (t+t_0) \varrho + s_0 \theta \big) %\\
            = \widetilde{G}_0 (t+t_0, s_0)
\end{align*}
is continuous and in $H^1_{\mathrm{loc}} (\R)$ with derivative
\begin{align}
  \frac{d}{dt} G_0 (x + t a, \omega + t b)
  & = \frac{d}{dt} \widetilde{G}_0 (t + t_0, s_0) \nonumber \\
  & = a \, [\partial_1 G] \big( (t+t_0) \varrho + s_0 \theta \big)
      + b \, [\partial_2 G] \big( (t+t_0) \varrho + s_0 \theta \big) \nonumber \\
  & = a \, [\partial_1 G] (x + t a, \omega + t b) + b \, [\partial_2 G] (x + t a, \omega + t b)
  \label{eq:NiceRepresentativeDerivative}
\end{align}
for almost all $t \in \R$, for each fixed $(x,\omega) \in \R^2 \backslash N_5$.

Finally, let
\(
  N_6
  := \bigcup_{k,\ell \in \Z}
       \big( (N_4 \cup N_5) + (\tfrac{\ell}{Q} - \tfrac{k}{P}, 0) \big) \subset \R^2
\),
which is a null-set.
If $(x,\omega) \in \R^2 \backslash N_6$,
then $\big( x + \tfrac{k}{P} - \tfrac{\ell}{Q}, \omega \big) \in \R^2 \backslash (N_4 \cup N_5)$
for all $k,\ell \in \Z$.

\medskip{}

\textbf{Step 3:} In this step, we use the ``nice'' representative $G_0$ of $G$ to construct
for almost all $v = (x,\omega) \in \R^2$ two locally absolutely continuous functions
${R_v : \R \to \C^{P {\times} Q}}$ and $L_v : \R \to \C^{P {\times} Q}$ which satisfy
the differential equations $R_v '(t) = R_v (t) W_v (t)$ and $L_v '(t) = L_v (t) W_v (t)$
for almost all $t \in \R$, for the matrix function
$W_v \in L_{\mathrm{loc}}^2( \R ; \C^{Q {\times} Q})$ defined in Step~1.
We then use this differential equation to deduce $R_v = L_v$.
In Step~4 we will finally employ this identity to complete the proof for the case
$\Lambda = \tfrac{1}{Q}\Z {\times} P \Z$.

First, define
\[
  A : \quad
  \R^2 \to \C^{P {\times} Q}, \quad
  (x,\omega) \mapsto P^{-1/2} \cdot \big(
                                      G_0 (x + \tfrac{k}{P} - \tfrac{\ell}{Q}, \omega)
                                    \big)_{k,\ell = 0}^{P-1,Q-1} \, ,
\]
noting $A =\! A_g$ almost everywhere.
Next, note for ${v = (x,\omega) \in \R^2 \backslash (N_0 \!\cup\! N_1 \!\cup\! N_6)}$ that
$(x + \tfrac{k}{P} - \tfrac{\ell}{Q}, \omega) \in \R^2 \backslash (N_4 \cup N_5)$
for all $k,\ell \in \Z$, so that Equations~\eqref{eq:NiceRepresentativeDerivative},
\eqref{eq:fundament}, and \eqref{eq:GIsGNaughtOnSlices} show that
the function ${E_v : \R \to \C^{P {\times} Q}, t \mapsto A(v_t) = A( x + t a, \omega + t b )}$
is locally absolutely continuous and satisfies
  \begin{align}
    %& P^{1/2} \, E_v ' (t) \\
      E_v ' (t) %\nonumber \\
%%%   & \lefteqn{E_v ' (t) } \nonumber \\
    & = P^{-1/2}
        \Big(
          a \, (\partial_1 G) \big( x + \tfrac{k}{P} - \tfrac{\ell}{Q} + t a, \omega + t b \big)
          \nonumber \\
          & \qquad \qquad \qquad
          + b \, (\partial_2 G) \big( x + \tfrac{k}{P} - \tfrac{\ell}{Q} + t a, \omega + t b \big)
        \Big)_{k,\ell = 0}^{P-1,Q-1} \nonumber \\
    & = a \, (\partial_1 A_g) ( x + t a, \omega + t b)
        + b \, (\partial_2 A_g) (x + t a, \omega + t b) \nonumber \\
    & = a \, [\partial_1 A_g] (v_t) + b \, [\partial_2 A_g] (v_t) \nonumber \\
    %({\scriptstyle{}})
    & = 2 \pi i b \, D_P A_g (v_t) + A_g (v_t) W_v (t) \nonumber \\
    & = 2 \pi i b \, D_P A (v_t) + A (v_t) W_v (t)
      = 2 \pi i b \, D_P \, E_v (t) + E_v (t) W_v (t)
  \label{eq:EDerivative}
  \end{align}
for almost all $t \in \R$.

Next, Lemma~\ref{l:matrix} shows that $\essinf_{z \in \R^2} \sigma_0 (A_g (z)) > 0$,
since $(g,\Lambda)$ is a Riesz sequence.
Hence, we also have $\essinf_{z \in \R^2} \sigma_0 (A (z)) > 0$,
which means that $(A^\ast A) (x,\omega)$ is invertible for almost all $(x,\omega) \in \R^2$,
say for all $(x,\omega) \in \R^2 \backslash N_7$.

For $v = (x,\omega) \in \R^2 \backslash (N_0 \cup N_1 \cup N_6 \cup N_7)$, set
$C_v := A(v) \big( A^\ast (v) A (v) \big)^{-1} A^\ast (v)$ (so that $C_v \in \C^{P {\times} P}$)
and furthermore
\begin{align*}
  R_v & : \quad \R \to \C^{P {\times} Q}, \quad
          t \mapsto e^{-2 \pi i t b \, D_P} \,\, A(x + t a, \omega + t b)
            =       e^{-2 \pi i t b \, D_P} \,\, E_v (t) , \\
  L_v & : \quad \R \to \C^{P {\times} Q}, \quad
          t \mapsto C_v \cdot e^{- 2 \pi i t b \, D_P} \cdot A(x + ta, \omega + tb)
            =       C_v \cdot R_v (t) ,
\end{align*}
where as before $D_P = \diag (k / P)_{k = 0,\dots,P-1} \in \R^{P {\times} P}$.

Since $v = (x,\omega) \in \R^2 \backslash (N_0 \cup N_1 \cup N_6)$,
we see as a consequence of the product rule for Sobolev functions
(see for instance \cite[Section~4.25]{AltLinearFunctionalAnalysis})
and of Equation~\eqref{eq:EDerivative} that $R_v$ is locally absolutely continuous, with
%\begin{align*}
%  R_v ' (t)
%  & = - 2 \pi i b D_P R_v (t)
%    + e^{-2 \pi i t b \, D_P} E_v ' (t) \\
%  ({\scriptstyle{\text{Equation } \eqref{eq:EDerivative}
%                 \text{ and since } e^{-2 \pi i t b \, D_P} D_P = D_P \, e^{-2 \pi i t b \, D_P}}})
%  & = R_v (t) \, M_v (t)
%\end{align*}
%for almost all $t \in \R$.
\[
  R_v ' (t)
  = - 2 \pi i b D_P R_v (t)
    + e^{-2 \pi i t b \, D_P} E_v ' (t) \\
 = R_v (t) \, W_v (t)
    \quad \text{for almost all } t \in \R .
\]
where the last equality follows from Equation~\eqref{eq:EDerivative}
combined with the elementary identity ${e^{-2 \pi i t b \, D_P} D_P = D_P \, e^{-2 \pi i t b \, D_P}}$.
This easily implies that $L_v$ is locally absolutely continuous as well,
with $L_v ' (t) = C_v \, R_v ' (t) = C_v \, R_v (t) \, W_v (t) = L_v (t) \, W_v (t)$
for almost all $t \in \R$.
Finally, note that
\[
  L_v (0)
  = C_v \cdot E_v (0)
  = A(v) \big( A^\ast (v) A(v) \big)^{-1} A^\ast (v) A(v)
  = A(v)
  = R_v (0) \, .
\]
Therefore, Lemma~\ref{l:unique} shows $L_v (t) = R_v (t)$ for all
$v \in \R^2 \backslash (N_0 \cup N_1 \cup N_6 \cup N_7)$ and all $t \in \R$.

\medskip{}

\textbf{Step 4:} We complete the proof for the case $\Lambda = Q^{-1} \Z {\times} P \Z$.
To this end, let $t \in \R$ be arbitrary, and note that the matrix function
$H_{(-ta, -tb)}$ defined in Lemma~\ref{l:dist} satisfies
for \emph{almost} all $v = (x,\omega) \in \R^2$ that
\begin{align*}
& \lefteqn{   H_{(-ta, -tb)} (x,\omega) }\\
  & = P^{1/2}
      \cdot A_g (v) \cdot \big( A_g^\ast (v) A_g (v) \big)^{-1} A_g^\ast (v)
      \cdot e^{- 2 \pi i t b \, D_P}
      \cdot A_g (x + ta, \omega + t b) \\
  & = P^{1/2}
      \cdot C_v \cdot e^{-2 \pi i t b \, D_P} \cdot A(x + ta, \omega + t b) \\
  & = P^{1/2} \cdot L_v (t)
    = P^{1/2} \cdot R_v (t) .
\end{align*}
Hence,
\begin{align*}
  \|H_{(-ta, -tb)} (x,\omega) e_0\|_{\C^P}^2
  &  = P \, \|R_v (t) e_0\|_{\C^P}^2
    = P \, \|A(v_t) e_0\|_{\C^P}^2
    = P \, \|A_g (v_t) e_0\|_{\C^P}^2 \\
  & = \sum_{k=0}^{P-1} \big| G \big( x + ta + \tfrac{k}{P}, \omega + t b \big) \big|^2
  \quad \text{for a.e. } (x,\omega) \in \R^2 .
\end{align*}
By Lemma~\ref{l:dist} and by the quasi-periodicity of $G = Z g$
(which implies that $|G|$ is $(1,1)$-periodic), this implies that
\begin{align*}
  \dist^2\big(\pi(-ta,-tb)g,\calG(g,\Lambda)\big) % \\
  &= \|g\|_{\LTwoIndex}^2 \!- \!\! \int_0^1 \!\!
                   \int_0^{\frac{1}{P}} \,
                     \sum_{k=0}^{P-1}
                       \big|
                         G \big( x \!+\! t a \!+\! \tfrac{k}{P}, \omega \!+\! t b \big)
                       \big|^2
                   \,dx
                 \,d\omega \\
  &= \|g\|_{\LTwoIndex}^2 - \int_0^1
                   \int_0^{1}
                     |G(x+ta,\omega+tb)|^2
                   \,dx
                 \,d\omega\\
  &= \|g\|_{\LTwoIndex}^2 - \int_0^1
                   \int_0^{1}
                     |G(x,\omega)|^2
                   \,dx
                 \,d\omega
   = 0 \, .
\end{align*}
That is, $\pi(- t a, - t b) g \in \calG(g, \Lambda)$ for each $t \in \R$.
By Theorem \ref{t:clpp}, this means that $(-ta, -tb) \in \Lambda$ for every $t \in \R$.
%But this is only possible for $a = b = 0$.
Because of $(a,b) \neq (0,0)$ and since $\Lambda \subset \R^2$ is discrete,
this yields the desired contradiction.

\medskip{}

\textbf{Step 5:} Let $\Lambda\subset\R^2$ be an arbitrary lattice of rational density,
and assume again that $-a g' + 2\pi i b X g \in \calG(g,\Lambda)$ for some $a,b \in \R$.
Then there exists a matrix $B\in \GL(2,\R)$ with $\det B = 1$ and certain $P,Q \in \N$
such that ${\Lambda_1 := B\Lambda = Q^{-1}\Z{\times} P\Z}$.
With the symplectic operator $U_B$ (see \eqref{e:mp}), set $g_1 := U_B \, g$.
Then $(g_1,\Lambda_1)$ is a Riesz basis for $\calG(g_1,\Lambda_1) = U_B \, \calG(g,\Lambda)$ and,
as $\HH^1(\R)$ is invariant under symplectic operators
(see the discussion after Equation~\eqref{eq:MetaplecticMultiplicative}),
we have $g_1 \in \HH^1(\R)$.
For $f \in \HH^1(\R)$, let us set $T_f(x,\omega) := \rho(x,\omega)f$,
$x,\omega\in\R$, cf.~\eqref{e:rho}.
Using Corollary~\ref{c:derivative} we find that
\begin{align*}
&  \lefteqn{T_f'(x,\omega)\smallvek a b} \\
  &= e^{-\pi ix\omega}
     \left[
       \left(
         -\pi i\omega S_f(x,\omega) + \partial_1 S_f(x,\omega)
       \right)
       a
       + \left(
           -\pi ix S_f(x,\omega) + \partial_2 S_f(x,\omega)
         \right)
         b
     \right]\\
  &= e^{-\pi ix\omega}
     \left[
       -\pi i(a\omega+bx) S_f(x,\omega)
       -a \pi(x,\omega) f'
       + 2\pi ibX\pi(x,\omega) f
     \right].
\end{align*}
In particular,
\begin{equation}
  T_f'(0,0)\smallvek a b = -a f' + 2\pi i b X f .
  \label{eq:TDerivative}
\end{equation}
We have (see \eqref{e:mp})
\[
  U_B \, T_g(x,\omega)
  = \rho(B \smallvek{x}{\omega}) g_1
  = T_{g_1}(B \smallvek{x}{\omega}) \, .
\]
Differentiating this with respect to $(x,\omega)$
gives $U_B \, T_g'(x,\omega) = T_{g_1}' (B \smallvek{x}{\omega}) \circ B$.
Hence, by Equation~\eqref{eq:TDerivative}, we see that%Lemma \ref{l:differentiable},
\[
  U_B \, (- a g' + 2\pi i b X g)
  = U_B \, T_g'(0,0) \smallvek{a}{b}
  = T_{g_1}'(0,0) \big( B \smallvek{a}{b} \big)
  = -\alpha g_1' + 2\pi i \beta X g_1 \, ,
\]
where $\smallvek{\alpha}{\beta} = B \smallvek{a}{b}$.
That is, $-\alpha g_1' + 2\pi i \beta X g_1 \in U_B \calG (g,\Lambda) = \calG(g_1, \Lambda_1)$,
which, by the first part of this proof, implies that $\alpha = \beta = 0$ and thus $a = b = 0$.
\end{proof}

\begin{rem}\label{r:weak}
Proposition~\ref{p:weak} is closely related to the so-called weak subspace Balian-Low Theorem
(cf.~\cite[Thm.\ 8]{ghhk}) which states that if $g\in L^2(\R)$
and $\Lambda \subset \R^2$ is a lattice such that $(g,\Lambda)$ is a Riesz basis
for its closed linear span $\calG$, then at least one
of the distributions $g', Xg, \wt{g}', X \wt{g}$ is not contained in $\calG$,
where $\wt{g}$ denotes the dual window of $(g,\Lambda)$.
More precisely, Proposition~\ref{p:weak} implies that if $g', X g\in L^2(\R)$
and $\Lambda \subset \R^2$ is a lattice of rational density such that $(g,\Lambda)$
is a Riesz sequence (and hence also $\wt{g}', X \wt{g} \in L^2(\R)$ by Proposition~\ref{p:dual}),
then {\em none} of $g', Xg, \wt{g}', X \wt{g}$ is contained in $\calG$.
In fact, it even asserts that none of the real linear combinations of $i g'$ and $X g$
except $0$ can belong to $\calG$.
Similarly, none of the real linear combinations of $i\wt{g}'$ and $X \wt{g}$ except $0$
can belong to $\calG$.  
\end{rem}

We are now ready to prove Theorem~\ref{t:quanty}.

%\begin{thm}
%Let $\Lambda\subset\R^2$ be a lattice of rational density
%and let $g\in\HH^1(\R)$ such that $(g,\Lambda)$ is a Riesz basis
%for its closed linear span $\calG(g,\Lambda)$.
%Then there exists $c > 0$ such that for all $\mu\in\R^2$ we have
%$$
%\dist\big(\pi(\mu)g,\calG(g,\Lambda)\big)\,\ge\,c\cdot\dist(\mu,\Lambda).
%$$
%\end{thm}
\begin{proof}[Proof of Theorem \rmref{t:quanty}]
Let us denote by $\bP$ the orthogonal projection from $L^2(\R)$ onto $\calG := \calG(g,\Lambda)$.
Proposition \ref{p:weak} implies that the $\R$-linear mapping
\[
  \R^2 \to L^2 (\R), \;\;
  (a,b) \mapsto(\Id - \bP)(-a g' + 2\pi i b X g) \, ,
\]
with $L^2(\R)$ considered as an $\R$-linear space, is injective.
Since $\R^2$ is finite-dimensional, this implies
$\|(\Id - \bP)(-a g' + 2\pi i b X g)\|_{\LTwoIndex} \ge 2 \gamma\|(a,b)\|_2$
for some $\gamma > 0$ and all $(a,b) \in \R^2$.
On the other hand, Lemma~\ref{l:differentiable} gives a family of functions
$\{ \veps(a,b) \}_{(a,b) \in \R^2} \subset L^2(\R)$ such that
%%% that for each $(a,b) \in \R^2$ there exists $\veps(a,b) \in L^2(\R)$ such that}
\[
  \pi(a,b) g - g
  = -a g' + 2\pi i b X g + \veps(a,b)
  \quad \text{and} \quad
  \lim_{(a,b)\to (0,0)}
    \frac{\| \veps(a,b) \|_{L^2}}{\|(a,b)\|_2} = 0 \, .
\]
In particular, there exists some $\delta > 0$ such that
$\|\veps(a,b)\|_{\LTwoIndex} \le \gamma\|(a,b)\|_2$ for $\|(a,b)\|_2 < \delta$.
Combining these observations and the fact that $(\Id - \bP) g = 0$,
we see for $\|(a,b)\|_2 < \delta$ that
\begin{align*}
  2\gamma\|(a,b)\|_2
  &\le \| (\Id - \bP)(-a g' + 2\pi i b X g) \|_{\LTwoIndex}
   =   \big\| \big( \Id - \bP \big) \big(\pi(a,b) g - \veps(a,b) \big) \big\|_{\LTwoIndex} \\
  &\le \| (\Id - \bP) \pi(a,b) g \|_{\LTwoIndex}
       + \! \| \veps(a,b) \|_{\LTwoIndex}
    \le \dist(\pi(a,b) g ,\calG) + \gamma \, \|(a,b)\|_2 \, ,
\end{align*}
%%% where we used that $(\Id - \bP) g = 0$.
that is, $\dist(\pi(a,b)g, \calG) \ge \gamma \|(a,b)\|_2$ for $\|(a,b)\|_2 < \delta$.

Now, consider the compact set $R := \{\mu\in\R^2 : \|\mu\|_2 = \dist(\mu,\Lambda)\}$ and
denote by $B = B_\delta(0,0) \subset \R^2$ the open ball of radius $\delta > 0$ centered at $(0,0)$.
By possibly shrinking $\delta$, we may assume that $B\subset R$;
in fact, since $\Lambda$ is discrete, there is some $\delta_0 > 0$ such that
$\|\lambda \|_2 \geq 2 \delta_0$ for all $\lambda \in \Lambda \backslash \{0\}$.
We then have $B \subset R$ as soon as $0 < \delta \leq \delta_0$.

We will show that $\|(\Id - \bP)\pi(a,b)g\|_{\LTwoIndex} \ge \gamma'\|(a,b)\|_2$
for a suitable $\gamma ' > 0$ and all $(a,b)\in R \backslash B$.
Towards a contradiction, suppose that there is no such $\gamma' > 0$.
%%% Let us assume towards a contradiction that there does not exist $\gamma' > 0$
%$\|(\Id - \bP)\pi(a,b)g\|_{\LTwoIndex} \ge \gamma'\|(a,b)\|_2$
%for all $(a,b)\in R \backslash B$.
Then there exists a sequence $(\mu_n)_{n \in \N} \subset R \backslash B$
such that $(\Id - \bP)\pi(\mu_n)g\to 0$ as $n\to\infty$.
As $R\backslash B$ is compact, we may assume that $\mu_n\to\mu_0$
as $n\to\infty$ for some $\mu_0\in R\backslash B$.
But then, since $\mu\mapsto\pi(\mu)g$ is continuous,
it follows that $(\Id - \bP)\pi(\mu_0)g = 0$, that is, $\pi(\mu_0)g\in\calG$,
which by Theorem~\ref{t:clpp} is only possible if $\mu_0\in\Lambda$;
but this implies $\|\mu_0\|_2 = \dist(\mu_0, \Lambda) = 0$,
in contradiction to $\mu_0 \in R \backslash B$.

Hence, $\dist(\pi(a,b)g,\calG) = \|(\Id - \bP)\pi(a,b)g\|_{\LTwoIndex} \ge \gamma'\|(a,b)\|_2$
for some $\gamma ' > 0$ and all $(a,b)\in R \backslash B$.
As a consequence, we have with $\LowerBound := \min\{\gamma,\gamma'\} > 0$,
%%% Therefore, we see for all $\mu\in R$ that
\[
  \dist(\pi(\mu)g,\calG)
  \geq \LowerBound \cdot \|\mu\|_2
  =    \LowerBound \cdot \dist(\mu,\Lambda)
  \quad \text{for all} \;\; \mu\in R.
\]
%%% where $\alpha := \min\{\gamma,\gamma'\} > 0$.

%Finally, it is easy to see that $\R^2 = \bigcup_{\la\in\Lambda}(\la+R)$.
Finally, we note that for each $\mu\in\R^2$ there exist $\la\in\Lambda$ and $\nu\in R$
with $\mu = \la+\nu$; indeed, there exists $\lambda \in \Lambda$
with $\| \mu - \lambda \|_2 = \dist(\mu, \Lambda)$,
and then $\nu := \mu - \lambda$ satisfies $\|\nu\|_2 = \dist(\mu, \Lambda) = \dist(\nu, \Lambda)$.
Thus, we obtain (see Lemma~\ref{l:commute})
\begin{align*}
  \dist(\pi(\mu)g, \calG)
  &=       \dist(\pi(\nu)g, \calG)
   \,\ge\, \LowerBound \cdot \dist(\nu,\Lambda)
   =       \LowerBound \cdot \dist(\mu,\Lambda) .
\end{align*}
In view of Proposition~\ref{p:upper}, this completes the proof.
\end{proof}

\section{An Explicit Local Bound}
\label{s:ortho}

As mentioned in the introduction, we were unable to derive an \emph{explicit} constant $\LowerBound$
for \eqref{e:main}.
Nevertheless, we can find a constant $\LowerBound$ that is valid for $(u,\eta)$
close to the lattice $\Lambda$.
For this, however, we have to assume that $(g,\Lambda)$ is an orthonormal sequence.
The following result makes a first step towards finding such a constant $\LowerBound$;
it improves Proposition~\ref{p:weak} under the additional assumption of orthonormality.
%%% The following result makes a first step for deriving our explicit lower bound.
%%% It can be seen as a more precise version of Proposition~\ref{p:weak},
%%% under the additional assumption of orthonormality.

\begin{prop}\label{p:ortho_derivative}
%%% Let $\Lambda\subset\R^2$ be a lattice, and let $g\in\HH^1(\R)$ be such that $(g,\Lambda)$ is
%%% an orthonormal basis of its closed linear span $\calG(g,\Lambda)$.
Let $g\in\HH^1(\R)$ and let $\Lambda\subset\R^2$ be a lattice such that $(g,\Lambda)$ is
an orthonormal basis of its closed linear span $\calG(g,\Lambda)$.
Then for any $(a,b)\in\R^2$,
\[
  \dist\left(-a g' + 2\pi i b X g,\, \calG(g,\Lambda)\right)
  \,\ge\, \frac{\pi}
               {\sqrt{\|g'\|_{\LTwoIndex}^2 + \|2\pi i X g\|_{\LTwoIndex}^2}}
          \, \left\|(a,b)\right\|_2.
\]
\end{prop}

\begin{rem}
The classical uncertainty principle (see e.g., \cite[Theorem 2.2.1]{g}),
combined with elementary computations, implies because of $\| g \|_{L^2} = 1$ that
the lower bound appearing in Proposition~\ref{p:ortho_derivative} is bounded by
\[
  \frac{\pi}{\sqrt{\|g'\|_{\LTwoIndex}^2 + \|2\pi i X g\|_{\LTwoIndex}^2}}
  \leq \sqrt{\pi / 2} .
\]
%and achieves the maximum value $\sqrt{\pi / 2}$ exactly when $g(x) = 2^{1/4} c \cdot e^{-\pi x^2}$
%for some $c \in \C$ with $|c| = 1$.
\end{rem}

The proof of Proposition~\ref{p:ortho_derivative} hinges crucially on the following
lemma which describes a general property of Hilbert spaces.

\begin{lem}\label{lem:DistanceLargerThanImaginaryPartOfInnerProduct}
  Let $\calH$ be a Hilbert space, and let $f,g \in \calH$ with $f \neq 0$ or $g \neq 0$.
  Then
  \begin{align*}
    \| a f + b g\|^2
    \geq \frac{\|f\|^2 \cdot \|g\|^2 - (\Re \langle f,g \rangle)^2}
              {\|f\|^2 + \|g\|^2}
         \cdot \| (a,b) \|_2^2
    \geq \frac{(\Im \langle f,g \rangle)^2}
              {\|f\|^2 + \|g\|^2}
         \cdot \| (a,b) \|_2^2
    %\quad \forall \, (a,b ) \in \R^2 \, .
  \end{align*}
  for all $a,b \in \R$.
\end{lem}

\begin{proof}
  Let $\alpha := \|f\|_\calH^2$, $\gamma := \|g\|_\calH^2$,
  and $\beta := \Re \langle f,g \rangle$.
  Moreover, set ${A := \alpha + \gamma}$ and $B := \alpha \gamma - \beta^2$.
  Because of $f \neq 0$ or $g \neq 0$, we have $A > 0$.
  Besides, the Cauchy-Schwarz inequality shows $\beta \leq |\beta| \leq \sqrt{\alpha \gamma}$,
  and thus $B \geq 0$.
  Finally, a direct computation shows
  \(
    A^2 - 4 B
    = (\alpha + \gamma)^2 - 4(\alpha \gamma - \beta^2)
    = (\alpha - \gamma)^2 + 4 \beta ^2
    \geq 0
    .
  \)

  Given these notations, another direct computation shows for $a,b \in \R$ that
  \begin{equation}
    \| a f + b g\|_\calH^2
    = \langle a f + b g, a f + b g \rangle_\calH
    = \big\langle \!
        \left( \begin{smallmatrix} a \\ b \end{smallmatrix} \right) , \,
        M \left( \begin{smallmatrix} a \\ b \end{smallmatrix} \right)
      \! \big\rangle_{\R^2}
    \!\quad \text{where} \quad\!
    M \! := \! \left(
                 \begin{smallmatrix}
                   \alpha & \beta \\
                   \beta  & \gamma
                 \end{smallmatrix}
               \right) \! .
    \label{eq:HilbertSpaceLemmaMainIdentity}
  \end{equation}
  Note that the matrix $M$ is real-symmetric, with characteristic polynomial
  \[
    \chi_M (\lambda)
    = \det \left(
             \begin{matrix}
               \lambda - \alpha & - \beta \\
               - \beta          & \lambda - \gamma
               \end{matrix}
           \right)
    = \lambda^2 - A \, \lambda + B \, ,
  \]
  which has the roots
  \[
    \lambda_{1/2}
    = \frac{A}{2} \pm \sqrt{\frac{A^2}{4} - B}
    = \frac{A \pm \sqrt{A^2 - 4 B}}{2} \, .
  \]
  Therefore, and because of $\sqrt{A^2 - 4 B} \leq \sqrt{A^2} = A$,
  the smallest eigenvalue of $M$ satisfies
  \[
    \lambda_{\min}
    = \frac{A - \sqrt{A^2 - 4 B}}{2}
    = \frac{1}{2} \frac{A^2 - (A^2 - 4B)}{A + \sqrt{A^2 - 4 B}}
    = \frac{2 B}{A + \sqrt{A^2 - 4 B}}
    \geq \frac{B}{A}
    \geq 0 \, .
  \]
  Since $M$ is real symmetric, this implies
  $\langle x, M x \rangle_{\R^2} \geq \frac{B}{A} \|x\|_2^2$ for all $x \in \R^2$.

  Now, Equation~\eqref{eq:HilbertSpaceLemmaMainIdentity} shows that
  \(
    \|a f + b g\|_{\calH}^2
    = \langle
        \left( \begin{smallmatrix} a \\ b \end{smallmatrix} \right),
        M \left( \begin{smallmatrix} a \\ b\end{smallmatrix} \right)
      \rangle_{\R^2}
    \geq \frac{B}{A} \cdot \|(a,b)\|_2^2
  \)
  for all $a,b \in \R$,
  which establishes the first part of the claim.
  For the second part, note that the Cauchy-Schwarz inequality implies
\begin{align*}
B &= \alpha \gamma - \beta^2
      = \|f\|_\calH^2 \, \|g\|_\calH^2 - (\Re \langle f,g \rangle_\calH)^2 \\
      &\geq |\langle f,g \rangle_\calH|^2 - (\Re \langle f,g \rangle_\calH)^2
      =    (\Im \langle f,g \rangle_\calH)^2 \, .
      \qedhere
\end{align*}
\end{proof}

\begin{proof}[Proof of Proposition~\rmref{p:ortho_derivative}]
Denote by $\bP$ the orthogonal projection from $L^2(\R)$ onto $\calG(g,\Lambda)$ in $L^2(\R)$.
Since $(g,\Lambda)$ is an orthonormal sequence,
\[
  \bP f
  = \sum_{\la\in\Lambda}
      \<f,\pi(\la)g\>\pi(\la)g \, ,
  \quad \text{whence} \quad
  %\qquad f\in L^2(\R) \, .
  \<\bP g', i X g\>
  = \sum_{\la\in\Lambda}
      \<g',\pi(\la)g\>\<\pi(\la)g,i X g\> \, .
\]
Let $a,b\in\R$.
By integration by parts and translation, and by using the elementary identity
$(\pi(a,b))^\ast = e^{-2 \pi i a b} \pi(-a, -b)$, we see that
\begin{equation}
  \begin{split}
    \<g',\pi(a,b)g\>
    & = - \Big\langle g, \,\, \frac{d}{dx} \big( e^{2 \pi i b x} \, g(x - a) \big) \Big\rangle \\
      &= - \big\langle g, 2\pi i b \cdot \pi(a,b) g \big\rangle
        - \big\langle g, \pi(a,b) g' \big\rangle \\
    & = 2 \pi i b \cdot \big\langle g, \,\, \pi(a,b) g \big\rangle
        - e^{-2 \pi i a b} \cdot \big\langle \pi(-a,-b) g, g' \big\rangle \, ,
    %= 2\pi ib\<g,\pi(a,b)g\> - e^{-2\pi iab}\<\pi(-a,-b)g,g'\>
  \end{split}
  \label{eq:OrthonormalCaseIdentityOne}
\end{equation}
as well as
\begin{align}
  \<\pi(a,b) g, i X g\>
  & = \big\langle g, e^{-2 \pi i a b} \, \pi(-a, -b) [i X g] \big\rangle
      \nonumber \\
  & = e^{2 \pi i a b} \cdot
      \big\langle
        g, \,\,
        M_{-b} \big[\, i \,\, ( (\cdot) + a ) \,\, g(\cdot + a) \,\big]
      \big\rangle
      \nonumber \\
  & = e^{2 \pi i a b} \,\,
      \big(
        -i a \, \big\langle g, \,\, M_{-b} [g(\cdot + a)] \big\rangle
        + \big\langle -i X g, \,\, M_{-b} [g(\cdot + a)] \big\rangle
      \big)
      \nonumber \\
  & = e^{2 \pi i a b} \,\,
      \big(
        - i a \, \big\langle g, \,\, \pi(-a,-b) g \big\rangle
        - \big\langle i X g, \,\, \pi(-a,-b) g \big\rangle
      \big) \, .
  %= e^{2\pi iab}\left[-ia\<g,\pi(-a,-b)g\> - \<i X g,\pi(-a,-b)g\>\right].
  \label{eq:OrthonormalCaseIdentityTwo}
\end{align}
From Equations~\eqref{eq:OrthonormalCaseIdentityOne} and \eqref{eq:OrthonormalCaseIdentityTwo},
we see by orthonormality of $(g, \Lambda)$ for arbitrary $(a,b) \in \Lambda$ that
\begin{align*}
  \langle g', \pi(a,b) g \rangle
  & = \begin{cases}
        - \langle g, g' \rangle \, ,                                & \text{if } (a,b) = 0 \, , \\
        - e^{-2 \pi i a b} \, \langle \pi(-a,-b) g, g' \rangle \, , & \text{otherwise}
      \end{cases} \\
  & = - e^{-2 \pi i a b} \, \langle \pi (-a, -b) g, g' \rangle
\end{align*}
and
\begin{align*}
  \langle \pi(a,b) g, i X g \rangle
  & = \begin{cases}
        - \langle i X g, g \rangle \, ,                                & \text{if } (a,b) = 0 \,, \\
        - e^{2 \pi i a b} \, \langle i X g, \pi(-a, -b) g \rangle \, , & \text{otherwise}
      \end{cases} \\
  & = - e^{2 \pi i a b} \, \langle i X g, \pi(-a, -b) g \rangle \, .
\end{align*}
Combining these identities, we arrive at
%Thus, if $(a,b)\in\Lambda$, exploiting once more that $(g,\Lambda)$ is orthonormal, we obtain
\[
  \<g', \pi(a,b)g\> \<\pi(a,b)g, i X g\>
  = \<\pi(-a,-b)g, g'\> \<i X g, \pi(-a,-b)g\> \, ,
\]
for all $(a,b) \in \Lambda$.
Therefore, with $\mu = - \lambda$, we see that
\begin{align*}
  \<\bP g',i X g\>
  &= \sum_{\la \in \Lambda}
      \< g', \pi(\la) g \>
      \< \pi(\la) g, i X g \> \\
  &= \sum_{\mu \in \Lambda}
      \< \pi(\mu)g, g' \>
      \< i X g, \pi(\mu)g \>
  = \< i X g, \bP g' \> \, ,
\end{align*}
which shows that $\Im\<\bP g',i X g\> = 0$.

\medskip{}

We now intend to use partial integration to get
$\langle g', Xg \rangle = - \|g\|_{\LTwoIndex}^2 - \langle X g, g' \rangle$;
however, since $X g \notin H^1(\R)$, we cannot \emph{directly} apply such a partial integration.
Instead, pick $\varphi \in C_c^\infty (\R)$ with $0 \leq \varphi \leq 1$,
$\supp \varphi \subset (-2,2)$, and $\varphi \equiv 1$ on $(-1,1)$,
and set $\varphi_n : \R \to [0,1], x \mapsto \varphi(x/n)$.
We then have $\varphi_n \to 1$ pointwise, so that the dominated convergence theorem implies
$\langle f, \varphi_n \cdot h \rangle \to \langle f, h \rangle$ for all $f,h \in L^2(\R)$.
Likewise, we have $\varphi_n ' (x) = n^{-1} \cdot \varphi' (x/n)$ and hence $\varphi_n' \to 0$
uniformly, which implies $\langle f, \varphi_n ' \cdot h \rangle \to 0$ for $f,h \in L^2(\R)$.
Overall, since $\varphi_n X g \in H^1(\R)$, we thus see
\begin{align*}
  \langle g', X g \rangle
  & = \lim_{n \to \infty} \langle g', \varphi_n X g \rangle
    = \lim_{n \to \infty} - \langle g, (\varphi_n X g)' \rangle \\
  & = - \lim_{n \to \infty}
        \big[
          \langle g, \varphi_n ' \cdot X g \rangle
          + \langle g, \varphi_n g \rangle
          + \langle g, \varphi_n X g' \rangle
        \big]
    = - \|g\|_{\LTwoIndex}^2 - \langle X g, g' \rangle \, .
\end{align*}
Here, we used in the last step that
$\langle g, \varphi_n X g' \rangle = \langle X g, \varphi_n g' \rangle$ with $X g, g' \in L^2(\R)$.
%Besides, partial integration yields
%\[
%  \<g', X g\>
%  = \int
%      g' \cdot x \overline{g}
%    \,dx
%  = -\int
%       g \, (\overline{g} + x \overline{g'})
%     \,dx
%  = -\|g\|_2^2 - \<X g, g'\> \, ,
%\]

In view of the last displayed equation, we get
$2 \Re \< g', X g \> = - \|g\|_{\LTwoIndex}^2 = -1$,
and hence $\Im\<g',2\pi i X g\> = -2 \pi \Re\<g', X g\> = \pi$.
Therefore,
\[
  \Im\<(I-\bP)g', 2\pi i X g\>
  = \Im\<g',2\pi i X g\> - 2 \pi \Im \< \bP g', i X g \>
  = \pi.
\]
Setting $f := (I - \bP) [- g']$ and $h := (I - \bP) [2 \pi i X g]$,
we have shown up to now that $\Im \langle f, h \rangle = -\pi \neq 0$,
which in particular implies that $f \neq 0$ and $h \neq 0$.
Thus, an application of Lemma~\ref{lem:DistanceLargerThanImaginaryPartOfInnerProduct} shows
for arbitrary $(a,b) \in \R^2$ that
\begin{align*}
  & \lefteqn{\dist^2 \big( - a g' + 2 \pi i b X g,\, \calG(g,\Lambda) \big) }\\
  & = \| (I - \bP) (- a g'+ 2 \pi i b X g)\|_{\LTwoIndex}^2
    = \| a \cdot f + b \cdot h \|_{\LTwoIndex}^2 \\
  & \geq \frac{(\Im \langle f,h \rangle)^2}
              {\|f\|_{\LTwoIndex}^2 + \|h\|_{\LTwoIndex}^2}
         \cdot \| (a,b) \|_2^2
    \geq \frac{\pi^2 \cdot \| (a,b) \|_2^2}
              {\|g'\|_{\LTwoIndex}^2 + \|2 \pi i X g\|_{\LTwoIndex}^2} \,.
  %&= \|(I-\bP)g'\|_2^2 \, a^2
  %   - 2 \Re\<(I-\bP)g',2\pi ixg\> \, ab
  %   + \|(I-\bP)(2\pi ixg)\|_2^2 \, b^2 \\
  %&= {{\vek a b}^T
  %    \mat{\|(I-\bP)g'\|_2^2}
  %        {-\Re\<(I-\bP)g',2\pi ixg\>}
  %        {-\Re\<(I-\bP)g',2\pi ixg\>}
  %        {\|(I-\bP)(2\pi ixg)\|_2^2}
  %    \vek a b }.
\end{align*}
%As the smallest eigenvalue of a real symmetric matrix
%$\smallmat \alpha \beta \beta \gamma$ is given by
%\begin{align*}
%  \la_{\min}
%  &= \frac{\alpha+\gamma -\sqrt{(\alpha+\gamma)^2-4(\alpha\gamma-\beta^2)}}
%          {2} \\
%  &= \frac{4(\alpha\gamma-\beta^2)}
%          {2\,(\alpha+\gamma + \sqrt{(\alpha+\gamma)^2-4(\alpha\gamma-\beta^2)})}
%  \,\ge\, \frac{\alpha\gamma-\beta^2}
%               {\alpha+\gamma},
%\end{align*}
%we obtain that
%\begin{align*}
%  &\frac{\dist^2 \left(-ag' + 2\pi ibxg,\,\calG(g,\Lambda)\right)}
%       {\left\|(a,b)\right\|_2^2} \\
%  &\geq \frac{\|(I-\bP)g'\|_2^2\|(I-\bP)(2\pi ixg)\|_2^2 - (\Re\<(I-\bP)g',2\pi ixg\>)^2}
%             {\|(I-\bP)g'\|_2^2 + \|(I-\bP)(2\pi ixg)\|_2^2} \\
%  &\ge\frac{(\Im\<(I-\bP)g',2\pi ixg\>)^2}
%           {\|g'\|_2^2 + \|2\pi ixg\|_2^2} .
%\end{align*}
This concludes the proof of the proposition.
\end{proof}

\begin{thm}\label{t:ortho}
%%% Let $\Lambda\subset\R^2$ be a lattice and let $g\in\HH^1(\R)$ be such that
%%% $(g,\Lambda)$ is an orthonormal basis of its closed linear span $\calG(g,\Lambda)$.
Let $g\in\HH^1(\R)$ and let $\Lambda\subset\R^2$ be a lattice
such that $(g,\Lambda)$ is an orthonormal basis of its closed linear span $\calG(g,\Lambda)$.
Then there exists $\veps > 0$ such that
\[%\label{e:ortho}
  \dist \big( \pi(\mu)g,\, \calG(g,\Lambda) \big)
  \,\geq\, \frac{\pi/2}{\sqrt{\|g'\|_{\LTwoIndex}^2 + \|2\pi i X g\|_{\LTwoIndex}^2}}\,
           \dist(\mu,\Lambda)
  \quad \forall \, \mu\in\Lambda \!+\! B_\veps(0) .
\]
If $g\in\HH^2(\R)$, then $\veps$ can be chosen as
$\veps := \pi \Big/ \Bigl(2 C_g \sqrt{\|g'\|_{\LTwoIndex}^2 + \|2\pi i X g\|_{\LTwoIndex}^2} \,\Bigr)$
with $C_g$ as in Equation~\eqref{e:Cg}.
\end{thm}

\begin{proof}
For $(a,b)\in\R^2$ let $\gamma(a,b) := \pi(a,b)g - g - (-a g'+2\pi i b X g)$.
Denote by $\bP$ the orthogonal projection from $L^2(\R)$ onto $\calG(g, \Lambda)$.
Due to Proposition \ref{p:ortho_derivative} we have
\begin{align*}
  \frac{\pi}{\sqrt{\|g'\|_{\LTwoIndex}^2 + \|2\pi i X g\|_{\LTwoIndex}^2}}
  \|(a,b)\|_2
  &\le\|(I-\bP)(-a g'+2\pi i b X g)\|_{\LTwoIndex} \\[-0.5cm]
  &= \big\|(I-\bP) \big( \pi(a,b)g - g - \gamma(a,b) \big) \big\|_{\LTwoIndex}\\
  &\le\|(I-\bP)\pi(a,b)g\|_{\LTwoIndex} + \|\gamma(a,b)\|_{\LTwoIndex} \, .
\end{align*}
In the last inequality we used that $(I - \bP) g = 0$ and $\|I - \bP\| = 1$.
By Lemma~\ref{l:differentiable} there exists $\veps > 0$ such that
\[
  \|\gamma(a,b)\|_{\LTwoIndex}
  \le\frac{\pi/2}{\sqrt{\|g'\|_{\LTwoIndex}^2 + \|2\pi i X g\|_{\LTwoIndex}^2}} \, \|(a,b)\|_2
  \qquad\text{for $\|(a,b)\|_2 < \veps$} \, .
\]
Moreover, this is satisfied in the case $g\in\HH^2(\R)$ if $\veps$ is
as given in the theorem (see Lemma \ref{l:differentiable}).
Hence, if $(\alpha,\beta)\in\Lambda + B_\veps(0)$, say $(\alpha,\beta) = \la + (a,b)$
with $\lambda \in \Lambda$ and $(a,b) \in B_\eps (0)$,
then (see Lemma \ref{l:commute})
\begin{align*}
  \|(I \!-\! \bP)\pi(\alpha,\beta)g\|_{\LTwoIndex}
  = \|(I \!-\! \bP)\pi(a,b)g\|_{\LTwoIndex}
  &\geq \frac{\pi/2}{\sqrt{\|g'\|_{\LTwoIndex}^2 + \|2\pi i X g\|_{\LTwoIndex}^2}}
        \cdot \|(a,b)\|_2 \, .
\end{align*}
This proves the theorem.
\end{proof}

\begin{rem*}%\label{r:H1}
In the case $g\in\HH^1(\R)$, the value of $\veps$ in Theorem~\ref{t:ortho}
depends on the convergence to zero of the following quantities
(see the proof of Lemma~\ref{l:differentiable}):
\[
  \int_{|x|> b}x^2 |g(x)|^2\,dx,
  \quad
  \int_{|x|>b}|g'(x)|^2\,dx
  \quad\text{and}\quad
  \int_{|\omega|>b}\omega^2|\wh g(\omega)|^2\,d\omega
  \quad\text{ as $b\to\infty$}.
\]
\end{rem*}

Note that the lattice $\Lambda$ in Theorem~\ref{t:ortho} is not necessarily of rational density.
The following corollary suggests that the rational density condition of $\Lambda$
in Theorems~\ref{t:clpp} and \ref{t:quanty} might be redundant.

\begin{cor}\label{c:new}
%%% Let $\Lambda\subset\R^2$ be a lattice and let $g\in\HH^1(\R)$
%%% be such that $(g,\Lambda)$ is an orthonormal basis of its closed linear span $\calG(g,\Lambda)$.
Let $g\in\HH^1(\R)$ and let $\Lambda\subset\R^2$ be a lattice such that $(g,\Lambda)$
is an orthonormal basis of its closed linear span $\calG(g,\Lambda)$.
Then there exists an $N\in\N$ such that $\pi(\mu)g\notin\calG(g,\Lambda)$
for all $\mu \in \R^2 \backslash \frac{1}{N} \Lambda$;
that is, $\calG(g,\Lambda)$ is invariant only under time-frequency shifts
with parameters in a subset of $\frac{1}{N} \Lambda$.
%%% In particular, $\pi(\mu)g\notin\calG(g,\Lambda)$ holds for all $\mu\in\R^2$
%%% apart from a discrete subset of $\R^2$.
\end{cor}

\begin{proof}
This follows by combining Theorem~\ref{t:ortho}
with \cite[Lemma~3.1]{CarageaBalianLowForSubspaces}.
\end{proof}

%%%%%%%%%%%%%%%%%%%%%%%%%%%%%%%%%%%%%%%%%%%%%%%%%%%%%%%%%%%%%%%%%%%%%
\appendix

\renewcommand{\thetheorem}{A.\arabic{theorem}}
\setcounter{theorem}{0}

\section{Auxiliary results}
\label{sec:Auxiliary}

\subsection{Matrix multiplication operators}
\label{ss:multi}
Let $(\Omega,\Sigma,\mu)$ be a measure space.
To avoid trivialities, we assume that there exists $S \in \Sigma$ with $0 < \mu(S) < \infty$.
Now, let $B : \Omega \to \C^{n{\times} m}$ be a measurable matrix-valued function.
Then the multiplication operator
\[
  M_B : \quad \dom(M_B) \subset L^2(\Omega, \C^m) \to L^2(\Omega, \C^n)
\]
is defined by
\[
  (M_Bf)(\omega) := B(\omega)f(\o),
  \qquad\omega\in\Omega, \; f\in\dom(M_B),
%  \qquad\omega\in\Omega,\,f\in\dom(M_B),
\]
where
\[
  \dom(M_B)
  := \left\{
       f\in L^2(\Omega; \C^m)
       \colon
       \int \|B(\o)f(\o)\|_{\C^n}^2\,d\mu(\omega) < \infty
     \right\}.
\]
It is easy to see that the operator $M_B$ is bounded
if and only if each entry of $B(\cdot)$ is essentially bounded as a function on $\Omega$,
if and only if $\dom(M_B) = L^2(\Omega; \C^m)$.
%Here, as in all of the paper, we consider the space $L^2(\Omega;\C^k)$ as equipped with the inner product
%$\langle f,g \rangle = \int \langle f(\omega), g(\omega) \rangle_{\C^k} d\mu(\omega)$,
%where $\langle \cdot, \cdot \rangle_{\C^k}$ denotes the standard inner product on $\C^k$.

Let $A : \calH \to \calH$ be a bounded self-adjoint operator in a Hilbert space $\calH$.
Then for any continuous, real-valued function $\vphi\in C(\sigma(A); \R)$,
the operator $\vphi(A)$ is defined by $\varphi(A) := \lim_{n\to\infty} p_n(A)$,
where $(p_n)_{n \in \N}$ is a sequence of real-valued polynomials
converging uniformly to $\vphi$ on $\sigma(A) \subset \R$
and the limit is taken with respect to the operator norm.
Since $\|p(A)\| = \|p\|_{C(\sigma(A))}$ for polynomials $p$, this definition is meaningful.
One then has ${\|\vphi(A)\| = \|\vphi\|_{C(\sigma(A))}}$
and ${\sigma(\vphi(A)) = \{\vphi(\la) : \la\in\sigma(A)\}}$.
Furthermore, $\varphi(A)$ is self-adjoint for all ${\varphi \in C(\sigma(A); \R)}$,
since this is easily seen to hold for all polynomials $p_n$.
For more details on this {\em continuous functional calculus},
see \mbox{\cite[Section~VII.1]{rsi}}.

For the case $n = m$, the next lemma connects the spectral properties of the multiplication
operator $M_B$ to those of the matrices $B(\omega)$, $\omega \in \Omega$.

\begin{lemma-A}\label{l:multi}
Let $B : \Omega \to \C^{n{\times} n}$ be a measurable, essentially bounded matrix-valued function
satisfying $B(\o) = B(\o)^*$ for a.e.\ $\o\in\Omega$.
Then the following statements hold:
\begin{enumerate}
  \item[{\rm (i)}]   The operator $M_B$ is bounded and self-adjoint.

  \item[{\rm (ii)}]  For a.e.\ $\o\in\Omega$ we have
                     \[
                       \sigma(B(\o))\subset\sigma(M_B).
                     \]

  \item[{\rm (iii)}] For every set $N\subset\Omega$ of zero measure,
                     \[
                       \sigma(M_B)
                       \subset \ol{
                                   \bigcup_{\omega \in \Omega \backslash \!N}
                                     \! \sigma(B(\omega))
                                  } \, .
                     \]

  \item[{\rm (iv)}]  For every function $\vphi\in C(\sigma(M_B); \R)$ we have
                     \[
                       \vphi(M_B) = M_{\vphi(B)},
                     \]
                     where $(\vphi(B))(\omega) := \vphi(B(\omega))$ is well-defined
                     for almost all $\omega\in\Omega$.
\end{enumerate}
\end{lemma-A}

\begin{proof}
Part~(i) follows easily from
\(
  \langle B(\omega)f (\omega), g(\omega) \rangle_{\C^n}
  = \langle f(\omega), B(\omega) g(\omega) \rangle_{\C^n} ,
\)
which holds for almost all $\omega \in \Omega$.

\smallskip{}

For the proofs of both (ii) and (iii), we will use \cite[Proposition 1]{h}, which shows
for any $\lambda \in \C$ that
\begin{equation}
  \la\in\varrho(M_B)
  \Llra \, \exists \, \veps > 0 :
             \mu ( \{\o \in \Omega \colon \sigma(B(\o)) \cap B_\veps(\la) \neq \emptyset \}) = 0
  \, .
  \label{eq:MultiplicationOperatorSpectrumCharacterization}
\end{equation}
To prove (ii), let us assume towards a contradiction that the claim is false;
that is, the set
\[
  \Omega_0
  := \big\{
       \omega \in \Omega
       \colon
       \sigma(B(\omega)) \cap \varrho(M_B) \neq \emptyset
     \big\}
\]
is \emph{not} a null-set.
Since $\varrho(M_B)$ is an open set, we have $\varrho(M_B) = \bigcup_{k \in \N} I_k$
for certain compact sets $I_k \subset \varrho(M_B) \subset \C$.
Setting ${\Omega_k := \{ \omega \in \Omega \colon \sigma(B(\omega)) \cap I_k \neq \emptyset \}}$,
we then have $\Omega_0 = \bigcup_{k \in \N} \Omega_k$, so that there is some $k \in \N$ for which
$\Omega_k$ is not a null-set.
Let us choose a dense subset $\{\lambda_n \colon n \in \N \}$ of $I_k$, and define
\[
  \Omega_{m,n}
  := \big\{
       \omega \in \Omega
       \colon
       \sigma(B(\omega)) \cap B_{1/m} (\lambda_n) \neq \emptyset
     \big\}
  \qquad \text{for } m,n \in \N \, .
\]
By density, we have $I_k \subset \bigcup_{n \in \N} B_{1/m} (\lambda_n)$ for every $m \in \N$,
and hence ${\Omega_k \subset \bigcup_{n \in \N} \Omega_{m,n}}$.
Thus, for each $m \in \N$, there is some $n_m \in \N$ such that $\Omega_{m,n_m}$ is not a null-set.

Since $(\lambda_{n_m})_{m \in \N}$ is a sequence in the compact set $I_k$, there is a subsequence
$(\lambda_{n_{m_\ell}})_{\ell \in \N}$ such that
$\lambda_{n_{m_\ell}} \to \lambda \in I_k \subset \varrho(M_B)$ as $\ell \to \infty$.
By \eqref{eq:MultiplicationOperatorSpectrumCharacterization},
there is some $\veps > 0$ such that
$\Theta := \{ \omega \in \Omega \colon \sigma(B(\omega)) \cap B_\veps (\lambda) \neq \emptyset \}$
is a null-set.
But for $\ell \in \N$ large enough,
we have $\frac{1}{m_\ell} + |\lambda_{n_{m_\ell}} - \lambda| < \veps$, and hence
$B_{1/m_\ell}(\lambda_{n_{m_\ell}}) \subset B_\veps (\lambda)$, which shows that
$\Omega_{m_\ell, n_{m_\ell}} \subset \Theta$ is a null-set.
This is the desired contradiction.

%For (ii) we shall show that $\varrho(M_B)\subset\varrho(B(\omega))$ for a.e.\ $\o$.
%To this end, let $C_k$ be compact intervals such that
%$\varrho(M_B) = \bigcup_{k\in\N}C_k$.
%By \cite[Prop.\ 1]{h} we have that
%\begin{align}
%  \la\in\varrho(M_B)
%  &\Llra \, \exists\veps > 0 :
%              |\{\o : \sigma(B(\o))\cap B_\veps(\la)\neq\emptyset\}| = 0\label{e:first}\\
%  &\Llra \, \exists\veps > 0 :
%              \text{For a.e.\ $\o$ we have }B_\veps(\la)\subset\varrho(B(\o))\label{e:second}.
%\end{align}
%From \eqref{e:second} and a compactness argument we get that for each $k$
%we have $C_k\subset\varrho(B(\o))$ for a.e.\ $\omega\in\Omega$.
%Hence, $\varrho(M_B)\subset\varrho(B(\omega))$ holds for a.e.\ $\o\in\Omega$.

\smallskip{}

To prove (iii) let $\la \in \sigma(M_B)$ and let $N \subset \Omega$ be of zero measure.
If $k \in \N$ is arbitrary, then by \eqref{eq:MultiplicationOperatorSpectrumCharacterization},
the set $\{\o \in \Omega : \sigma(B(\o))\cap B_{1/k}(\la)\neq\emptyset\}$
does not have measure zero, and thus has non-empty intersection with $\Omega \backslash N$.
Hence, we can pick $\o_k \in \Omega \backslash N$ and $\la_k\in\sigma(B(\o_k))$
such that $|\la_k-\la|<1/k$.
This proves the inclusion in (iii).

\smallskip{}

Part~(iv) is obvious for polynomials $\vphi$.
Given general $\vphi\in C(\sigma(M_B); \R)$, we can approximate $\vphi$
uniformly on $\sigma(M_B) \subset \R$ by polynomials $p_n$.
Then ${\vphi(M_B) - p_n(M_B)}$ converges to zero in operator norm, and $p_n(M_B) = M_{p_n(B)}$.
Hence, we see for every $f \in L^2(\Omega;\C^n)$
that $\varphi(M_B) f = \lim_{n \to \infty} M_{p_n(B)} f$.
%Hence, it remains to estimate the operator norm $\|M_{p_n(B)} - M_{\vphi(B)}\|$.
But by (ii), we have $\sigma(B(\omega)) \subset \sigma(M_B)$ for almost every $\omega \in \Omega$.
For each such $\omega \in \Omega$,
\[
  \|p_n (B(\omega)) - \varphi(B(\omega))\|
  = \|p_n - \varphi\|_{C(\sigma(B(\omega)))}
  \leq \|p_n - \varphi\|_{C(\sigma(M_B))} \,.
\]
Thus, we see
\(
  [M_{p_n(B)} f] (\omega)
  = p_n (B(\omega)) f(\omega) \to \varphi(B(\omega)) f(\omega)
  = [M_{\varphi(B)} f](\omega)
\)
for almost every $\omega$, for every $f \in L^2(\Omega; \C^n)$.
Since also have $M_{p_n(B)} f \to \varphi(M_B) f$ with convergence in $L^2(\Omega,\C^n)$,
this implies $\varphi(M_B) f = M_{\varphi(B)} f$, as claimed.
\end{proof}

\subsection{Operators with closed range and their pseudo-inverse}
\label{sub:PseudoInverseGeneral}

In this subsection, we review the notion of the pseudo-inverse
of an operator with closed range and some of its elementary properties.
%%% All of these properties are well-known in general; yet, since some readers
%%% might not be too familiar with this notion, we prefer to repeat the essentials.
All of these properties are well-known in general; yet, as some readers might not be familiar
with them we decided to include the essentials.
Throughout this subsection $\calH$, $\calK$, and $\calL$ denote Hilbert spaces.

\begin{lemma-A}\label{lem:closed_equis}
  Let $A : \calH \to \calK$ be a bounded linear operator. Then
  \begin{equation}\label{e:kerran}
    (\ker A)^\perp = \overline{\ran A^\ast}\,.
  \end{equation}
  Moreover, the following statements are equivalent:
  \begin{enumerate}
    \item[{\rm (a)}] $\ran A$ is closed in $\calK$.

    \item[{\rm (b)}] $\ran(AA^*)$ is closed in $\calK$.

    \item[{\rm (c)}] $\ran(A^*A)$ is closed in $\calH$.

    \item[{\rm (d)}] $\ran A^*$ is closed in $\calH$.

    \item[{\rm (e)}] $\sigma_1(A) > 0$.

    \item[{\rm (f)}] $\sigma_1(A^*) > 0$.
  \end{enumerate}
  In one of these properties holds, then the following identities hold:
  \begin{equation}\label{e:range_identities}
  \ran(AA^*) = \ran A, \quad \ran(A^*A) = \ran A^*,
  \quad\text{and}\quad
  \sigma_1(A) = \sigma_1(A^*).
  \end{equation}
\end{lemma-A}

\begin{proof}
The identity \eqref{e:kerran} is a simple exercise
(see \cite[Theorem~58.2]{HeuserFunktionalanalysis}).

For the equivalence of (a)--(f), we refer to \cite[Theorem~2]{OuldAliClosedRangeOperators}.

Next, if (a)--(f) hold, then Equation~\eqref{e:kerran} shows $(\ker A)^{\perp} = \ran (A^\ast)$.
This implies $\ran A = \ran (A|_{(\ker A)^\perp}) = \ran(A|_{\ran A^\ast}) = \ran(A A^\ast)$,
which proves the first part of Equation~\eqref{e:range_identities}.
The second part follows by applying the first part to $A^\ast$ instead of $A$.

The last identity in \eqref{e:range_identities} follows directly from the definition of $\sigma_1$
and the well-known Jacobson lemma which states that for arbitrary bounded linear operators
$S : \calH \to \calK$ and $T : \calK \to \calH$
we have $\sigma(ST) \backslash \{0\} = \sigma(TS) \backslash \{0\}$.
It can indeed be easily seen that $\la\in\varrho(TS)\backslash\{0\}$ implies $\la\in\varrho(ST)$,
by virtue of the identity
\[
  (ST-\la I_\calK)^{-1}
  = \frac 1\la\left[S(TS-\la I_\calH)^{-1}T - I_\calK\right] \, .
\]
By symmetry, this implies $\varrho(T S) \backslash \{0\} = \varrho(S T) \backslash \{0\}$.
%\todo{I would like to retain this proof, unless we can find a proper reference.
%I could only find references for the case where $S, T : \calH \to \calH$, not for the case
%when two distinct Hilbert spaces are involved.}
%An elementary proof of this fact (based on Problem 76 in \cite{HalmosHilbertSpaceProblemBook})
%goes as follows:
%If $1 \notin \sigma(S T)$, there is a bounded operator $C : \calK \to \calK$ such that
%$I_{\calK} = C (I_\calK - S T) = (I_\calK - S T) C$,
%or equivalently $C S T = S T C = C - I_{\calK}$.
%Thus, an easy computation shows
%$(I_\calH + T C S) (I_\calH - T S) = I_\calH = (I_\calH - T S) (I_\calH + T C S)$,
%whence $1 \notin \sigma (T S)$.
%By symmetry (we can swap $T,S$), this shows $1 \in \sigma(TS) \Longleftrightarrow 1 \in \sigma(ST)$.
%Finally, if $\lambda \in \C \backslash \{0\}$ is arbitrary, then
%\[
  %\lambda \in \sigma(T S)
  %\quad \Longleftrightarrow \quad
  %1 \in \sigma \Big( T \frac{S}{\lambda} \Big)
  %\quad \Longleftrightarrow \quad
  %1 \in \sigma \Big( \frac{S}{\lambda} T \Big)
  %\quad \Longleftrightarrow \quad
  %\lambda \in \sigma (S T) \, ,
%\]
%as desired.
\end{proof}

\begin{lemma-A}\label{lem:BoundedBelowAStarA}
  A bounded operator $A : \calK \to \calH$ is bounded below
  (meaning that there is $c > 0$ with $\|A x\|_\calH \geq c \, \|x\|_\calK$ for all $x \in \calK$)
  if and only if $A^\ast A : \calK \to \calK$ is bounded below.

  Furthermore, a bounded \emph{self-adjoint} operator $T : \calH \to \calH$
  is bounded below if and only if $T$ is boundedly invertible.
\end{lemma-A}

\begin{proof}
  Using the bounded inverse theorem, it is easy to see that a bounded operator $T$ between two
  Hilbert spaces is bounded below if and only if ${\ker T = \{0\}}$ and if $\ran T$ is closed.
  Lemma~\ref{lem:closed_equis} shows that $\ran A$ is closed if and only if $\ran (A^\ast A)$
  is closed.
  Since furthermore $\ker A = \ker(A^*A)$, we obtain the first claim.

  For the second part of the claim, let $T : \calH \to \calH$ be bounded, self-adjoint,
  and bounded below.
  As seen above, this implies that $\ran T$ is closed and that $\ker T = \{0\}$.
  Therefore, Equation~\eqref{e:kerran} shows
  $\calH = (\ker T)^\perp = \overline{\ran T^\ast} = \ran T$.
  Hence, $T : \calH \to \calH$ is bijective, so that the bounded inverse theorem shows
  that $T$ is boundedly invertible.
  It is clear that if $T$ is boundedly invertible, then $T$ is bounded below.
\end{proof}

The next lemma follows directly from \cite[Cor.~5.5.2 and Cor.~5.5.3]{ChristensenIntroductionToFrames}.

\begin{lemma-A}\label{lem:FrameSequenceIffGramianHasClosedRange}
  A Bessel sequence $(\varphi_i)_{i \in I}$ in a Hilbert space $\calH$ is a \emph{frame sequence}
  %{\rm (}that is, a frame for its closed linear span{\rm )}
  if and only if its \emph{analysis operator}
  \(
    \mathbf{A} : \calH \to \ell^2(I), f \mapsto \big( \langle f, \varphi_i \rangle \big)_{i \in I}
  \)
  has closed range.
\end{lemma-A}

%\begin{proof}
  %``$\Rightarrow$:'' Let $\Psi = (\psi_j)_{j \in J}$ be a frame sequence,
  %let $A$ denote the associated analysis operator,
  %and set $V := \overline{\linspan \{\psi_j : j \in J \}}$.
  %We have $\ran (A) = \ran (A|_V)$, since $A = 0$ on $V^\perp$.
  %Since $\Psi$ is a frame sequence, there is some $c > 0$ such that
  %$\|A|_V f\|_{\ell^2} \geq c \cdot \|f\|_{\calH}$ for all $f \in V$.
  %By standard arguments, this shows that $\ran (A) = \ran(A|_V) \subset \ell^2 (J)$ is closed.
%
  %``$\Leftarrow$:'' Let $\Psi = (\psi_j)_{j \in J}$ be Bessel in $\calH$ with
  %analysis operator $A : \calH \to \ell^2 (J)$,
  %and assume that $\ran (A) \subset \ell^2 (J)$ is closed.
  %Note that $\ker A = \{\psi_j : j \in J\}^\perp = V^\perp$ for
  %$V := \overline{\linspan \{\psi_j : j \in J\}}$.
  %Therefore, the restricted operator $A_0 : V \to \ran(A), x \mapsto A \, x$ is easily seen
  %to be bijective.
  %By the bounded inverse theorem, $A_0^{-1}$ is bounded, and hence
  %$\|x\|_\calH = \|A_0^{-1} A_0 x\|_\calH \leq c^{-1} \|A_0 x\|_{\ell^2}$, for $c > 0$ small enough.
  %In other words, $\|A f\|_{\ell^2}^2 = \|A_0 f\|_{\ell^2}^2 \geq c^2 \cdot \|f\|_\calH^2$
  %for all $f \in V$.
  %Hence, $\Psi$ is a frame for $V$, as desired.
%\end{proof}

Let $A : \calH \to \calK$ be a bounded linear operator with closed range.
Then the operator
\begin{equation}\label{eq:PseudoInverseBuildingBlock}
  A_0 : (\ker A)^\perp \to \ran A,\qquad  x \mapsto A \, x,
\end{equation}
is boundedly invertible by the bounded inverse theorem.
Hence, the \emph{pseudo-inverse}
\[
  \pseudo{A} := \iota_{(\ker A)^{\perp}} \circ A_0^{-1} \circ \, P_{\ran A}
\]
of $A$ defines a bounded linear operator from $\calK$ to $\calH$.
Here, $\iota_{(\ker A)^{\perp}}$ is the inclusion map
$(\ker A)^{\perp} \to \calH$, $x\mapsto x$.
%\blue{CHECK for other notations and definition that are repeated. }

In the following lemma we list some of the properties of the pseudo-inverse.

\begin{lemma-A}\label{lem:TechnicalPseudoInverseIdentities}
Let $A : \calH \to \calK$ be a bounded linear operator with closed range.
Then the following hold:
\begin{enumerate}
  \item[(i)] $\pseudo{A} A = P_{(\ker A)^\perp}$.

  \item[(ii)] $A\pseudo{A} = P_{\ran A}$.

  \item[(iii)] $(A^\dagger)^* = (A^*)^\dagger$.

  \item[(iv)] $\pseudo{(A^\ast A)} A^\ast = \pseudo{A} = A^\ast \pseudo{(A A^\ast)}$.
\end{enumerate}
\end{lemma-A}

\begin{proof}
Properties~(i)--(iii) can be found in
\cite[Lemma~2.5.2]{ChristensenIntroductionToFrames}.

For the first identity in (iv),
we refer to \cite[Theorem~1]{DesoerPseudoInverses}.
The remaining identity follows from the first one and (iii) by applying
the first part of (iv) on the right-hand side
of the identity $\pseudo{A} = \bigl(\pseudo{(A^\ast)}\bigr)^\ast$.
\end{proof}

\begin{lemma-A}\label{l:vphi}
Let $A : \calH \to \calH$ be a self-adjoint operator with closed range and set $c := \sigma_1(A)$.
Then $\sigma(A)\subset\{0\}\cup(\R \backslash (-c,c))$ and $A^\dagger = \vphi(A)$,
where $\vphi : \R\to\R$ is defined by $\vphi(t) = \tfrac 1 t$ for $t\neq 0$ and $\vphi(0) = 0$.
\end{lemma-A}

\begin{rem*}
  Since $0$ is an isolated point of $\sigma(A) \subset \{0\} \cup \big( \R \backslash (-c,c) \big)$,
  $\varphi|_{\sigma(A)}$ is continuous.
\end{rem*}

\begin{proof}
  Lemma~\ref{lem:closed_equis} shows $c = \sigma_1 (A) > 0$.    
  By definition of $\sigma_1 (A)$ 
  (see Equation~\eqref{eq:SigmaDefinitions}), we thus see that $A^2 = A^\ast A$ satisfies
  $\sigma(A^2) \subset \{0\} \cup [c^2, \infty)$.
  %By Lemma \ref{lem:closed_equis} the operator $A^2 = A^*A$ has spectrum
  %in $\{0\}\cup [c^2,\infty)$, where $c = \sigma_1(A) > 0$.
  As $\sigma(A^2) = \{\la^2 : \la\in\sigma(A)\}$ and since $\sigma(A) \subset \R$
  because of $A^\ast = A$, it follows that $\sigma(A) \subset \{0\} \cup (\R \backslash (-c,c))$.
  In particular, this entails that $\varphi|_{\sigma(A)}$ is continuous.

  To prove $\pseudo{A} = \varphi(A)$, define $\psi := \indicator_{\{ 0 \}}$ and note
  $\psi \in C(\sigma(A); \R)$ since $0$ is an isolated point of $\sigma(A)$
  (or even $0 \notin \sigma(A)$).
  Since $\psi^2 = \psi$, we see that $P := \psi(A)$ satisfies $P^2 = P = P^\ast$,
  so that $P = P_V$ is the orthogonal projection onto a closed subspace $V \subset \calH$.
  For $x \in \ker A$ we have $A x = 0 x$, so that \mbox{\cite[Theorem~VII.1(d)]{rsi}} shows
  $P x = \psi(A)x = \psi(0) x = x$; hence, $\ker A \subset V$.
  Conversely, we have $\id_{\sigma(A)} \cdot \psi \equiv 0$
  and hence $0 = (\id_{\sigma(A)} \cdot \psi)(A) = A P$,
  which shows $V = \ran P \subset \ker A$ and hence $V = \ker A$.

  Next, observe that $\varphi \cdot \id_{\sigma(A)} = 1 - \psi$, whence
  ${\varphi(A) A = \id_{\calH} - P = P_{V^\perp} = \pseudo{A} A}$,
  where the last step used Lemma~\ref{lem:TechnicalPseudoInverseIdentities}(a).
  Hence, $\varphi(A) = \pseudo{A}$ on $\ran A$.
  Finally, we have $\varphi(A) P_V = (\varphi \cdot \psi)(A) = 0$,
  meaning $\varphi(A) = 0 = \pseudo{A}$ on ${V = \ker A = (\ran A)^\perp}$.
  Overall, this shows $\varphi(A) = \pseudo{A}$, as claimed.
\end{proof}

\begin{corollary-A}\label{cor:PseudoInverseConjugation}
  Let $A : \calH \to \calH$ be a bounded, self-adjoint operator with closed range,
  and let $U : \calK \to \calH$ be unitary.
  Then $U^\ast A U : \calK \to \calK$ is also bounded and self-adjoint with closed range,
  and we have $\pseudo{(U^\ast A U)} = U^\ast \pseudo{A} U$.
\end{corollary-A}

\begin{proof}
  It is clear that $U^\ast A U$ is bounded and self-adjoint with closed range.
  Furthermore, a direct calculation shows $p(U^\ast A U) = U^\ast p(A) U$
  for every polynomial $p \in \R [x]$.
  By definition of the continuous spectral calculus, we thus get
  $\varphi(U^\ast A U) = U^\ast \varphi(A) U$ for all $\varphi \in C(\sigma(A); \R)$,
  where we note $\sigma(A) = \sigma(U^\ast A U)$.
  Now, the claim follows from Lemma~\ref{l:vphi}.
\end{proof}

\subsection{Some properties of Sobolev functions}
\label{sub:FatH2MixedOperators}

\subsubsection{Essentially bounded \braces{matrix-valued} Sobolev functions}

Our main objective in this subsection is to prove that the space of
matrix-valued functions with all entries in $H^1(\Omega)\cap L^\infty(\Omega)$
is stable under matrix multiplication and inversion.
For this, the following lemma will be crucial.

\begin{lemma-A}[{\cite[Cor.~2.7]{ak}}]\label{lem:SobolevChainRule}
  Let $\Omega\subset\R^n$ be open and let $\gamma : \C\to\C$ be a Lipschitz continuous map.
  In case of $\Lebesgue(\Omega) = \infty$, assume additionally that $\gamma(0) = 0$.
  If $f\in H^1(\Omega)$, then $\gamma\circ f\in H^1(\Omega)$.
\end{lemma-A}

\begin{lemma-A}\label{lem:SobolevAlgebra}
Let $\Omega\subset\R^n$ be open and let $f,g\in H^1(\Omega)\cap L^\infty(\Omega)$.
Then:
\begin{enumerate}
  \item[{\rm (a)}] $f \cdot g \in H^1(\Omega)\cap L^\infty(\Omega)$.

  \item[{\rm (b)}] If $\essinf |f| > 0$, then also $1/f\in H^1(\Omega)\cap L^\infty(\Omega)$.
\end{enumerate}
\end{lemma-A}

\begin{proof}
(a) Clearly, $fg\in L^2(\Omega)\cap L^\infty(\Omega)$.
Further, \cite[Section~4.25]{AltLinearFunctionalAnalysis} shows that the weak derivatives
of $fg$ exist and satisfy
\[
  \partial_j(fg)
  = (\partial_j f) \cdot g + f\cdot(\partial_j g)
  \quad \text{for } j \in \{ 1,\dots,n \}.
\]
As $\partial_jf,\partial_jg\in L^2(\Omega)$ and $f,g\in L^\infty(\Omega)$
it follows that $\partial_j(fg)\in L^2(\Omega)$.

\smallskip{}

(b) Let $r := \essinf |f| > 0$.
We trivially have $1/f \in L^\infty (\Omega)$.
Note that $\Lebesgue(\Omega) < \infty$ since $f \in H^1(\Omega) \subset L^2(\Omega)$
and $|f(x)| \geq r > 0$ almost everywhere.
Let $B := \{ z \in \C \colon |z| < r \}$ and define
$\gamma_0 : \C \setminus B \to \C, z \mapsto z^{-1}$.
Then $\gamma_0$ is well-defined and Lipschitz continuous,
since $|z^{-1} - w^{-1}| = \big| \frac{w - z}{z w} \big| \leq r^{-2} |w - z|$
for $z,w \in \C \setminus B$.
Now, \cite[Theorem~1 in Section~3.1.1]{EvansGariepy} shows that there exists a Lipschitz continuous
extension $\gamma : \C \to \C$ of $\gamma_0$.
Since $|f(x)| \geq r$ almost everywhere, we have $\gamma \circ f = \gamma_0 \circ f = 1/f$
almost everywhere, and Lemma~\ref{lem:SobolevChainRule} shows
$1/f = \gamma \circ f \in H^1(\Omega)$.
%Now, define ${\gamma : \C \to \C}$ by
%\[
%  \gamma(z)
%  := \begin{cases}
%       1/z       & \text{for $|z|\ge r$,} \\
%       \overline{\vphantom{t} z}/r^2 & \text{for $|z|<r$}.
%     \end{cases}
%\]
%Then, $\gamma(0) = 0$ and $\gamma \circ f = 1/f$ almost everywhere.
%Thus, if we can show that $\gamma$ is Lipschitz continuous, then Lemma~\ref{lem:SobolevChainRule}
%implies that $1 / f \in H^1(\Omega)$.
%To see that $\gamma$ is indeed Lipschitz continuous, we distinguish four cases for $z,w \in \C$.
%If $|z|, |w| \geq r$, then $|z^{-1} - w^{-1}| = |\tfrac{w - z}{z w}| \leq r^{-2} \cdot |w - z|$.
%On the other hand, if $|z|,|w| \leq r$, then
%$| \overline{\vphantom{t} z} / r^2 - \overline{\vphantom{t} w} / r^2 | = r^{-2} \cdot |w - z|$.
%Next, if $|z| < r$ and $|w| \geq r$, then by the intermediate value theorem there is some
%$t \in [0,1]$ such that $u := t z + (1 - t) w$ satisfies $|u| = r$.
%Note that $u^{-1} = \overline{\vphantom{t} u} / |u|^2 = r^{-2} \, \overline{\vphantom{t} u}$.
%Using the estimates from the preceding two cases, we thus see
%\begin{align*}
%  |\gamma(z) - \gamma(w)|
%  & \leq |\gamma(z) - \gamma(u)| + |\gamma(u) - \gamma(w)|
%    \leq r^{-2} \cdot \big( |z - u| + |u - w| \big) \\
%  & =    r^{-2} \cdot \big( \, |(1-t)z - (1-t)w| + |t z - t w| \, \big)
%    =    r^{-2} \cdot |z - w| \, .
%\end{align*}
%Finally, if $|z| \geq r$ and $|w| < r$, the same arguments as in the preceding case apply.
%%Then $\gamma$ is Lipschitz continuous, $\gamma(0) = 0$, and $\gamma\circ f = 1/f$.
%%The claim thus follows from Lemma \ref{lem:SobolevChainRule}.
\end{proof}

In the following we denote by $H^1(\Omega; \C^{k{\times}\ell})$
the space of all matrix-valued functions ${A : \Omega\to\C^{k{\times}\ell}}$
for which each component function is in $H^1(\Omega)$.
We similarly define $L^p(\Omega;\C^{k{\times}\ell})$ for $p \in [1, \infty]$.

\begin{lemma-A}\label{lem:SobolevMatrices}
  Let $\Omega\subset\R^n$ be open and let
  $A\in H^1(\Omega;\C^{k{\times}\ell})\cap L^\infty(\Omega;\C^{k{\times}\ell})$
  and $B\in H^1(\Omega;\C^{\ell{\times} m})\cap L^\infty(\Omega;\C^{\ell{\times} m})$.
  Then the following statements hold:
  \begin{enumerate}
    \item[{\rm (a)}] $AB\in H^1(\Omega;\C^{k{\times} m})\cap L^\infty(\Omega;\C^{k{\times} m})$.
                     \vspace*{0.1cm}

    \item[{\rm (b)}] If $k=\ell$ and ${\displaystyle \essinf_{x\in\Omega}}\sigma_0(A(x)) > 0$, then
                     $A^{-1}\in H^1(\Omega;\C^{k{\times} k})\cap L^\infty(\Omega;\C^{k{\times} k})$.

    \item[{\rm (c)}] If $\displaystyle \essinf_{x\in\Omega}\sigma_0(B(x)) > 0$, then
                     $B^\dagger\in H^1(\Omega;\C^{m{\times}\ell})\cap L^\infty(\Omega;\C^{m{\times}\ell})$.
  \end{enumerate}
\end{lemma-A}

\begin{proof}
Statement (a) follows from Lemma \ref{lem:SobolevAlgebra} (a),
since $(A B)_{j, n} = \sum_t A_{j,t} B_{t, n}$.
For (b) we first observe that Leibniz's formula
\[
  \det A(x)
  = \sum_{\sigma \in S_k}
      \Big[
       \mathrm{sign}(\sigma)
       \prod_{j=1}^k A_{\sigma(j), j}(x)
      \Big]
\]
and Lemma \ref{lem:SobolevAlgebra} (a) yield $\det A\in H^1(\Omega)\cap L^\infty(\Omega)$.
Now, the condition on $A$ implies that $A(x)$ is invertible
for a.e.\ $x\in\Omega$ so that $A(x)^{-1}$ indeed exists for a.e.\ $x\in\Omega$.
Moreover, for a.e.\ $x\in\Omega$, for the smallest eigenvalue $\la(x)$ of $A(x)^*A(x)$
we have that $\la(x)\ge c:= \essinf (\sigma_0(A))^2 > 0$.
Therefore, we conclude that
\[
  |\det A(x)|^2
  = \det(A(x)^*A(x))
  \ge \la(x)^k
  \ge c^k
\]
for a.e.\ $x\in\Omega$.
Hence, Lemma~\ref{lem:SobolevAlgebra} (b) shows that
$(\det A)^{-1}\in H^1(\Omega) \cap L^\infty (\Omega)$.
Also,
\[
  \|A(x)^{-1}\|^2
  = \|[A(x)^*A(x)]^{-1}\|
  \le \frac{1}{c}
\]
for a.e.\ $x\in\Omega$ implies that $A^{-1}\in L^\infty(\Omega;\C^{k{\times} k})$.
Finally, $A^{-1}\in H^1(\Omega;\C^{k{\times} k})$ follows from Lemma~\ref{lem:SobolevAlgebra} (a),
combined with the so-called \emph{cofactor formula} for the inverse of a matrix
(see for instance \cite[Equations (5-22) and (5-23)]{HoffmanKunzeLinearAlgebra}).
It states for $A \in \C^{k {\times} k}$ with $k > 1$ and $\det A \neq 0$ that
\[
  A^{-1} = \frac{\mathrm{adj} A}{\det A}
  \qquad \text{with} \quad
  (\mathrm{adj} A)_{i,j} = (-1)^{i+j} \cdot \det A^{(j,i)} \, ,
\]
where $A^{(j,i)}$ is obtained from $A$ by deleting its $j$-th row
and its $i$-th column.
In the remaining case $k = 1$, we have
$A^{-1} = (\det A)^{-1} \in H^1 (\Omega) \cap L^\infty (\Omega)$ as well.

The condition on $B$ in (c) implies that $B(x)^*B(x)$ is invertible for a.e.\ $x\in\Omega$
with $\essinf_{x\in\Omega}\sigma_0(B(x)^*B(x)) = \essinf_{x\in\Omega}\sigma_0(B(x))^2 > 0$.
The claim now follows from (a), (b), and the identity $B(x)^\dagger = (B(x)^*B(x))^{-1}B(x)^*$
(see Part~(iv) of Lemma~\ref{lem:TechnicalPseudoInverseIdentities}).
\end{proof}

\subsubsection{A certain property of the space \texorpdfstring{$\HH^2 (\R)$}{H²(ℝ)}}

\begin{lemma-A}\label{lem:H2FunctionsNonSymmetricFourierOperators}
%%%  Let $g \in \HH^2 (\R)$ be arbitrary. Then $Xg' \in L^2 (\R)$. More precisely, we have
%%% If $g \in \HH^2 (\R)$, then $Xg' \in L^2 (\R)$ and moreover we have the estimate
If $g \in \HH^2 (\R)$, then $Xg' \in L^2 (\R)$ with the estimate
  \begin{align*}
    \|X \, g'\|_{L^2}
    & \leq 45 \cdot \big( \|g''\|_{L^2}^2 + \|X^2 \, g\|_{L^2}^2 + \|g'\|_{L^2}^2 \big)^{1/2} \\
    & \leq 45 \cdot \big(
                    (1 + 4 \pi^2) \, \|g''\|_{L^2}^2
                    + \|X^2 \, g\|_{L^2}^2
                    + 4 \pi^2 \, \|g\|_{L^2}^2
                  \big)^{1/2} \, .
    %\leq \|g'\|_{L^2} + ??? \cdot (\|g''\|_{L^2} + \|x^2 \cdot g\|_{L^2}) \, .
  \end{align*}
\end{lemma-A}

\begin{proof}
  It follows from \cite[Lemma~5.4]{AdamsSobolevSpaces} that for any $\eta > 0$
  and $f \in C^2 ([0,\eta])$,
  \[
    |f'(0)|^2
    \leq \frac{C}{\eta}
         \cdot \left(
                 \eta^2 \cdot \int_0^\eta |f''(t)|^2 \, dt
                 + \eta^{-2} \cdot \int_0^\eta |f(t)|^2 \, dt
               \right) \, ,
    %\label{eq:DerivativeInterpolationEstimate}
  \]
  where $C := 2 \cdot 9^2$.
  One can see that this remains true for $f \in H^2 \big( (0,\eta) \big)$,
  by a density argument since $H^2 \big( (0,\eta) \big) \hookrightarrow C^1 ( [0,\eta])$
  (see for instance \cite[Thm.\ 4.12, Part II]{AdamsSobolevSpaces}).

  Given $g \in \HH^2 (\R)$ and $x \in [1,\infty)$,
  we can apply the above estimate to the function $t \mapsto g(x+t)$ to obtain
  \[
    |g'(x)|^2
    \leq C \cdot \fint_0^\eta
                   \big(\eta^2 \cdot |g''(x+t)|^2
                   + \eta^{-2} \cdot |g(x+t)|^2\big)
                 \, d t \, ,
  \]
  where we denote by $\fint_\Omega f(x) \, dx = \frac{1}{\Lebesgue(\Omega)} \int_\Omega f(x) \, dx$
  the \emph{average} of $f$ over $\Omega$, with $\Lebesgue(\Omega)$ denoting the Lebesgue measure
  of $\Omega$.

  Now, fix $n \in \N_0$ for the moment, and let $x \in [2^n, 2^{n+1})$.
  If we set $\eta = 2^{-n}$, then $\frac{2^n}{\eta} = 2^{2n} \leq x^2 \leq (x+t)^2$
  for all $t \in [0,\eta]$.
  Therefore,
  \begin{equation}
    |x \cdot g'(x)|^2
    \leq 4 \cdot 2^{2n} \cdot |g'(x)|^2
    \leq 4C \fint_0^{2^{-n}} \!\!\!\!
              \big(
                |g''(x+t)|^2
                + |(x+t)^2 \cdot g(x+t)|^2
              \big)
            \, d t \, .
    \label{eq:HigherDerivativeInterpolation}
  \end{equation}
  For brevity, set $F(y) := |g''(y)|^2 + |y^2 \cdot g(y)|^2$.
  Then, for any $t \in [0, 2^{-n}] \subset [0, 1]$, we have $2^{n+1} + t \leq 2^{n+2}$ and hence
  $\int_{2^n}^{2^{n+1}} F(x+t) \, d x \leq \int_{2^n}^{2^{n+2}} F(y) \, d y$.
  By combining this observation with the trivial estimate
  $\fint_\Omega G(t) \, d t \leq \|G\|_{L^\infty (\Omega)}$, and by integrating
  Equation~\eqref{eq:HigherDerivativeInterpolation} over $x \in [2^n, 2^{n+1})$, we arrive at
  \[
    \int_{2^n}^{2^{n+1}}
      |x \cdot g'(x)|^2
    \, d x
    \leq 4C \cdot \fint_0^{2^{-n}}
                    \int_{2^n}^{2^{n+1}} F(x+t) \, d x
                  \, d t
    \leq 4C \cdot \int_{2^n}^{2^{n+2}} F(y) \, d y \, .
  \]
  Summing over $n \in \N_0$, we conclude that
  \begin{equation}
  \begin{split}
    \int_1^\infty |x \cdot g'(x)|^2 \, d x
    & = \sum_{n=0}^\infty
          \int_{2^n}^{2^{n+1}}
            |x \cdot g'(x)|^2
          \, d x \\
    & \leq 4C \cdot \int_1^\infty
                      F(y)
                      \cdot \sum_{n=0}^\infty
                              \indicator_{(2^n, 2^{n+2})} (y)
                    \, d y \\
    & \leq 12C \cdot \big(
                       \|g''\|_{L^2 ( (1,\infty) )}^2
                       + \|X^2 \, g\|_{L^2 ( (1,\infty) )}^2
                     \big) \, .
  \end{split}
  \label{eq:NonSymmetricOperatorPositiveAxis}
  \end{equation}
  Here the last step used that $\sum_{n=0}^\infty \indicator_{(2^n, 2^{n+2})} (y) \leq 3$;
  indeed, if ${2^n < y < 2^{n+2}}$, then each $k \in \Z$ for which also $2^k < y < 2^{k+2}$
  satisfies $2^n < 2^{k+2}$ and $2^k < 2^{n+2}$, so that $k \in \{ n-1, n, n+1 \}$.

  By applying estimate~\eqref{eq:NonSymmetricOperatorPositiveAxis} to
  $h : \R \to \C, x \mapsto g(-x)$ instead of $g$, we easily get
  \(
    \int_{-\infty}^{-1} |x \cdot g'(x)|^2 \, d x
    \leq 12 C \cdot \big(
                      \|g''\|_{L^2 ( (-\infty,-1) )}^2
                      + \|X^2 \, g\|_{L^2 ( (-\infty,-1) )}^2
                    \big)
  \).
  Adding this to \eqref{eq:NonSymmetricOperatorPositiveAxis} and using the trivial
  estimate $\int_{-1}^1 |x \cdot g'(x)|^2 \, dx \leq \|g'\|_{L^2}^2$, we finally arrive at
  \[
    \int_{\R} |x \cdot g'(x)|^2 \, d x
    \leq \|g'\|_{L^2}^2 + 12C \cdot \big( \|g''\|_{L^2}^2 + \|X^2 \, g\|_{L^2}^2 \big) \, .
  \]
  This easily implies the first part of the stated estimate.

  For the last part, recall that $\Fourier [g'] (\xi) = 2\pi i \xi \, \widehat{g}(\xi)$
  and $\Fourier [g''](\xi) \!=\! (2\pi i \xi)^2 \, \widehat{g} (\xi)$.
  Thanks to Plancherel's theorem and the elementary estimate
  ${|\xi|^2 \leq 1 + |\xi|^4}$, we thus see
  \begin{align*}
    \|g'\|_{L^2}^2
    & = \int_{\R}
          |2 \pi \xi \cdot \widehat{g}(\xi)|^2
        \, d \xi
      \leq (2 \pi)^2 \cdot
           \int_{\R}
             |\widehat{g} (\xi)|^2
             + |(2\pi i \xi)^2 \, \widehat{g} (\xi)|^2
           \, d \xi \\
    & =    (2 \pi)^2 \cdot \big( \|g\|_{L^2}^2 + \|g''\|_{L^2}^2 \big) \, .
  \end{align*}
  Together with the first part of the lemma, this implies the second part.
\end{proof}

\subsubsection{Sobolev functions on slices and the AC-property}
%It is well-known (see for instance \cite[Theorem 2.1.4]{ZiemerWeaklyDifferentiable})
%that Sobolev functions restrict to absolutely continuous functions on lines parallel to the
%coordinate axes.
%Since this fact is central to our proof, and since the proof in \cite{ZiemerWeaklyDifferentiable}
%is quite condensed, we provide a detailed argument for completeness.
Let $A \subset \R^n$ be Borel measurable, where $n > 1$.
For $i \in \{1, \ldots, n\}$ and $x \in \R^{n-1}$ we define the following
Borel measurable subset of $\R$:
\[
  A_{i,x}
  = \{t\in\R : (x_1,\ldots,x_{i-1},t,x_{i},\ldots,x_{n-1})\in A\} \, .
\]
Note that $A_{i,x}$ is open if $A$ is so.
The following lemma is an easy consequence of Fubini's theorem.

\begin{lemma-A}\label{l:andrei}
  A Borel set $N \subset \R^n$ has measure zero if and only if for some (and then all)
  $i \in \{1,\ldots,n\}$ and a.e.\ $x \in \R^{n-1}$ the set $N_{i,x}$ has measure zero in $\R$.
\end{lemma-A}

We say that a function $h : U \to \C$, where $U \subset \R$ is open,
is \emph{locally absolutely continuous (LAC)} on $U$
if it is LAC on each connected component of $U$;
this is equivalent to $h$ being LAC on each open subinterval of $U$.
Here, a function $f : I \to \C$ with an open interval $I \subset \R$ is called
locally absolutely continuous if there is a function $g \in L^1_{\mathrm{loc}} (I)$ such that
$f(x) - f(y) = \int_y^x g(t)\,dt$ for all $x,y \in I$.
In particular, each LAC function is continuous.

\begin{definition-A}
  Let $\Omega\subset\R^n$ be an open set.
  A (pointwise defined) function ${f : \Omega\to\C}$ is said to have the {\em AC-property}
  \braces{on $\Omega$},
  if for each $i\in\{1,\ldots,n\}$ and almost all $x\in\R^{n-1}$ the function
  \[
    f_{i,x} : \Omega_{i,x}\to\C,\quad t\mapsto f(x_1,\ldots,x_{i-1},t,x_{i},\ldots,x_{n-1})
  \]
  is LAC on $\Omega_{i,x}$.
\end{definition-A}

Note that the classical partial derivatives $\partial_i f$ of a function $f : \Omega\to\C$
having the AC-property exist a.e.~on $\Omega$ by \cite[Theorem 3.35]{FollandRA} and
Lemma~\ref{l:andrei}.

\begin{lemma-A}\label{l:slices}
  Let $\Omega\subset\R^n$ be open and let $f\in W^{1,1}_{\rm loc}(\Omega)$.
  Then there is a representative $g : \Omega\to\C$ of $f$ which has the AC-property on $\Omega$.
  In particular, we have $\partial_i g = D_i f$ a.e.\ on $\Omega$, $i=1,\ldots,n$.
  Here, $D_i f$ denotes the weak derivative of $f$.
\end{lemma-A}

\begin{proof}
Let $\Omega^{(0)} := \emptyset$ and
$\Omega^{(k)} := \{x \in \Omega : \dist(x,\Omega^c) > 1/k\} \cap (-k,k)^n$, $k\in\N$.
Then each $\Omega^{(k)}$ is open in $\R^n$, $\ol{\Omega^{(k)}} \subset \Omega$ is compact,
$\Omega^{(k)} \subset \Omega^{(k+1)}$, and $\bigcup_k \Omega^{(k)} = \Omega$.
By \cite[Thm.~1.41]{mz}, for each $k \in \N$ there exists a representative $f^{(k)}$ of $f$
which has the AC-property on $\Omega^{(k)}$.
It is clear that the function $g : \Omega\to\C$,
\[
  g
  := \sum_{k=1}^\infty
       \one_{\Omega^{(k)} \backslash \Omega^{(k-1)}} \cdot f^{(k)}
\]
is a representative of $f$.
Let us show that it has the AC-property on $\Omega$.

First of all, for each $k \in \N$ and $i \in \{1, \ldots, n\}$
there exists a set $L^{(k)}_i\subset\R^{n-1}$ of measure zero such that
$f^{(k)}_{i,x}$ is LAC on $\Omega^{(k)}_{i,x}$ for every $x \in \R^{n-1} \backslash L^{(k)}_i$.
Let $L := \bigcup_{i,k}L_i^{(k)}$.

Fix $k\in\N$.
Then $f^{(k+1)} = f^{(k)}$ a.e.\ on $\Omega^{(k)}$.
In particular, by Lemma~\ref{l:andrei}, for each $i\in\{1,\ldots,n\}$ there exists a set
$M_i^{(k)}\subset\R^{n-1}$ of measure zero such that for all $x\in \R^{n-1} \backslash M_i^{(k)}$
we have that $f^{(k+1)}_{i,x} = f^{(k)}_{i,x}$ a.e.\ on $\Omega^{(k)}_{i,x}$.
Let $M := \bigcup_{i,k}M_i^{(k)}$.

Let $N := L\cup M\subset\R^{n-1}$.
Then $N$ is a null-set, and for each $i\in\{1,\ldots,n\}$, each $k\in\N$,
and each $x \in \R^{n-1} \backslash N$ we have that $f_{i,x}^{(k)}$ is LAC on $\Omega_{i,x}^{(k)}$
and $f^{(k+1)}_{i,x} = f^{(k)}_{i,x}$ on $\Omega^{(k)}_{i,x}$;
indeed, since $f_{i,x}^{(k+1)} = f_{i,x}^{(k)}$ almost everywhere on $\Omega_{i,x}^{(k)}$,
and since both functions are continuous on the open set $\Omega_{i,x}^{(k)}$,
they agree everywhere on $\Omega_{i,x}^{(k)}$.
Now, if $i\in\{1,\ldots,n\}$, $x \in \R^{n-1} \backslash N$,
and if $K \subset \Omega_{i,x}$ is compact,
then $\bigcup_k \Omega_{i,x}^{(k)}$ is an open cover of $K$.
Thus, there is some $k = k(i, x, K) \in \N$ such that
$K \subset \Omega_{i,x}^{(k)}$ and $g_{i,x} = f_{i,x}^{(k)}$ on $K$.
Therefore, $g_{i,x}$ is LAC on $\Omega_{i,x}$.

\medskip{}

For the ``in particular''-part, it suffices to prove that $\partial_i g = D_i f$ almost everywhere
on every open rectangular cell $R = \prod_{j=1}^n (a_j,b_j)$ satisfying $\overline{R} \subset \Omega$.
To see this, set ${R_i := \prod_{j\neq i} (a_j,b_j) \subset \R^{n-1}}$,
and observe that for any $\varphi \in C_c^\infty(R)$ we have
\begin{align*}
  & \int_R
      (\partial_i g) \cdot \varphi
    \,dy
    + \int_R
        g \cdot (\partial_i\varphi)
      \,dx \\
  & = \int_{R_i}
        \int_{a_i}^{b_i}
          g_{i,x}'(t) \varphi_{i,x}(t)
        \,dt
      \,dx
    + \int_{R_i}
        \int_{a_i}^{b_i}
          g_{i,x}(t) \varphi_{i,x}'(t)
        \,dt
      \,dx
  = 0 \, .
\end{align*}
Hence, $\int_R(\partial_i g - D_if)\varphi\,dx = \int_R(f - g)\partial_i\varphi\,dx = 0$
for every $\varphi\in C_c^\infty(R)$.
The claim thus follows from the fundamental lemma of the calculus of variations
(see for instance \cite[Section~4.22]{AltLinearFunctionalAnalysis}).
\end{proof}

We close with this subsection with a result that generalizes Lemma~\ref{l:andrei}
in the case $n=2$ to sections of $\R^2$ that are not necessarily parallel to the coordinate axes.

\begin{lemma-A}\label{lem:FubiniGeneralDirection}
  Let $N \subset \R^2$ be a null-set, and let $(a,b) \in \R^2 \backslash \{0\}$.
  Then there is a null-set $N_0 \subset \R^2$ such that
  for all $(x,\omega) \in \R^2 \backslash N_0$, we have
  \[
    (x + t a, \omega + t b) \in \R^2 \backslash N
    \quad \text{for almost all } t \in \R .
  \]
\end{lemma-A}

\begin{rem*}
  The set of $t \in \R$ for which $(x + t a, \omega + t b) \in \R^2 \backslash N$
  depends on $(x,\omega)$.
\end{rem*}

\begin{proof}
  Set $\theta := (a,b) \in \R^2 \backslash \{0\}$,
  and choose $\varrho \in \R^2 \backslash \{0\}$ with $\varrho \perp \theta$.
  Let us define $T : \R^2 \to \R^2, (t,s) \mapsto t \theta + s \varrho$.
  Note that $T$ is linear and bijective, so that the same holds of $T^{-1}$.
  In particular, $T$ and $T^{-1}$ are Lipschitz continuous, and thus map null-sets to null-sets.

  \smallskip{}

  Let $\widetilde{N} := T^{-1} N \subset \R^2$.
  By Lemma~\ref{l:andrei}, there is a null-set $\widetilde{N}_1 \subset \R$ such that for all
  $s \in \R \backslash \widetilde{N}_1$, the set
  $\widetilde{N}_{1,s} = \{ t \in \R \colon (t,s) \in \widetilde{N} \}$ is a null-set.
  Let $N_0 := T (\R {\times} \widetilde{N}_1)$, and note that $N_0 \subset \R^2$
  is indeed a null-set.

  We claim that if $(x,\omega) \in \R^2 \backslash N_0$,
  then $(x + t a, \omega + t b) \in \R^2 \backslash N$ for almost all $t \in \R$.
  To see this, let $(x,\omega) \in \R^2 \backslash N_0$.
  This implies ${(x,\omega) = T (t_0, s_0)}$ for certain
  $(t_0, s_0) \in \R {\times} (\R \backslash \widetilde{N}_1)$,
  so that $\widetilde{N}_{1, s_0}$ is a null-set.
  Finally, if ${t \in \R \backslash (\widetilde{N}_{1, s_0} - t_0)}$
  (which holds for almost all $t \in \R$),
  then $t + t_0 \notin \widetilde{N}_{1, s_0}$, which means
  that ${(t + t_0, s_0) \notin \widetilde{N} = T^{-1} N}$, and hence
  $(x + t a, \omega + t b) = T (t + t_0, s_0) \in \R^2 \backslash N$, as claimed.
\end{proof}

\subsection{Invariance properties of Gabor spaces}
\label{sub:GaborSpaceInvariance}

%The following lemma refines the statement of \cite[Proposition~A.1]{clmp}
%for the case of Gabor spaces in $L^2(\R)$.

\begin{lemma-A}\label{lem:GaborSpaceInvariance}
  Let $g \in L^2(\R)$ and let $\Lambda \subset \R^2$ be a lattice.
  Define $\calG := \calG(g,\Lambda)$
  Then $\pi(\lambda) \calG \subset \calG$ for all $\lambda \in \Lambda$.
  %Let $g \in L^2 (\R)$ and let $\Lambda \subset \R^2$ be a set containing $0$
  %which is closed under addition.
  %Setting
  %\[
  %  \Gamma
  %  := \big\{
  %       \mu \in \R^2
  %       \colon
  %       \pi(\mu) g \in \calG(g, \Lambda)
  %     \big\}
  %  \quad \text{and} \quad
  %  \Gamma_0
  %  := \big\{
  %       \mu \in \R^2
  %       \colon
  %       \pi(\mu) \calG (g, \Lambda) \subset \calG(g, \Lambda)
  %     \big\} ,
  %\]
  %where $\calG (g, \Lambda) = \ol{\linspan}\{\pi(\lambda) g \colon \lambda \in \Lambda\}$,
  %we have that $\Gamma = \Gamma_0 \supset \Lambda$ and that $\Gamma$ is closed under addition.
\end{lemma-A}

\begin{proof}
For $\lambda,\lambda' \in \Lambda$, there exists a unimodular constant $c = c(\lambda,\lambda') \in \C$
satisfying $\pi(\lambda) \pi(\lambda') = c \, \pi(\lambda + \lambda')$.
Hence, $\pi(\lambda) [\pi(\lambda') g] \in \calG$.
Since $\calG$ is spanned by the elements $\pi(\lambda') g$, $\lambda' \in \Lambda$,
this shows $\pi(\lambda) \subset \calG$ for all $\lambda \in \Lambda$.
\end{proof}

\subsection{Failure of the main result for general elements of \texorpdfstring{$\calG(g,\Lambda)$}{𝓖(g,Λ)}}

We close this paper with an example showing that the relation
\begin{equation}
  \dist \big( \pi(u,\eta) f, \mathcal{G}(g,\Lambda) \big)
  \asymp \dist ( (u,\eta), \Lambda) \cdot \|f\|_{L^2},
  \label{eq:GeneralizedQuantitativeBalianLow}
\end{equation}
which holds for $f = g$, does \emph{not} extend to general $f \in \mathcal{G}(g,\Lambda)$.
The example is constructed based on a footnote in \cite{cmp}.
%We remark that the general idea of the construction is based on a footnote in \cite{cmp}.

\begin{example-A}\label{exa:DistanceOnlyForGenerator}
  Let $\varphi : \R \to \R, x \mapsto e^{- \pi x^2}$ denote the Gaussian.
  We will repeatedly make use of the following two facts:
  First, \cite[Theorem 7.5.3]{g} shows that if $\alpha, \beta > 0$, then
  $(\varphi, \alpha \Z {\times} \beta \Z)$ is a frame for $L^2(\R)$ if and only if $\alpha \beta < 1$.
  By Ron-Shen duality (see \cite[Theorem 7.4.3]{g}), this implies that
  $(\varphi, \alpha \Z {\times} \beta \Z)$ is a Riesz sequence
  (a Riesz basis for its closed linear span) if and only if $\alpha \beta > 1$.

  Set $\Lambda := 2 \Z {\times} \tfrac{2}{3} \Z$ and
  $\Lambda_0 := \Lambda \cup ( (1,0) + \Lambda ) = \Z {\times} \tfrac{2}{3} \Z$.
  Then $(\varphi,\Lambda_0)$ is a frame for $L^2(\R)$
  %\cite[Theorem 7.5.3]{g} shows that $(\varphi,\Lambda_0)$ is a frame for $L^2(\R)$.
  %Furthermore, %it follows by Ron-Shen duality (see \cite[Theorem 7.4.3 and Corollary 7.5.1]{g}) that
  but \emph{not} a Riesz sequence.
  Thus, the synthesis operator
  \[
    T : \quad
    \ell^2(\Z^2) \to L^2(\R) , \quad
    (c_{k,\ell})_{k,\ell \in \Z} \mapsto \sum_{k,\ell \in \Z}
                                           c_{k,\ell} \,\, \pi(k, \tfrac{2}{3} \ell) \varphi
  \]
  is surjective, but \emph{not} injective, since otherwise the bounded inverse theorem would imply
  that $T$ is boundedly invertible, meaning that $(\varphi, \Lambda_0)$ is a Riesz basis for $L^2(\R)$.
  In other words, there exist $\ell^2$ sequences $c = ( c_{m,n} )_{m,n \in \Z}$
  and $d = ( d_{m,n} )_{m,n \in \Z}$
  with $(c,d) \neq 0$ and
\begin{align*}
    \sum_{m,n \in \Z}
      c_{m,n} \, \pi(2m, \tfrac{2}{3} n) \varphi
    &= \sum_{m,n \in \Z}
        d_{m,n} \, \pi(2m + 1, \tfrac{2}{3} n) \varphi \\
    &= \pi(1,0)
      \bigg[
        \sum_{m,n \in \Z}
          \widetilde{d}_{m,n} \, \pi(2m, \tfrac{2}{3} n) \varphi
      \bigg],
\end{align*}
  where $\widetilde{d}_{m,n} := e^{\frac{4}{3} \pi i n} d_{m,n}$ for $m,n \in \Z$.
  Then
  \(
    f := \sum_{m,n \in \Z}
           \widetilde{d}_{m,n} \, \pi(2m, \tfrac{2}{3} n) \varphi
  \)
  satisfies $f \in \mathcal{G}(\varphi, \Lambda)$ and $\pi(1,0) f \in \mathcal{G}(\varphi, \Lambda)$.
  Now, once we show that $f \neq 0$, we will have disproved \eqref{eq:GeneralizedQuantitativeBalianLow}.

  To see that $f \neq 0$, we note that $(\varphi, \Lambda)$ is a Riesz sequence.
  %as follows by Ron-Shen duality (\cite[Theorem 7.4.3]{g}) from \cite[Theorem 7.5.3]{g}.
  If $f = 0$, we would have $\widetilde{d} = 0$ and therefore $d = 0$.
  In turn, the above identity gives
  ${0 = \sum_{m,n \in \Z} c_{m,n} \pi(2m, \tfrac{2}{3}n) \varphi}$, whence $c = 0$,
  again since $(\varphi, \Lambda)$ is a Riesz sequence.
  Therefore, $f = 0$ implies $(c,d) = 0$ which is a contradiction.
\end{example-A}

\renewcommand{\thetheorem}{\arabic{theorem}}
%%% back to numerical numbering

\section*{Acknowledgments}
D.G.~Lee acknowledges support by the DFG Grants PF 450/6-1 and PF 450/9-1.
A.~Caragea acknowledges support by the DFG Grant PF 450/11-1.
The authors would like to thank G\"otz E.~Pfander and Peter Jung
for fruitful discussions.
The authors thank the editor and anonymous reviewers for their valuable comments
and suggestions which improved the paper.





\end{document}